\documentclass[12pt]{article}
\usepackage{amsthm,amsfonts,amssymb,amsmath}
\usepackage{fullpage}
\usepackage{enumerate}
\usepackage[colorlinks=true]{hyperref}
\usepackage{xcolor}
\usepackage{tikz-cd}

\newcommand{\BC}{{\mathbb {C}}}

\newcommand{\BQ}{{\mathbb {Q}}}
\newcommand{\BR}{{\mathbb {R}}}

\newcommand{\BZ}{{\mathbb {Z}}}

\newcommand{\CA}{{\mathcal {A}}}

\newcommand{\CD}{{\mathcal {D}}}

\newcommand{\CH}{{\mathcal {H}}}

\newcommand{\CL}{{\mathcal {L}}}
\newcommand{\CM}{{\mathcal {M}}}
\newcommand{\CN}{{\mathcal {N}}}
\newcommand{\CO}{{\mathcal {O}}}
\newcommand{\CP}{{\mathcal {P}}}

\newcommand{\CU}{{\mathcal {U}}}

\newcommand{\CX}{{\mathcal {X}}}
\newcommand{\CY}{{\mathcal {Y}}}

\newcommand{\an}{{\mathrm{an}}}

\newcommand{\Div}{{\mathrm{Div}}}

\renewcommand{\div}{{\mathrm{div}}}

\newcommand{\emb}{{\hookrightarrow}}

\renewcommand{\Im}{{\mathrm{Im}}}

\newcommand{\intb}{{\mathrm{int}}}

\newcommand{\NS}{{\mathrm{NS}}}

\newcommand{\ord}{{\mathrm{ord}}}

\newcommand{\Pic}{\mathrm{Pic}}
\newcommand{\Picc}{{\mathcal{P}\mathrm{ic}}}

\DeclareMathOperator{\Spec}{Spec}

\newcommand{\supp}{{\mathrm{supp}}}

\newcommand{\red}{{\mathrm{red}}}

\newcommand{\sk}{{\mathrm{sk}}}
\newcommand{\val}{{\mathrm{val}}}

\newcommand{\wt}{\widetilde}
\newcommand{\wh}{\widehat}

\newcommand{\pair}[1]{\langle {#1} \rangle}

\newcommand{\ol }{\overline}

\newcommand{\lra}{\longrightarrow}

\newcommand{\fh}{\mathfrak{h}}

\newcommand{\nef}{\mathrm{nef}}

\newcommand{\CDD}{\overline{\mathcal D}}

\newcommand{\OL}{{\overline{L}}}
\newcommand{\OM}{{\overline{M}}}
\newcommand{\OD}{{\overline{D}}}
\newcommand{\ON}{{\overline{N}}}
\newcommand{\OH}{{\overline{H}}}

\newcommand{\TL}{{\widetilde{L}}}

\renewcommand{\AA}{\mathbb{A}}
\newcommand{\CC}{\mathbb{C}}
\newcommand{\RR}{\mathbb{R}}
\newcommand{\ZZ}{\mathbb{Z}}
\newcommand{\QQ}{\mathbb{Q}}
\newcommand{\FF}{\mathbb{F}}
\newcommand{\PP}{\mathbb{P}}

\newcommand{\SX}{\mathsf{X}}
\newcommand{\SY}{\mathsf{Y}}

\newcommand{\kkk}{Let $k$ be either $\ZZ$ or a field. }

\newcommand\HPic{\widehat{\mathrm{Pic}}}

\newcommand\TD{\widetilde{D}}
\renewcommand\TH{\widetilde{H}}

\newcommand\intt{\mathrm{int}}
\newcommand\TPic{\widetilde{\mathrm{Pic}}}

\newcommand\m{\mathrm{mod}}
\newcommand\dist{\mathrm{dist}}

\newcommand\fin{\mathrm{fin}}
\newcommand\ft{\mathrm{ft}}
\renewcommand{\d}{\textnormal{d}}

\newcommand\ttrd{\mathrm{Top.tr.deg}}
\newcommand\trd{\mathrm{tr.deg}}
\newcommand\Frac{\mathrm{Frac}}

\newtheorem{thm}{Theorem}[subsection]
\newtheorem{cor}[thm]{Corollary}
\newtheorem{lem}[thm]{Lemma}
\newtheorem{prop}[thm]{Proposition}
\newtheorem{conj}[thm]{Conjecture}
\newtheorem{defn}[thm]{Definition}

\newtheorem{theorem}[thm]{Theorem}
\newtheorem{proposition}[thm]{Proposition}
\newtheorem{lemma}[thm]{Lemma}
\newtheorem{corollary}[thm]{Corollary}

\theoremstyle{definition}
\newtheorem{definition}[thm]{Definition}
\newtheorem{example}[thm]{Example}

\theoremstyle{remark}
\newtheorem{remark}[thm]{Remark}

\begin{document}

\title{Equidistribution of small points over finitely generated fields}
\author{Ruoyi Guo, Lai Shang, Chengyuan Yang, Xinyi Yuan}
\maketitle

\setcounter{tocdepth}{4}
\tableofcontents

\section{Introduction}

In \cite{SUZ}, Szpiro--Ullmo--Zhang proved a novel equidistribution theorem of small points over number fields, which played a fundamental role in the proof of the Bogomolov conjecture by Ullmo \cite{Ull} and Zhang \cite{Zh3}. 
The equidistribution theorem was later generalized to a relative setting by Moriwaki \cite{Mor3}, to non-archimedean places by 
Chambert-Loir \cite{CL}, and to semipositive metrics by Yuan \cite{Yua08}.
The equidistribution theorem has become a fundamental tool in Diophantine geometry and arithmetic dynamics. 

In \cite{YZ26}, Yuan--Zhang introduced a theory of adelic line bundles on quasi-projective varieties, which extended Zhang's original theory over projective varieties in \cite{Zha2}. 
In \cite{YZ26}, the equidistribution theorem was further extended to quasi-projective varieties over number fields. 
In the case of abelian schemes, the theorem was proved independently by K\"uhne \cite{Kuh} and applied to refine the uniform Mordell conjecture proved by Dimitrov--Gao--Habegger \cite{DGH}. 

Moreover, Yuan--Zhang \cite{YZ26} proposed an equidistribution conjecture, which extends their equidistribution theorem from number fields to finitely generated fields.
The importance of this conjecture is justified by the Lefschetz principle, which asserts that any quasi-projective variety over any field can be descended to a finitely generated field. This lies in the mechanism of \cite{YZ26} to apply arithmetic methods to treat results on quasi-projective varieties over arbitrary fields. 
The goal of this paper is to prove the equidistribution conjecture at fully transcendental valuations, and present a counterexample to an earlier version of the conjecture. 

\subsection{An equidistribution conjecture}
\label{subsec equi conj}

Let us recall an early version of the equidistribution conjecture in \cite[Conj. 5.4.1]{YZ26}. We will follow the terminology of adelic line bundles in Yuan--Zhang \cite{YZ26}.

\kkk
Let $F$ be a finitely generated field over $k$, i.e. a finitely generated extension of the fraction field $\Frac(k)$ of $k$. 
Let $X$ be a geometrically integral quasi-projective variety of dimension $n$ over $F$. 
Let $\overline L\in\HPic(X/k)_\nef$ be a nef adelic line bundle on $X/k$.

For any integrable adelic line bundle $\OH\in \wh\Pic(F/k)_{\intb}$, the \emph{Moriwaki height function}
$$
h_{\overline L}^{\overline H}:  X(\overline F)\lra
\RR
$$
is defined as follows. For an element $x$ in $X(\overline F)$ with image $x'\in X$, set $\deg(x')=\deg(x)=[F(x):F]$.
Let $\OL|_{x'}$  and $\OH|_{x'}$ be pullbacks of $\OL$ and $\OH$ along the morphisms $x'\emb X$ and $x'\to \Spec(F)$ respectively. Then $\OL|_{x'}$ and $\OH|_{x'}$ are adelic line bundles in $\HPic(x'/k)$. Define the \emph{Moriwaki height} by
$$
h_{\overline L}^{\overline H}(x)=\frac{1}{\deg(x)}\OL|_{x'}\cdot \OH|_{x'}^{d-1}.
$$
Here if $k=\ZZ$, we denote $d=\trd(F/\QQ)+1$; if $k$ is a field, we denote $d=\trd(F/k)$.
Here $\trd$ denotes the transcendence degree. 

Denote by $\wt L$ the image of $\OL$ under the canonical map
$\wh\Pic(X/k)_{\nef} \to \wh\Pic(X/F)_{\nef}$
introduced in \cite[\S 2.5.5]{YZ26}.
Assume that the self-intersection number 
$$
\deg_{\wt L}(X/F)=\wt L^{\dim X}>0.
$$
Then we have a well-defined \emph{Moriwaki height} of $X$ as
$$h_{\overline L}^{\overline H}(X)
=\frac{ \overline L^{n+1} \cdot \overline H^{d-1}}
{(n+1)\deg_{\wt L}(X/F) }.$$

A sequence $\{x_m\}_{m\geq 1}$ in $X(\overline F)$
is said to be \emph{generic} if any closed subvariety $Y\subsetneqq X$ contains only finitely many terms of the sequence.

Let $\overline H\in \HPic (F/k)_{\intb}$ be any integrable adelic line bundle.
A sequence $\{x_m\}_{m\geq 1}$ in $X(\overline F)$ is said to be \emph{$h_{\overline L}^{\OH}$-small} if $h_{\overline L}^{\OH} (x_m)$ converges to $h_{\overline L}^{\OH} (X)$.

A sequence $\{x_m\}_{m\geq 1}$ in $X(\overline F)$ is said to be {\em numerically small} with respect to $\OL$ if it is {$h_{\overline L}^{\OH}$-small} for every $\overline H\in \HPic (F/k)_{\intb}.$
Note that by linear combination, it suffices to require this for all nef adelic line bundles
 $\overline H\in \HPic (F/k)_{\nef}.$

If $d=1$, then $F$ is either a number field (for $k=\ZZ$) or a function field of one variable (for $k$ being a field). In this case, $h_{\overline  L}^{\overline  H}$ is just the usual height function $h_{\overline  L}$. Then the numerical smallness is equivalent to the usual one given by $h_{\overline  L}(x_m)\to h_{\overline  L}(X)$.

Let $v$ be a non-trivial valuation of $F$, and denote by $F_v$ the completion.
Let $X_v^\an$ be the analytic space associated to $X_{F_v}$. 
Namely, if $F_v\simeq \CC$, then  $X_v^\an$ is just the complex analytic space $X(F_v)$; if $F_v\simeq \RR$, then  $X_v^\an$ is just the quotient of the complex analytic space $X(\ol F_v)$ by the complex conjugation;
if $v$ is non-archimedean,  $X_v^\an$ is the Berkovich analytic space associated to $X_{F_v}$ from \cite{Ber1}. 
  
From \cite[\S3.6.8]{YZ26}, there is an \emph{equilibrium measure}
$$
\d\mu_{\overline L,v}:=\frac{1}{\deg_{\wt L}(X/F)}c_1(\OL)_v^n
$$
over the analytic space $X_v^\an$.
If $v$ is archimedean, this is essentially the classical {Monge--Amp\`ere} measure.
If $v$ is non-archimedean, this is a generalized Chambert-Loir measure.

For each point $x\in X(\overline F)$, we have the measure
$$\mu_{x,v}:=\frac{1}{\deg(x)}\delta_{x'_v}$$ 
on $X_v^\an$.
Namely, $x'$ is the closed point of $X$ corresponding to $x$ as above,
and $x'_v=x'\times_F F_v$ is a 0-dimensional closed subscheme of $X_{F_v}$. 
If $x'_v$ is reduced, then it is equivalent to {a finite set of closed points} of $X_{F_v}$, and thus a finite subset of classical points of $X_v^\an$; in this case, 
$\delta_{x'_v}$ is the Dirac measure of $x'_v$ in $X_v^\an$. 
If $x'_v$ is not reduced, which only happens if $k$ is a field of positive characteristic, then we take 
$$
\delta_{x'_v}=\sum_{z\in x'_v} \mathrm{mult}_z(x'_v) \delta_z,
$$
where the multiplicity $\mathrm{mult}_z(x'_v)$ is the length of the local ring $\CO_{x'_v,z}$, and we view $z$ as a classical point of $X_v^\an$ in $\delta_z$.

Let $\d\rho_v$ be a probability measure on $X_v^\an$, and our equidistribution theorem concerns the case
$\d\rho_v=\d\mu_{\overline L,v}$. 
We say that the Galois orbit of a sequence $\{x_m\}_{m\geq 1}$ of points of $X(\overline F)$ is \emph{equidistributed in $X_v^\an$ for} $\d \rho_v$ if the weak convergence
$$
\mu_{x_m,v}\lra \d\rho_v
$$
holds on $X_v^\an$.
Namely,
$$\int_{X_v^\an} f\, \mu_{x_m,v}\lra \int_{X_v^\an} f\, \d\rho_v$$
for any $f\in C_c(X_v^\an)$. Here $C_c(X_v^\an)$ is the space of real-valued continuous and compactly supported functions on $X_v^\an$.

Denote by $\CM(F/k)$ the set of valuations $v$ on $F$ such that $v|_k$ is trivial if $k$ is a field.
Finally, we are ready to state an early version of the equidistribution conjecture in Yuan--Zhang \cite{YZ26}.

\begin{conj}[equidistribution conjecture] \label{equi conj}
	\kkk
	Let $F$ be a finitely generated field over $k$.
	Let $X$ be a geometrically integral quasi-projective variety over $F$.
	Let $\overline L$ be a nef adelic line bundle on $X/k$ such that $\deg_{\wt L}(X/F)>0$.
	Let $\{x_m\}_m$ be a generic sequence  in $X(\overline F)$ which is numerically small with respect to $\OL$.
	Let $v$ be a non-trivial valuation of $F$. 
	If $k$ is a field, assume that the restriction $v|_k$ is trivial.  
	Then the Galois orbit of $\{x_m\}_m$ is equidistributed in $X_v^\an$ for $\d\mu_{\overline L,v}$.
\end{conj}

In the arithmetic case ($k=\ZZ$), if $F$ is a number field and $X$ is projective, the conjecture was already known.
In fact, the pioneering work of Szpiro--Ullmo--Zhang \cite{SUZ} proved the equidistribution for number fields $F$ and archimedean places $v$ assuming pointwise positivity of the Chern form $c_1( L, \|\cdot\|_v)$.
Their work was extended to non-archimedean places $v$ by Chambert-Loir \cite{CL}.
The full case of number fields with $L$ ample was proved by Yuan \cite{Yua08} by developing a bigness theorem for difference of ample hermitian line bundles.
The proof of \cite{Yua08} actually works by replacing the ampleness of $L$ by the positivity $\deg_{L}(X)>0$. For more history of this subject, we refer to \cite[\S6.3]{Yua12}.

The above arguments were also generalized to the geometric case.
In that case, if $X$ is projective over $F$, and the valuation $v$ of $F$ comes from a prime divisor of a projective model of $F$ over $k$, the conjecture was proved independently by Faber \cite{Fab} when the {transcendence degree} of $F/k$ is 1 and by Gubler \cite{Gub08} for general {transcendence degrees}.
 
In \cite[Thm. 5.4.3]{YZ26}, the conjecture is proved for any quasi-projective $X$ and for any number field $F$ or function field $F$ of one variable.

\subsection{A counterexample}

The equidistribution conjecture in Conjecture \ref{equi conj} appeared as 
Conjecture 5.4.1 in the {arXiv} versions (arXiv:2105.13587v9 and earlier ones) of \cite{YZ26}. 
However, in preparation of our current paper, we came up with a counterexample to the conjecture, which motivated Yuan--Zhang {to make} a revised conjecture in the published version of \cite{YZ26}. 
Here we describe the counterexample and make some remarks on the revised conjecture. 

Recall the construction of invariant adelic line bundles of polarized dynamical systems in \cite[\S6.1]{YZ26}. 
We state our counterexample as an existence theorem as follows.

\begin{theorem}[Theorem \ref{thm counterexample super}] 
 \label{thm counterexample super intro}
Denote $k = \mathbb{Z}$, $F = \mathbb{Q}(t)$, and $X = \mathbb{P}^1_F$.
Let $f:X\to X$ be the square map.
Let $L$ be the tautological line bundle $\mathcal{O}(1)$ on $X$, and let $\overline{L}$ be the $f$-invariant adelic line bundle on $X/\ZZ$ extending $L$. 
Then there exists a generic sequence $\{x_m\}_m$ of $X(\overline{F})$, numerically small with respect to $\overline{L}$, such that for any valuation $v$ of $F$ satisfying the condition that $F_v$ is a finite extension of $\QQ_w$ for the restriction $w=v|_\QQ$,  
 the Galois orbit of $\{x_m\}_m$ is not equidistributed in $X_v^{\mathrm{an}}$ for any probability measure on $X_v^{\mathrm{an}}$. 
\end{theorem}

There are plenty of valuations $v$ of $F$ satisfying the condition of the theorem. 
First, every archimedean $v$ automatically satisfies the condition. 
Second, for any (non-trivial) non-archimedean place $w$ of $\QQ$, as
$\QQ_w$ is transcendental over $\QQ$, there are plenty of embeddings $F\to \ol\QQ_w$, and every such embedding induces a valuation $v$ on $F$ by restriction. 

Let us briefly recall our construction of the small sequence. 
The idea is to choose a suitable sequence  $\{p_m\}_m$ of polynomials in $\QQ[t]$ and a corresponding sequence $\{n_m\}_m$ of integers, and then define $x_{2m-1}$ to be the $\overline{F}$-point
$$
x_{2m-1} = \left[1: (p_m)^{1/n_m}\right]\in X(\ol F).
$$
We may carefully choose the sequences so that $\{x_{2m-1}\}_m$ is small, and 
the $\mu_{x_{2m-1},v}$ on $X_v^\an$ converges to the Dirac measure of the origin or the infinity. 
{This is achieved} by some explicit estimates. 
On the other hand, we define $x_{2m}=[1:2^{\frac{1}{m}}]\in X(\ol F)$.
It is easy to see that the measure $\mu_{x_{2m},v}$ on $X_v^\an$ converges to the Haar measure of the unit circle when $v$ is archimedean, or the Dirac measure at the Gauss point when $v$ is non-archimedean.
As the limit measures of $\mu_{x_{2m-1},v}$  and $\mu_{x_{2m},v}$ do not agree, we see that $\mu_{x_{m},v}$ does not converge for any $v$ in our consideration.

Let us return to Conjecture \ref{equi conj}, which fails
by the counterexample in Theorem \ref{thm counterexample super intro}.  
The failure might be explained {by the fact that} the condition ``$h_\OL^\OH$-small for every  $\OH\in \HPic (F/k)_{\intb}$'' is not as strong as it seems.
In fact, in the dynamical case (for example, in Theorem \ref{thm counterexample super intro}), since the Moriwaki height $h_\OL^\OH(X)=0$, ``$h_\OL^\OH$-small'' means $h_\OL^\OH(x_m)\to 0$. 
We claim that ``$h_\OL^\OH$-small'' for a nef and big $\OH \in \HPic (F/k)_{\nef}$ implies ``$h_\OL^{\OH'}$-small'' for all integrable $\OH'\in  \HPic (F/k)_{\intb}$. 
In fact, by multilinearity, we can assume that $\OH'$ is nef. 
By Yuan's arithmetic Siu inequality (cf. \cite[Thm. 5.2.2(2)]{YZ26}), 
 $\overline H-\epsilon\OH'$ is big for some rational number $\epsilon>0$. 
Then we have $h_\OL^{\OH}(x_m)\geq h_\OL^{\epsilon\OH'}(x_m)\geq0$. 
Then ``$h_\OL^\OH$-small'' implies ``$h_\OL^{\OH'}$-small''. 

Now we sketch Conjecture 5.4.1 in the published version of \cite{YZ26}, which revises the arxiv version (i.e. Conjecture \ref{equi conj}). This new version will not be used in this paper. 
The idea is to replace the condition ``numerically small'' by a stronger condition ``small''.
First, via the Deligne pairing (or just the norm map), we have a 
\emph{vector-valued height function} 
$$
\fh_{\OL}:X(\ol F)\lra \HPic (F/k)_{\intb, \QQ}.
$$
If $X$ is projective over $F$, we also have a {vector-valued} height  $\fh_{\OL}(X)\in  \HPic (F/k)_{\intb, \QQ}$.
The {vector-valued} heights refine the Moriwaki heights in the sense that 
$$
h_{\OL}^\OH(x)=\fh_{\OL}(x)\cdot \OH^{d-1}, \qquad 
h_{\OL}^\OH(X)=\fh_{\OL}(X)\cdot \OH^{d-1}. 
$$
There is a canonical quotient group
$$\wh \NS(F/k)=\wh\Pic(F/k)/\wh\Pic^0(F/k),$$
which can be viewed as an adelic version of the classical N\'eron--Severi group. 
We say that a sequence $\{\alpha_m\}_{m\geq1}$ in $\wh \NS(F/k)_\RR$ 
\emph{converges} to an element $\alpha\in \wh \NS(F/k)_\RR$ \emph{under the linear topology} if there exists a finite-dimensional $\RR$-subspace $W$ of $\wh \NS(F/k)_\RR$ containing $\{\alpha_m\}_m$ and $\alpha$ such that $\{\alpha_m\}_m$ converges to $\alpha$ in $W$ under the Euclidean topology of $W$. 
A sequence $\{x_m\}_{m\geq1}$ of $X(\ol F)$ is \emph{small} if $\fh_{\OL}(x_m)$ converges to $\fh_{\OL}(X)$ under the linear topology. 
Then Yuan--Zhang revised the conjecture using this strong notion ``small''  in the case where $X$ is projective.

\subsection{Equidistribution at fully transcendental valuations}

By the counterexample in Theorem \ref{thm counterexample super intro}, 
Conjecture \ref{equi conj} fails if $F_v$ is algebraic over $\QQ_w$ in the case $k=\ZZ$.
This disproves the conjecture for all archimedean $v$, but only for some special non-archimedean $v$. 
Based on the foundational work of Temkin \cite{Tem} about valuations, we introduce a notion of \emph{fully transcendental valuations} of $F$ in \S\ref{tran deg section}, which can be viewed as the opposite extreme of the property that $F_v$ is algebraic over $\QQ_w$. 
Then our main theorem proves the conjecture for every {fully transcendental} 
valuation $v$ on $F$.  

To be precise, let $k$ be either $\ZZ$ or a field. 
Denote by $k'=\Frac(k)$ the fraction field. 
Let $F$ be a finitely generated field over $k$. 
Denote by $\CM(F/k)$ the set of valuations $v$ on $F$ such that $w=v|_{k'}$ is trivial if $k$ is a field.
The \emph{topgebraic generating degree} of $F_v$ over $k'_w$, 
denoted by
$\ttrd(F_v/k'_w)$, is the smallest cardinality $\# S$ among all subsets $S$ of $F_v$ which generates $F_v/k'_w$ topgebraically, i.e. $F_v$ is contained in the completion of the algebraic closure of $k'_w(S)$ in a natural sense. 

For a valuation $v\in \CM(F/k)$, it is easy to have 
$$\ttrd(F_v/k'_w)\leq \trd(F/k'). $$
We say that $v$ is \emph{fully transcendental} over $k$ if 
$$
\ttrd(F_v/k'_w)= \trd(F/k'). 
$$
Denote by $\CM(F/k)^\ft$ the subset of $\CM(F/k)$ consisting of \emph{fully transcendental} valuations.

As an easy example, let $F$ be the function field of a smooth projective curve $Y$ over $\QQ$. A non-archimedean valuation $v$ of $F$ extending a non-archimedean place $w$ of $\QQ$ defines a point in the Berkovich curve $Y_{\QQ_w}^\an$. 
It follows that $v$ is a fully transcendental valuation of $F$ if and only if $v$ is a point of type $2,3,4$ in $Y_{\QQ_w}^\an$; i.e. $v$ is the image of a point of type $2,3,4$ via the natural map $Y_{\CC_w}^\an\to Y_{\QQ_w}^\an$.
The valuation $v$ in Theorem \ref{thm counterexample super intro} always corresponds to a point of type 1, i.e., a classical point.

Finally, the following is our main theorem, which proves that the equidistribution conjecture (Conjecture \ref{equi conj}) holds at fully transcendental valuations. 

\begin{thm}[equidistribution at fully transcendental valuations] \label{equi ft}
	\kkk
	Let $F$ be a finitely generated field over $k$.
	Let $X$ be a geometrically integral quasi-projective variety over $F$.
	Let $\overline L$ be a nef adelic line bundle on $X/k$ such that $\deg_{\wt L}(X/F)>0$.
	Let $\{x_m\}_m$ be a generic sequence  in $X(\overline F)$ which is numerically small with respect to $\OL$.
	Then for every fully transcendental valuation $v\in\CM(F/k)$, the Galois orbit of $\{x_m\}_m$ is equidistributed in $X_v^\an$ for $\d\mu_{\overline L,v}$.
\end{thm}

It is worth noting that there are still valuations $v\in \CM(F/k)$ which are not covered in either Theorem \ref{thm counterexample super intro} or Theorem \ref{equi ft}. 
In fact, for $F=\QQ(t)$ as in Theorem \ref{thm counterexample super intro}, if $v$ is induced by an embedding $F \to \CC_w$ which sends the generator $t$ to an element of $\CC_w \setminus \ol\QQ_w$, then $F_v$ is not finite over $\QQ_w$, and $v$ is not fully transcendental by $\ttrd(F_v/\QQ_w)=0$.

If $F$ is endowed with a structure of a proper adelic curve in the sense of Chen--Moriwaki, the theorem was proved for $v$ under some strong assumptions. 
In fact, if $X$ is furthermore projective over $F$, then Chen--Moriwaki \cite[Thm. F]{CM24} proved the equidistribution at every $v\in\CM(F/k)$ with a strictly positive mass under the adelic curve structure. This result was further extended to the quasi-projective setting by the recent work of Biswas--Cai \cite{BC26}. 
These results imply Theorem \ref{equi ft} for divisorial valuations $v$. 
In this paper, we will obtain the divisorial case from an equidistribution theorem of Yuan--Zhang \cite{YZ26}, and the major effort of this paper is to deduce the fully transcendental case of Theorem \ref{equi ft} from the divisorial case.

\subsection{Idea of proof}\label{sec idea}

With many preliminary results in \S\ref{tran deg section} and \S\ref{NA preparation section}, we will prove our main theorem (Theorem \ref{equi ft}) in \S\ref{equi on NS fibers section}.
We will actually prove a slightly stronger variant of the theorem in Theorem \ref{equi ft2}.
The variant implies Theorem \ref{equi ft} by Lemma \ref{equivalence}, and they are equivalent if $k=\ZZ$ or $k$ is a field of characteristic 0. 
If $k$ is a field of {positive characteristic}, they might still be equivalent, but we do not know how to prove it due to some delicate issue of topgebraic generating degrees. 

We now sketch the strategy of the proof of the main theorem. 
Our starting point is the following equidistribution theorem in the relative setting from \cite[Thm. 5.4.6]{YZ26}, which in turn extends the equidistribution theorem of Moriwaki \cite{Mor3}.

\begin{thm}[relative equidistribution] \label{equi relative}
\kkk
Let $K$ be a number field if $k=\ZZ$; let $K$ be a function field of one variable over $k$ if $k$ is a field.
Let $\pi:\CX\to Y$ be a flat morphism of relative dimension $n$ of quasi-projective varieties over $K$.
Let $X\to \Spec F$ be the generic fiber of $\CX\to Y$.
Denote $d=\dim Y+1$.

Let $\overline L$  be an element of $\wh\Pic (X/k)_{\nef}$ such that 
$\deg_{\wt L}(X/F)>0$.
Here $\wt L$ is the image of $\OL$ in $\wh\Pic (X/F)_{\nef}.$
Let $\overline H$ be an element of $\wh\Pic (Y/k)_{\nef}$ satisfying the Moriwaki condition that $\overline H$ is nef, $\OH^{d}=0$ and $\wt H^{d-1}>0$. Here $\wt H$ is the image of $\OH$ in $\wh\Pic (Y/K)_{\nef}$.
	
Let $\{x_m\}_m$ be a generic and $h_{\overline L}^{\OH}$-small sequence in $X(\overline F)$.
Then for any place $w$ of $K$, there is a weak convergence
	$$
	\frac{1}{\deg(x_m)} \delta_{\triangle(x_m),w}\, c_1(\pi^*\OH)_w^{d-1}
	\lra \frac{1}{\deg_{\wt L}(X/F)} c_1(\OL)_w^{n} c_1(\pi^*\OH)_w^{d-1}
	$$
	of measures on $\CX_w^\an$.
	Here $\triangle(x_m)\subset \CX$ denotes the Zariski closure of the image of $x_m$ in $\CX$, and
		$\delta_{\triangle(x_m),w}$ denotes the Dirac current of $\triangle(x_m)_{K_w}^\an$ in $\CX_w^\an$.
\end{thm}

From the setting of Theorem \ref{equi ft} to that of Theorem \ref{equi relative}, we take an extra subfield $K$ of $F$ satisfying the requirement, take $w=v|_K$ to be the restriction, 
and take a quasi-projective model $\CX\to Y$ of $X\to \Spec F$ over $K$. 
Let $v$ be a fully transcendental valuation of $F/k$, viewed as a point of $Y_w^\an$ for $w=v|_K$. 
Note that the fiber of $\CX_w^\an\to Y_w^\an$ above $v$ is exactly $X_v^\an$.  
In Theorem \ref{equi relative}, if we can find an $\OH$ such that $c_1(\OH)_w^{d-1}$ is equal to a positive multiple of the Dirac measure $\delta_v$, then Theorem \ref{equi relative} for $(w,\OH)$ exactly gives Theorem \ref{equi ft} for $v$. 

This amounts to {solving} a non-archimedean Monge--Ampere equation. 
This approach {meets} two difficulties. 
The first difficulty is the restriction by the Moriwaki condition in the theorem. With  huge effort, the restriction might be removed following an idea of Chen--Moriwaki \cite[Thm. 9.11.2]{CM24}, which treated equidistribution in a different setting.
The second difficulty comes from the Monge--Ampere equation. 
The breakthrough work on this problem by Boucksom--Favre--Jonsson \cite{BFJ15} and its sequel \cite{BGJKM20, BGM22} proves existence of solutions under some technical assumptions. To prove an {unconditional} result, we can relax the equation drastically; for example, it suffices to find integrable adelic line bundles $\OH_1, \dots, \OH_{d-1}$ on $Y/k$ such that $c_1(\OH_1)\cdots c_1(\OH_{d-1})$ is a linear combination of Dirac measures, and that one of the Dirac measures is supported at $v$. This problem can be viewed as a non-archimedean Poisson equation. 

Due to technicality, we will address the above approach in an ongoing project. 
In this paper, we will take a different approach to prove the main theorem 
(Theorem \ref{equi ft}). Our idea is as follows. 

\paragraph*{Step 1: Equidistribution at divisorial valuations.}
In Step 1, we prove Theorem \ref{equi ft} for all divisorial valuations 
$v \in \CM(F/k)$, i.e. points of $Y_{K_w}^\an$ coming from discrete valuations induced by irreducible components of special fibers of integral models of $Y$ over $O_{K_w}$. 
These valuations are automatically fully transcendental. 
To achieve this, take the polarization $\OH$ in Theorem \ref{equi relative}
to be a model adelic line bundle, so $c_1(\OH)_w^{d-1}$ is a linear combination of Dirac measures supported at finitely many divisorial points. 
We can further arrange that $v$ is one of the divisorial points in the support. 

\paragraph*{Step 2: Reduce to an approximation theorem.}
Note that the divisorial points are dense in $Y_w^\an$; see \cite[Corollaire 4.5]{Poi13} and \cite[Proposition A.5]{GM}. 
Once we have the equidistribution theorem for divisorial valuations, 
we plan to approximate {a general fully transcendental valuation} $v$ by divisorial valuations. 
More precisely, we have the following approximation theorem. 

\begin{theorem}[Theorem \ref{approximation1}]
\label{approximation intro}
Let $K$ be an algebraically closed non-archimedean field. 
Let $F$ be a finitely generated field extension over $K$ of {transcendence degree} $1$. 
Let $X$ be a geometrically integral quasi-projective variety of dimension $n$ over $F$. Let $\ol L$ be a nef adelic line bundle on $X/O_K$ such that $\deg_{\wt L}(X/F)>0$.
Let $\{\alpha_m\}_{m\geq1}$ be a generic sequence of effective 0-cycles of $X$ with bounded $\TL_K$-heights over $K$.  If the weak convergence
	$$\mu_{\alpha_m,v}\longrightarrow \d\mu_{\OL,v}$$
holds for all divisorial valuations $v\in \CM(F/K)$, then it holds for all fully transcendental valuations $v\in\CM(F/K)^\ft$.
\end{theorem}

The datum $(K, F, X, \ol L)$ of Theorem \ref{approximation intro} comes from some slight variants of the datum $(K_w, F_v, X_{F_v}, \ol L |_{X_{F_v}})$ (with $w=v|_K$) from the setting of Theorem \ref{equi ft}. 
In particular, the approximation theorem is completely over the non-archimedean field $K$ (instead of the original global field $K$). 
The key condition to guarantee such a uniformity is that $\alpha_m$ has bounded 
$\TL_K$-heights, where $\TL_K$ is the image of $\OL$ in $\wt\Pic(X/K)$, and the height function 
$$
h_{\TL_K}: X(\ol F) \lra \RR
$$
is defined by viewing $F$ as a function field of one variable over $K$.

\paragraph*{Step 3: Choose special test functions.}
The weak convergence in Theorem \ref{approximation intro} means 
$$\int_{X_v^\an} f \mu_{\alpha_m,v}\longrightarrow \int_{X_v^\an} f  \d\mu_{\OL,v}$$
for all continuous and compactly supported functions $f\in C_c(X_v^\an)$. 
In this step, we choose special test functions $f$, which are more convenient to work on.

Let $Y$ be the unique smooth projective model of $F$ over $K$, and let
$U\to Y$ be a quasi-projective model of $X$ over $F$. 
Denote by 
$C_{P}^{\m}(U^\an)_{\rm pure}$
the set of functions on $U^\an$ of the form
$$
(\psi_1^*f_1)(\psi_2^*f_2)\cdots (\psi_N^*f_N),
$$
where $N\geq 1$ is an integer, every $\psi_i: U\to\PP^1$ is a $K$-morphism, and every $f_i: \PP^{1,\an}\to \RR$ is a model function. 
Denote 
$$
C_{P,c}^{\m}(U^\an)_{\rm pure}= C_{P}^{\m}(U^\an)_{\rm pure} \cap C_c(U^\an).
$$
Denote by $C_{P,c}^{\m}(U^\an)$ the $\QQ$-vector subspace of $C_c(U^\an)$ generated by $C_{P,c}^{\m}(U^\an)_{\rm pure}$.

In Theorem \ref{test function}, we prove that if $U$ is affine, then $C_{P,c}^{\m}(U^\an)$ is actually dense in $C_c(U^\an)$ under the uniform topology. 
This result is in the spirit of a theorem of Gubler \cite[Thm. 7.12]{Gub98}
on the density of model functions, and the proof is an application of the classical Stone--Weierstrass theorem. 

We take the morphisms $\psi_i$ to $\PP^1$ because it is hard to work on the high-dimensional space $U^\an$ especially when it involves integral models and skeletons.

\paragraph*{Step 4: Reduce to a uniform H\"older property.}
By the above step, we only need to prove the approximation theorem for every special test {function} $g\in C_{P,c}^{\m}(U^\an)_{\rm pure}$. 
In this step, we reduce the problem to a  H\"older property (cf. Theorem \ref{holder}), which asserts that 
there exist positive constants $A_0, A_1$ such that for any generic $x\in X(\ol F)$ and any two points $v, w\in Y^\an$ of type $2,3,4$, we have
	$$
	\left|
		\int_{X_v^\an} g \mu_{x,v}-\int_{X_w^\an} g \mu_{x,w} \right|
	\leq
	A_1\cdot\sqrt{h_{\TL}(x)+A_0}\cdot\sqrt{\dist(v,w)}.
	$$
Here $\dist(v,w)$ is the canonical distance on $Y^\an$ induced by the canonical metrics on its skeletons. 
Here we say that  $x\in X(\ol F)$
is \emph{generic} if it lies in a fixed Zariski open subset $X'$ of $X$. 
 
The property asserts that $\displaystyle\int_{X_v^\an} g \mu_{x,v}$ is uniformly continuous in $v$, and thus we can approximate it by divisorial points. 
Note that $\dist(v,w)$ is finite between two points $v,w$ of type $2,3,4$; see  
Theorem \ref{finite distance}, which is originally from
 Baker--Payne--Rabinoff \cite[Cor. 5.7]{BPR}. 
However, $\dist(v,w)$ becomes infinity if we allow one of the points to be of type 1. 
This is exactly the reason why the method {does not work} for points of type 1.

\paragraph*{Step 5: Reduce to a distance inequality.}

To illustrate the idea, let us assume $g=\psi_1^* f_1$ (and thus $N=1$), as the general case is not much more difficult. 

To prove the H\"older property, we take semistable models $\CY$ of $Y$ and $\CP_1$ of $\PP^1$ over $O_K$. The reduction graphs give skeletons $\Gamma$ of $Y^\an$
and $\Gamma_1$ of $\PP^1$. 
Assume that the morphisms involved are extended to the semistable models. 
As the function $g$ is essentially Lipschitz continuous, the uniform H\"older property is implied by a distance inequality (Theorem \ref{distance}), which asserts that 
there exist positive constants $A_0, A_2$ such that for any generic closed point $x\in X$, and for any two points $v, w\in \Gamma$, 
we have
$$
\frac{1}{\deg(x)} \dist_{\Gamma_1}([x_v]_{\Gamma_1},[x_w]_{\Gamma_1})
\leq A_2\cdot\sqrt{h_{\TL}(x)+A_0}\cdot \sqrt{\dist_\Gamma(v,w)}.
$$
Here $x_v$ denotes the Galois orbit of $x$ in $U_v^\an$, and $[x_v]_{\Gamma_1}$ denotes its image in the skeleton $\Gamma_1$. 
The distances are on the skeletons, but a further limit can convert them to the Berkovich spaces.

If $x\in X(\ol F)$ is not $F$-rational, we have to make lots of extra effort on the problem. In this case,   
$\dist_{\Gamma_1}([x_v]_{\Gamma_1},[x_w]_{\Gamma_1})$
is interpreted {as the distance} between two effective 0-cycles introduced in \S\ref{subsec transport plan}. The distance is calculated by choosing paths between points of the two cycles, and it turns out that the situation naturally fits the setting of a finite-dimensional linear programming problem.  
Once the terminology is {set up}, the treatment is similar to the case of rational points.

\paragraph*{Step 6: Prove the distance inequality.}

It remains to prove the distance inequality. 
Our proof of the distance inequality comes from the philosophy of \emph{partial heights} introduced in Xie--Yuan \cite{XY}, which was used there to prove some case of the geometric Bombieri--Lang conjecture.
The upshot is to use 
$h_{\TL}(x)$ to bound some integral, and use the integral to bound the distance $\dist_{\Gamma_1}([x_v]_{\Gamma_1},[x_w]_{\Gamma_1})$. 

First, for a generic closed $x\in X$, we have a bound of the form
$$
h_{\psi_1^*\CO(1)}(x)\leq A_3(h_{\TL}(x)+A_0).
$$
This follows from a height inequality of Yuan--Zhang \cite[Thm. 5.3.7]{YZ26}.

Second, we interpret the height in terms of integration by 
$$h_{\psi_1^*\CO(1)}(x)=\frac{1}{\deg(x)} \int_{Y'^\an} \psi_1'^*c_1(\ol \CO(1)).
$$
Here $Y'$ is the normalization of the multi-section of $U\to Y$ corresponding to $x$, the morphism $\psi_1':Y'\to \PP^1$ is induced by $\psi_1:U\to \PP^1$, and we will choose a reasonable piecewise smooth metric $\|\cdot\|$ on $\CO(1)$, so that the pullback $\psi_1'^*c_1(\ol \CO(1))$ is a $(1,1)$-form on $Y'^\an$. 
This process requires the theory of differential forms on Berkovich spaces originally introduced by Chambert-Loir--Ducros \cite{CLD}, and refined by Gubler--Jell--Rabinoff  \cite{GJR1}. We mainly require the further theory of Gubler--Jell--Rabinoff \cite{GJR2} on metric graphs and Berkovich curves. 

Third, take a piecewise smooth and strictly positive $(1,1)$-form $\omega_0$ on $\Gamma_1$. We can choose the metric of $\ol \CO(1)$ so that $c_1(\ol \CO(1))$
is greater than a positive multiple of $\omega_0$ on $\Gamma_1$. 
It follows that the integral 
$$
I(\omega_0)=\int_{\Gamma'}\psi_1'^*\omega_0
$$
is bounded above by $h_{\psi_1^*\CO(1)}(x)$ and thus by $h_{\TL}(x)$. 
Here $\Gamma'$ is the skeleton of a semistable integral model $\CY'$ of $Y'$ dominating $\CY$ and $\CP_1$. 

Finally, by computing the integral $I(\omega_0)$ on edges of $\Gamma'$, we eventually get a bound of the distance $\dist_{\Gamma_1}([x_v]_{\Gamma_1},[x_w]_{\Gamma_1})$.
To illustrate the idea, assume that $x\in X(F)$ is a rational point, and thus assume that $Y'\to Y$, $\CY'\to \CY$, and $\Gamma'\to \Gamma$ are isomorphisms.
We further assume that $\dist_\Gamma(v,w)$ is given by a line segment $\ell$ connecting $v$ and $w$ contained in an edge of $\Gamma$. 
Denote by  $\ell'$ the preimage of $\ell$ in $\Gamma'$, and by $\bar\ell$ the image of $\ell'$ in $\Gamma_1$, which is also assumed to be a line segment contained in an edge of $\Gamma_1$. 
The canonical metrics of the graphs give canonical coordinates $t, t',\bar t$ on $\Gamma, \Gamma', \Gamma_1$. 
Take $\omega_0$ with restriction 
$$
\omega_0|_{\bar \ell}= \d'\bar t\wedge \d''\bar t. 
$$
Then the pull-back 
$$
\psi_1'^*\omega_0|_{\ell'}=d_{\ell'}(\psi_1)^2 \d't'\wedge \d''t'. 
$$
Here the expansion factor of the linear map $\ell'\to \bar \ell$ is given by  
$$
d_{\ell'}(\psi_1)=\frac{l(\bar\ell)}{l(\ell')}=\left| \frac{\d \bar t}{\d t'}\right|, 
$$
where $l(\ell')$ and $l(\bar \ell)$ denote the lengths. 
It follows that 
$$
I(\omega_0)\geq \int_{\ell'} \psi_1'^*\omega_0
=\int_{\ell'} d_{\ell'}(\psi_1)^2 \d't'\wedge \d''t'
= d_{\ell'}(\psi_1)^2 l(\ell'). 
$$
Now we can get the distance inequality by 
$$
\dist_{\Gamma_1}([x_v]_{\Gamma_1},[x_w]_{\Gamma_1})=l(\bar \ell)
=d_{\ell'}(\psi_1) l(\ell')
$$
and 
$$
\dist_{\Gamma}(v,w)=l(\ell)
=l(\ell'). 
$$

\subsection{Notation and terminology}

We will take the terminology of Yuan--Zhang \cite{YZ26}, and we will review the relevant terminology before using it.
The following are some extra notation used in this paper. 

\begin{enumerate}[(1)]
\item For a real number $t$, define $\log^+t = \log\max\{t,1\}.$

\item By a \emph{variety}, we mean an integral scheme separated and of finite type over the base field. 
By a \emph{curve}, we mean a variety of dimension 1. 

\item
For a field $F$ with a valuation $v$, denote by $F_v$ the completion of $F$ with respect to $v$. 

\item 
For a field $F$, denote by $\ol F$ an algebraic closure of $F$.
If $F$ is complete with respect to a valuation, then the valuation extends uniquely to the algebraic closure $\ol F$, and we denote by $\wh{\ol F}$ the completion of the algebraic closure $\ol F$.  

\item 
By a \emph{analytic field} $K$, we mean a complete field equipped with a non-archimedean real-valued valuation.
We denote by $O_K$ the valuation ring of $K$, and by $\wt K$ the residue field of $O_K$. 
It is called a \emph{non-archimedean field} if the valuation is non-trivial. 
By convention, any inclusion of analytic fields is assumed to respect the valuations.

\item 
For a topological space $M$, denote by $C(M)$ the space of real-valued continuous functions on $M$, and denote by $C_c(M)$ the space of real-valued, compactly supported, and continuous functions on $M$.
\end{enumerate}

\subsubsection*{Acknowledgment}
The authors would like to thank Junyi Xie for many valuable communications. 
The authors are supported by the grants NO. 12250004 and NO. 12321001 from the National Natural Science Foundation of China. 
The fourth author would also like to thank the Xplorer Prize from the New Cornerstone Science Foundation and the support of the Sino--Russian Mathematics Center.

\subsubsection*{AI Statement}
When writing the paper, the authors used GPT to search for well-known mathematics facts and references; after finishing the first draft, the authors used GPT to check grammars, typos, and correctness of proofs.

\section{Fully transcendental points on Berkovich spaces} \label{tran deg section}
In this section, we review \emph{Berkovich spaces} and \emph{fully transcendental points}. These notions are central in the statement of our main theorem.  

\subsection{Berkovich spaces}
Berkovich spaces are best known as analytic spaces associated with varieties over non-archimedean fields, whose foundation was introduced by Berkovich \cite{Ber1}. 
By Berkovich \cite[\S1]{Ber2}, the base fields are relaxed to be Banach rings, and the old construction works similarly. 
The {latter} case plays an important role in the theory of adelic line bundles in \cite{YZ26}.
Here we review the basics of Berkovich spaces following the terminology of 
\cite{YZ26}. 

Let $(R,|\cdot|_\mathrm{Ban})$ be a commutative Banach ring with unity 1.
Namely, $R$ is a commutative ring with unity 1, and $|\cdot|_\mathrm{Ban}:R\to \RR_{\geq0}$ is a map satisfying the following properties:
\begin{enumerate}[(1)]
\item (norm property) $|a|_\mathrm{Ban}=0$ if and only if $a=0$;
\item (triangle inequality) $|a+b|_\mathrm{Ban}\leq |a|_\mathrm{Ban}+|b|_\mathrm{Ban}$ for all  $a,b\in R$;
\item (sub-multiplicativity) $|ab|_\mathrm{Ban}\leq |a|_\mathrm{Ban}\cdot |b|_\mathrm{Ban}$ for all  $a,b\in R$;
\item (completeness) $R$ is complete under the topology induced by $|\cdot|_\mathrm{Ban}$.
\end{enumerate}
In our applications, we will always have one of the following three cases:
\begin{enumerate}[(1)]
\item $R$ is a non-archimedean field, i.e. a complete field with a non-trivial non-archimedean valuation;
\item $R$ is a field, and $|\cdot|_\mathrm{Ban}$ is the trivial valuation $|\cdot|_0$;
\item $R=\ZZ$, and $|\cdot|_\mathrm{Ban}$ is the usual archimedean absolute value $|\cdot|_\infty$.
\end{enumerate}
To refer to the last two cases, in the uniform terminology of \cite[\S1.5]{YZ26}, we will always say that \emph{let $R$ be either $\ZZ$ or a field}.

Let $A$ be a commutative ring over $R$. 
Then the 
\emph{Berkovich space} $\CM(A/R)$, sometimes abbreviated as $\CM(A)$,  is the set of multiplicative
semi-norms on $A$ whose restriction to $R$ is bounded by 
$|\cdot|_\mathrm{Ban}$. 
In particular, we usually abbreviate $\CM(R/R)$ as $\CM(R)$.

For each $x\in \CM(A/R)$, 
denote its corresponding semi-norm on $A$ by $|\cdot|_x:A\to \RR_{\geq0}$.
For any $f\in A$, write $|f|_x$ as $|f(x)|$, which gives a real-valued function $|f|$ on $\CM(A/R)$.
The topology on $\CM(A/R)$ is the weakest one such that the function $|f|:\CM(A/R)\to\RR$ is continuous for all $f\in A$. 

We have an easy partition
$$
\CM(A/R)=\CM(A/R)_\fin\ \cup\ \CM(A/R)_\infty.
$$
where $\CM(A/R)_\fin$ (resp. $\CM(A/R)_\infty$) denotes the subset of non-archimedean (resp. archimedean) semi-norms.

Let $X$ be a scheme over $R$. 
Assume that $X$ is covered by an affine open cover $\{\Spec A_i\}_i$.
Then the \emph{Berkovich space} $(X/R)^\an$, sometimes abbreviated as $X^\an$, is the union of $\CM(A_i/R)$, glued canonically. 
The topology of $(X/R)^\an$ is the weakest one such that each $\CM(A_i/R)$ is an open subspace of $(X/R)^\an$.
As in the affine case, we also have a partition
$$
(X/R)^\an=(X/R)_\fin^\an\ \cup\ (X/R)^\an_\infty.
$$

Note that in the affine case $X=\Spec A$, the definition gives 
$(X/R)^\an=\CM(A/R)$. 
We have the following extra concepts.

\begin{enumerate}[(1)]
\item \emph{Residue field.}
For each $x\in \CM(A/R)$, the corresponding semi-norm $|\cdot|_x$ induces a norm on the integral domain $A/\ker(|\cdot|_x)$. The completion of the fraction field of  $A/\ker(|\cdot|_x)$ is called the \emph{completed residue field} of $x$ and denoted by $\CH_x$. Denote by $|\cdot|$ the valuation (multiplicative norm) on $\CH_x$ induced by $|\cdot|_x$. 
Then $|\cdot|_x:A\to \RR$ is equal to the composition 
$$A\lra \CH_x\overset{|\cdot|}{\lra} \BR.$$
We write the first map as $f\mapsto f(x)$, which is compatible with the convention $|f|_x=|f(x)|$.
The notation $\CH_x$ generalizes to any scheme $X$ over $R$.

\item \emph{Contraction.}
There is a canonical contraction map $\kappa: (X/R)^\an\to X$. It suffices to describe it in the case $X=\Spec A$. 
For each $x\in \CM(A)$, 
the kernel of the map $|\cdot|_x:A\to \RR$ is a prime ideal of $A$, and thus defines an element $\kappa(x)\in \Spec A$. 

\item \emph{Functoriality.}
Any morphism $f:X\to Y$ over $R$ induces a continuous map 
$$f^\an:(X/R)^\an\lra (Y/R)^\an.$$ 
For any point $v\in Y^\an$, the \emph{fiber} 
$$X_v^\an=(X/R)_v^\an=(f^\an)^{-1}(v),$$
 is a subspace of $(X/R)^\an$, and is canonically homeomorphic to the Berkovich space $(X_{\CH_v}/\CH_v)^\an$.
\end{enumerate}
By \cite[Lem. 1.1, Lem. 1.2]{Ber2}, we have the following basic topological properties: 
\begin{enumerate}[(1)]
\item If $X$ is separated and of finite type over $R$, then $(X/R)^\an$ is Hausdorff.
\item If $X$ is of finite type over $R$, then $(X/R)^\an$ is locally compact.
\item If $X$ is projective over $R$, then $(X/R)^\an$ is compact.
\end{enumerate}

In the definition, we do not assume that $X$ is reduced. 
This does not bring much trouble, since there is always a canonical 
homeomorphism $(X_{\red}/R)^\an\to (X/R)^\an$, where $X_\red$
denotes the reduced structure of $X$. 
Moreover, we can further write 
$(X_{\red}/R)^\an$ as a union of the Berkovich spaces associated to the irreducible components of $X_{\red}$.

\subsection{Topgebraic degrees}

To introduce fully transcendental points, we first recall the notion of \emph{topgebraic degree} introduced in Temkin \cite{Tem}. 

In this paper, by an \emph{analytic field}, we mean a complete field equipped with a non-archimedean real-valued valuation. 
It is called a \emph{non-archimedean field} if the valuation is non-trivial. 
By convention, any inclusion of analytic fields is assumed to respect the valuations.

For an analytic field $K$,  denote by $\ol K$ an algebraic closure of $K$, and denote by $\wt K$ the residue field of $K$. 
The valuation of $K$ extends uniquely to $\ol K$. Denote by
$\widehat{\ol K}$ the completion of $\ol K$.

Let $K\subset L$ be an extension of analytic fields.
An element $x\in L$ is \emph{topgebraic} over $K$ if $x\in \widehat{\ol K}\cap L$.
Here we view $\ol K$ as the algebraic closure of $K$ in $\ol L$, and thus 
$\widehat{\ol K}$ is viewed as the closure of $\ol K$ in $\widehat{\ol L}$.

Let $S \subset L$ be a subset.
We say that $S$ is a \emph{topgebraic generating subset} of $L$ over $K$, if every element of $L$ is topgebraic over $\widehat{K(S)}$.
Define the \emph{topgebraic generating degree} of $L$ over $K$ by
$$
\ttrd(L/K) := \inf \{ \#S \mid S \text{ is a topgebraic generating subset of } L \text{ over } K \}.
$$
This is either a non-negative integer or infinity. 
See \cite[\S2.1.9]{Tem} for more details. 

We say that a subset $S \subset L$ is \emph{topgebraically independent} over $K$ if for every $x \in S$, and every non-archimedean valuation of $K(S \setminus \{x\})$ extending the one on $K$, the element $x$ is not topgebraic over $\widehat{K(S \setminus \{x\})}$.
A \emph{topgebraic basis} of $L$ over $K$ is a topgebraically independent topgebraic generating subset $S$ of $L$. 

Before we prove any results about topgebraic generating degrees, let us first note the following basic result. 

\begin{lem} \label{trd compare}
Let $K$ be an analytic field, and $F/K$ be an extension of finite transcendence degree. 
Let $v$ be a valuation of $F$ extending that of $K$. 
Then 
$$\trd(\wt F_v/\wt K)\leq \trd(F/K).$$
\end{lem}

\begin{proof}
This is a special case of Abhyankar's lemma (cf. \cite[Lem. 1(1*)]{Abh}), but we sketch a proof {for convenience}.

Denote $d=\trd(F/K)$. 
We need to prove that any $d+1$ elements $\wt x_1, \dots, \wt x_{d+1}$ of $\wt F_v$ are algebraically dependent over $\wt K$. 
Let  
$x_1, \dots, x_{d+1}\in O_{F_v}\cap F$ be the liftings of
$\wt x_1, \dots, \wt x_{d+1}$. 
There is a non-trivial polynomial relation $f(x_1, \dots, x_{d+1})=0$ in $F$.
We can assume that the polynomial $f$ is primitive, i.e., the maximum of the valuations of the coefficients is 1. 
Then the reduction of $f(x_1, \dots, x_{d+1})=0$ gives a non-trivial polynomial relation of $\wt x_1, \dots, \wt x_{d+1}$.
\end{proof}

The following are some basic properties of the topgebraic degree.

\begin{prop} \label{prop of ttrd}
Let $L/K$ be an extension of analytic fields. Assume that
$\ttrd(L/K)<\infty$. Then the following hold.
\begin{enumerate}[(1)]
\item
There exists a topgebraic basis of $L/K$, and the cardinality of any such a topgebraic basis is $\ttrd(L/K)$. 
	
\item
For  any dense subfield $L_0$ of $L$, 
$$\trd(\wt L/\wt K)\leq \ttrd(L/K) \leq \trd(L_0/K).$$

\item
Let $E$ be an intermediate analytic field of $L/K$. Then
$$\ttrd(E/K)\leq \ttrd(L/K)\leq \ttrd(L/E)+\ttrd(E/K).$$	
If the residue field of $K$ has characteristic 0, then the second inequality is an equality. 

\end{enumerate}		
\end{prop}

\begin{proof}
Part (1) is from \cite[Theorem 3.2.3]{Tem}.

For (2), assume that $L$ is topgebraically generated by $x_1,\dots, x_d\in L$ over $K$. It suffices to prove $\trd(\wt L/\wt K)\leq d$. 
Denote by $L'$ the topological closure of $K(x_1,\dots, x_d)$ in $L$.
Then the residue field extension
$\wt L/\wt L'$ is algebraic.  Hence
$$
\trd(\wt L/\wt K)
=\trd(\wt L'/\wt K)
\leq \trd(K(x_1,\dots, x_d)/K)
\leq d.
$$
Here the first inequality follows from Lemma \ref{trd compare}.

For (3), the first inequality is a non-trivial consequence of \cite[Theorem 3.2.3]{Tem}, and the second inequality follows from the definition. 
The condition of the equality follows from
\cite[Thm. 4.2.7]{Tem}. 
\end{proof}

\begin{remark}
The second inequality in (3) is not necessarily an equality if 
the residue field of $K$ {has positive characteristic}. See \cite[Thm. 5.2.2]{Tem}.
\end{remark}

\subsection{Fully transcendental valuations}

Here we introduce two {important} notions: fully transcendental valuations and divisorial valuations.

\subsubsection{Extension of fields}

We have the following key definition, which concerns when the inequalities in Proposition \ref{prop of ttrd}(2) become equalities. 

\begin{defn}\label{def fully tran}
Let $F/K$ be a field extension of finite transcendence degree.
Let $v$ be a non-archimedean valuation of $F$ and let $v_0=v|_K$ be its restriction to $K$.   
\begin{enumerate}[(1)]
\item
We say that $v$ is \emph{fully transcendental} over $K$, or that $F_{v}/K_{v_0}$ is \emph{fully transcendental}, if
\[
\ttrd(F_{v}/K_{v_0})=\trd(F/K).
\]

\item
We say that $v$ is \emph{divisorial} over $K$, or that $F_{v}/K_{v_0}$ is \emph{divisorial}, if
\[
\trd(\wt F_v/\wt K_{v_0})=\trd(F/K).
\]
\end{enumerate}
\end{defn}

Denote by $\CM(F/K)^\ft_{v_0}$ the set of all fully transcendental valuations $v$ of $F$ extending $v_0$. 
If $K$ is an analytic field under a valuation $v_0$, then we may write $\CM(F/K)^\ft$ for $\CM(F/K)^\ft_{v_0}$. 

As a basic fact, every divisorial valuation is a fully transcendental valuation. In fact, with the notation of the definition,  Proposition \ref{prop of ttrd}(2) gives 
\[
\trd(\wt F_v/\wt K_{v_0})\leq  \ttrd(F_{v}/K_{v_0})\leq \trd(F/K).
\]
If $v$ is divisorial, then the inequalities are equalities.

Let $R$ be an integral domain, which includes especially the case $R=\ZZ$.
Let $F$ be a finitely generated field over $R$, i.e. a finitely generated field over the fraction field $\Frac(R)$ of $R$. 
By abuse of notation, we may also write the transcendence degree 
$\trd(F/\Frac(R))$ as $\trd(F/R)$.
A valuation $v$ of $F$ is  
\emph{fully transcendental} (resp. \emph{divisorial}) over $R$ if it is \emph{fully transcendental} (resp. \emph{divisorial}) over $\Frac(R)$.

\begin{example} \label{point types}
Let $K$ be an algebraically closed non-archimedean field, and $F=K(t)$ be the function field of $Y=\PP^1_K$.
Recall that points of $Y^\an$ are classified into 4 types by \cite[\S1.4.4]{Ber1} (cf. \cite[\S1]{Bak}), and every point corresponds to a decreasing sequence of closed discs in $K$. 
We have a natural injection $\CM(F/K)\to Y^\an$, and thus we can view a valuation $v\in \CM(F/K)$ as a point of $Y^\an$. 
Then $v$ is fully transcendental if and only if it corresponds to a point of type $2,3,4$ on $Y^\an$; and  
$v$ is divisorial if and only if it corresponds to a point of type $2$ on $Y^\an$.
\end{example}

We will need the following nice result. 

\begin{lem} \label{sub of fully trans}
Let $F/E/K$ be field extensions of finite transcendence degrees.
Let $v$ be a non-archimedean valuation of $F$. Denote by 
$w=v|_E$ and $v_0=v|_K$ the restrictions.
Then the following are true.
\begin{enumerate}[(1)]
\item
If $F_v/K_{v_0}$ is fully transcendental, then
$F_v/E_w$ and $E_w/K_{v_0}$ are also fully transcendental. 
If the residue field $\wt F_v$ has characteristic 0, then the converse statement also holds.

\item
If $F_v/K_{v_0}$ is divisorial, then
$F_v/E_w$ and $E_w/K_{v_0}$ are also divisorial. 
The converse statement always holds.
\end{enumerate}

\end{lem}
\begin{proof}
We first prove (1).
The second statement follows from the second statement of Proposition \ref{prop of ttrd}(3).  
For the first statement, we have a series of inequalities by
\begin{align*}
\ttrd(F_v/K_{v_0})
\leq &\  \ttrd(F_v/E_w)+\ttrd(E_w/K_{v_0})\\
\leq &\ \trd(F/E)+\trd(E/K)\\
=&\ \trd(F/K).
\end{align*}
Here the inequalities follow from Proposition \ref{prop of ttrd}(2)(3).
If $F_v/K_{v_0}$ is fully transcendental, then the equality of the first term and the last term forces all inequalities to be equalities. 

The proof of (2) is similar. In fact, 
the second statement follows from additivity of the transcendence degree.
For the first statement, 
we have 
\begin{align*}
\trd(F/K)
=&\ \trd(F/E)+\trd(E/K) \\
\geq&\ \trd(\wt F_v/\wt E_w)+\trd(\wt E_w/\wt K_{v_0})\\
=&\ \trd(\wt F_v/\wt K_{v_0}). 
\end{align*}
Here the inequality follows from Lemma \ref{trd compare}.
If $F_v/K_{v_0}$ is divisorial, then the equality of the first term and the last term forces all inequalities to be equalities. 
\end{proof}

We do not know whether the second statement of (1) holds if the characteristic of $\wt K_{v_0}$ is positive. 
However, we present the following special case, which will be sufficient for our future purpose. 

\begin{prop} \label{extension by divisorial}
Let $K$ be a non-archimedean field. 
Let $F/E/K$ be finitely generated field extensions with $\trd(E/K)=1$.
Let $v$ be a valuation of $F$ extending the valuation of $K$, and  $w=v|_E$ be the restriction to $E$.
If $F_v/E_w$ is divisorial and $E_w/K$ is fully transcendental, then $F_v/K$ is fully transcendental. 
\end{prop}
\begin{proof}
By assumption, $w$ is a point of type 2,3,4 on the Berkovich curve over $K$ corresponding to the extension $E/K$. 
If $w$ is of type 2, then it is divisorial, and thus $F_v/K$ is divisorial by Lemma \ref{sub of fully trans}(2). 
For general $w$, we hope to convert it to a divisorial valuation by a spherically complete base change. 

Denote $d=\trd(F/E)$. 
Take an element $x_0\in E$ transcendental over $K$, and take elements $x_1,\dots, x_d\in F$ algebraically independent over $E$. 
Replacing $E$ by $K(x_0)$ and $F$ by $K(x_0,\dots, x_d)$, and replacing $v$ and $w$ by the corresponding restrictions, we can assume that 
$E=K(x_0)$ and $F=K(x_0,\dots, x_d)$ are purely transcendental. 

Let $K'$ be a spherically complete and algebraically closed extension of $K$ with value group $\log |K'^\times|=\RR$. 
We refer to \cite[Thm. 1, Cor. 4, Prop. 9]{Poo93} for the existence of such a field.
Denote $E'=K'(x_0)$ and $F'=K'(x_0,\dots, x_d)$. 
View $F'/E'/K'$ as ``base change'' of $F/E/K$ by the extension $K'/K$. Moreover, all fields are subfields of $F'$. 

We claim that there is a valuation $v'$ of $F'=K'(x_0, \dots, x_d)$ extending the valuation of $K'$ and the valuation $v$ of $F=K(x_0,\dots, x_d)$ simultaneously. 
In fact, by induction, it suffices to prove that for any extension $L'/L$ of non-archimedean fields, the natural map $\CM(L'(t)/L')\to \CM(L(t)/L)$ is surjective. 
By \cite[Cor. 1.3.6]{Ber1}, we can prove that 
$\CM(\wh{\ol L}(t)/\wh{\ol L})\to \CM(L(t)/L)$ is surjective. 
Thus it suffices to prove that the natural map $\CM(\wh{\ol L'}(t)/\wh{\ol L'})\to \CM(\wh{\ol L}(t)/\wh{\ol L})$ is surjective. 
As in Example \ref{point types}, every
point of $\CM(\wh{\ol L}(t)/\wh{\ol L})$ is given by a decreasing sequence of discs in $\wh{\ol L}$ by \cite[\S1.4.4]{Ber1} (cf. \cite[\S1]{Bak}). Then we can lift the corresponding discs naturally from $\wh{\ol L}$ to $\wh{\ol L'}$. 

Now we have a valuation $v'$ of $F'$ extending the valuation of $K'$ and the valuation $v$ of $F$. 
Denote by $w'=v'|_{E'}$ the restriction to $E'=K'(x_0)$. 
Note that $w'$ is a point of type 2 over $K'$, since the other types can be excluded by the condition that $K'$ is spherically complete with a full value group $\log |K'^\times|=\RR$. 
It follows that $E'_{w'}/K'$ is divisorial. 

For the residue fields, we have
$$
\trd(\wt F'_{v'}/ \wt E'_{w'})
\geq \trd(\wt F_{v}/ \wt E_{w})
=d
=\trd(F'/E'). 
$$
It follows that $F'_{v'}/ E'_{w'}$ is divisorial. 
As $E'_{w'}/K'$ is divisorial, we see that $F'_{v'}/K'$ is also divisorial.

Therefore, we have 
$$
\ttrd(F_{v}/ K)
\geq \ttrd(F'_{v'}/K')
=d+1
=\trd(F/K). 
$$
This proves that 
$F_{v}/ K$ is fully transcendental. 
\end{proof}

The following result is a simple consequence by the previous results, but we state it here in order to clarify the difference between Theorem \ref{equi ft}
and Theorem \ref{equi ft2}.

\begin{lem}[almost equivalence] \label{equivalence}	
Let $F/k$ be a finitely generated field extension.
Let $v$ be a non-trivial valuation of $F$ whose restriction to $k$ is trivial. 
Consider the following two properties.
\begin{enumerate}[(a)]
\item
The valuation $v$ is fully transcendental over $k$.
\item
There is an intermediate field $K$ of $F/k$ with $\trd(K/k)=1$ such that $v|_K$ is non-trivial and $v$ is fully transcendental over $K$. 
\end{enumerate}
Then the following hold. 
\begin{enumerate}[(1)]
\item
We always have (a) implies (b). 
\item 
If furthermore the characteristic of $k$ is 0, then (b) implies (a).
\end{enumerate}
\end{lem}
\begin{proof}
To prove (1), assume that (a) holds, and take a nonzero $t\in F$ such that $|t|\neq 1$. Then $v$ is non-trivial over $K=k(t)$, which also implies that $K$ is not algebraic over $k$. Apply Lemma \ref{sub of fully trans}(1) to get (b).
Part (2) follows from Proposition \ref{prop of ttrd}(3).
\end{proof}

\subsubsection{Separability of completion}

Let $F$ be a field. 
A finitely generated extension $L$ of $F$ is \emph{separable} if it is finite separable over a purely transcendental extension of $F$. 
An extension $L$ of $F$ is \emph{separable} if every finitely generated subextension of $L$ is separable over $F$. 

An extension $L$ of $F$ is \emph{geometrically reduced} if, for every finite extension $F'/F$, the following equivalent conditions hold:
\begin{enumerate}[(1)]
\item the ring $F'\otimes_F L$ is reduced;
\item the ring $F'\otimes_F L$ is isomorphic to a direct product of finitely many fields.
\end{enumerate}
Note that these notions are automatic if $F$ has characteristic 0. 
If $F$ has characteristic $p>0$, it is known that an extension is {geometrically reduced} if and only if it is separable; and in the equivalent conditions, it suffices to take $F'=F^{1/p}$. 
See Stacks Project, Algebra, Lemma 10.44.2 and Definition 10.166.2.

Let $v$ be a non-archimedean valuation on a field $F$.
The extension $F_v/F$ is not necessarily separable. 
See \cite[Exam. 8.2.31]{Liu} for a counterexample, where $v$ is actually a discrete valuation. 
In this sense, the following result is surprising.

\begin{lem}\label{separable}
Let $F$ be a finitely generated extension over a perfect analytic field $K$.
Let $v$ be a fully transcendental valuation in $\CM(F/K)$. 
Then $F_v/F$ is a separable extension. 
\end{lem}

\begin{proof}
If $K$ has characteristic $0$, there is nothing to prove. Assume that $K$ has characteristic $p>0$.
We need to prove that $F'\otimes_F F_v$ is reduced for $F'=F^{1/p}$.
Denote $d=\trd(F/K)$, which is equal to $\ttrd(F_v/K)$ by definition. 
The valuation $v$ induces a unique valuation $v'$ on $F'$. The completion $F'_{v'}$ is a field, and the extension $F_v\to F'_{v'}$ is isomorphic to the extension $F_v\to F_v^{1/p}$. 
Consider the natural map $F'\otimes_F F_v \to F'_{v'}$, which is surjective since the image is dense and closed. It suffices to prove that this surjection is an isomorphism, and it suffices to prove 
$\dim_{F_v}(F'\otimes_F F_v)\leq \dim_{F_v} F'_{v'}$. 
In the following, we denote $L=F'_{v'}=F_v^{1/p}$. 

As $K$ is perfect, $F$ is separable over $K$, and $[F':F]=p^d$. 
We have $\dim_{F_v}(F'\otimes_F F_v)=\dim_F F'=p^d$. 
It suffices to prove $[L:F_v]\geq p^d$, or equivalently $[F_{v}^{1/p}:F_v]\geq p^d$. 
By \cite[Lem. 4.1.8(ii)]{Tem}, 
$$\mathrm{top.dim}_{L} \wh\Omega_{L/K}
=\mathrm{top.dim}_{F_{v}} \wh\Omega_{F_{v}/K} \geq d.$$
Note that there is an isomorphism $\wh\Omega_{L/K} \to \wh\Omega_{L/F_v}$. 
We have 
$$\mathrm{top.dim}_{L} \wh\Omega_{L/F_v} \geq d.$$
By definition, the $L$-vector space $\wh\Omega_{L/F_v}$ is the completion of the $L$-vector space
$\Omega_{L/F_v}$ under the K\"ahler semi-norm. 
Note that $[L:F_v]\leq \dim_{F_v}(F'\otimes_F F_v)=p^d$ is finite, so the completion process gives a surjection $\Omega_{L/F_v} \to \wh\Omega_{L/F_v}$. The surjection could fail to be injective if the K\"ahler semi-norm is not a norm. 
For our purpose, we always have 
$$\dim_{L}  \Omega_{L/F_v}
\geq 
\mathrm{top.dim}_{L} \wh\Omega_{L/F_v}
\geq d.$$
Finally, we have 
$$[L:F_v]= p^{\dim_{L} \Omega_{L/F_v}} \geq p^d.$$
This finishes the proof. 
\end{proof}

\subsubsection{Berkovich spaces}

We can easily extend the above notions to Berkovich spaces associated to quasi-projective schemes. Here we do not restrict to varieties, due to the fact that the base change of a variety by a field extension could fail to be integral.

Let $X$ be a equi-dimensional quasi-projective scheme over an analytic field $K$. Then a point $x\in (X/K)^\an$ is \emph{fully transcendental} over $K$ (or in $(X/K)^\an$) if $\ttrd(\CH_x/K)=\dim X$. 
The point $x$ is called \emph{divisorial} over $K$ (or in $(X/K)^\an$) if $\trd(\wt\CH_x/\wt K)=\dim X$. 
Denote by $((X/K)^\an)^\ft$ or just $(X^\an)^\ft$  the subset of fully transcendental points in $(X/K)^\an$.

Note that
 $\CH_x$ is the completion of the residue field $\kappa_x$ of $\kappa(x)\in X$, we have
$$
\ttrd(\CH_x/K) \leq \trd(\kappa_x/K) \leq \dim X.
$$
Hence, if $x$ is fully transcendental, then both inequalities are equalities, and thus $\kappa(x)$ must be the generic point of an irreducible component $X'$ of $X$ with $\dim X'=\dim X$. 
This holds similarly for divisorial points, and we also see that divisorial points are fully transcendental points.

Let $R$ be a commutative Banach integral domain.
Let $X\to Y$ be a flat quasi-projective morphism over $R$. 
Then a point $x\in (X/R)^\an$ is called \emph{fully transcendental} (resp. \emph{divisorial}) over $Y$ if $x$ is \emph{fully transcendental} (resp. \emph{divisorial})  in 
$(X/Y)^\an_y=(X_y/\CH_y)^\an$ over $\CH_y$, where $y$ is the image of $x$ in $(Y/R)^\an$.

The following theorem (cf. \cite[Prop. A.5, Prop. A.9]{GM}) gives an equivalent definition of divisorial points, which is also commonly seen in the literature. It also explains the name ``divisorial point''.

\begin{thm} \label{div pt equiv}
Let $K$ be a non-archimedean field with a discrete valuation.
Let $X$ be a normal quasi-projective variety over $K$.
Let $x\in X^\an$ be a point.
Then the following are equivalent.
\begin{enumerate}[(1)]
\item The point $x$ is divisorial.
\item There is a normal projective integral model $\CX$ of $X$ over $O_K$, 
and an irreducible component $D$ of the special fiber $\CX_{\wt K}$ of $\CX$, such that $|\cdot |_x$ is equal to a power of the valuation 
$$
|\cdot|_D: K(X)\lra \RR, \quad f\longmapsto e^{-\ord_D(f)}.
$$
\end{enumerate}
\end{thm}

The following lemma illustrates the compatibility between the two notions in the context of full transcendence.

\begin{lem} \label{prop of bijective map between ft points}
\kkk
If $k=\ZZ$, take $K=\QQ$ and take $v_0$ to be a non-archimedean place of $\QQ$. 
If $k$ is a field, let $K$ be a finitely generated field of {transcendence degree} 1 over $k$, and $v_0$ be a non-trivial valuation of $K$ over $k$.
Let $X$ be a  quasi-projective  variety over 
$K$ with a function field $F$.
Then there is a natural bijection
\[
\CM(F/k)_{v_0}^\ft 
\longrightarrow
 (X^\an_{v_0})^\ft.
\]
Moreover, the bijection induces a bijection between the subsets of divisorial points. 
\end{lem}
\begin{proof}
The morphism $\Spec F\to X$ induces an injection
$\CM(F/k) \to (X/k)^\an$. 
Taking fibers above $v_0$, we have an injection 
$\CM(F/k)_{v_0} \to X_{v_0}^\an$. 
It is clear that it induces an injection 
$\CM(F/k)_{v_0}^\ft \to (X^\an_{v_0})^\ft$. 
To see that it is surjective, it suffices to note that 
for any $x \in (X^\an_{v_0})^\ft$, the contraction point $\kappa(x)$, corresponding to the prime ideal $\ker(|\cdot|_x)$, is a generic point of $X_{v_0}$.
Here $X_{v_0}$ might be reducible, and thus contains finitely many generic points.
Then $x$ gives a valuation on the residue field of this generic point, and its restriction to $F$ gives a point of $\CM(F/k)_{v_0}^\ft$. 
\end{proof}

\section{Preliminaries on analysis on Berkovich spaces} \label{NA preparation section}

In this section, we collect some preliminary results for analysis on Berkovich spaces which will be used in our proof of the main theorem.

\subsection{Projection formula for integration}

In this subsection, we prove a projection formula for {integration} on Berkovich spaces. It is a direct analogue of a similar result from complex analytic spaces. 

Let $K$ be a non-archimedean field. Denote by $O_K$ the valuation ring of $K$.
Following the exposition of \cite[\S3.6]{YZ26}, for any quasi-projective variety $X$ over $K$,
there is a theory of adelic line bundles and adelic divisors of $X/O_K$.
There are  canonical injective analytification maps
$$
\wh\Div(X/O_K) \lra \wh\Div(X^\an),\qquad
\wh\Pic(X/O_K) \lra \wh\Pic(X^\an).
$$
The goal of this subsection is to prove the following result.

\begin{prop}[projection formula] \label{projection formula}
	Let $K$ be a non-archimedean field.
	Let $\pi:X\to Y$ be a quasi-projective and flat morphism of quasi-projective varieties over $K$.
	Denote $d=\dim Y$ and $n=\dim X-\dim Y$.
	Let $\OL_1,\dots, \OL_n$ be integrable adelic line bundles on $X/O_K$, and let
	$\OM_1,\dots, \OM_d$ be integrable adelic line bundles on $Y/O_K$.
	Then for any $f\in C_c(X^\an)$,
	\begin{align*}
	&\ \int_{X^\an} f\, c_1(\OL_1)\cdots c_1(\OL_n) c_1(\pi^*\OM_1)\cdots c_1(\pi^*\OM_d)\\
	=&\ \int_{Y^\an} \pi_*\big(f\, c_1(\OL_1)\cdots c_1(\OL_n)\big) c_1(\OM_1)\cdots c_1(\OM_d).
	\end{align*}
	Here the push-forward
	$$\pi_*\big(f\, c_1(\OL_1)\cdots c_1(\OL_n)\big)
	=\int_{X^\an/Y^\an} f\, c_1(\OL_1)\cdots c_1(\OL_n)
	$$
	is the function on $Y^\an$ which sends
	$y\in Y^\an$ to
	$$
	\int_{X_y^\an} f\, c_1(\OL_1|_{X_y})\cdots c_1(\OL_n|_{X_y}),
	$$
		where $X_y=X\times_Y \Spec \CH_y$ is the base change to the valuation field $\CH_y$ of $y\in Y^\an$.
\end{prop}

\begin{proof}
We illustrate the proof by three cases: model case, projective case, and quasi-projective. The model case essentially follows from a projection formula of the Deligne pairing, and the other two cases are its limit cases.
	
\medskip\noindent \emph{Case 1. model case.} 	Assume that $\pi:X\to Y$ extends to a projective and flat morphism $\pi:\CX\to \CY$ of
	projective $O_K$-models $\CX, \CY$ of $X, Y$, such that the adelic line bundles
	$\OL_1,\dots, \OL_n$  on $X/O_K$ and
	$\OM_1,\dots, \OM_d$ on $Y/O_K$
	are induced by line bundles
	$\CL_1,\dots, \CL_n$  on $\CX$ and
	$\CM_1,\dots, \CM_d$ on $\CY$.
	Assume further that $f\in C_c(X^\an)$ is induced by a vertical line bundle $\CN$ on $\CX$. More precisely, $\CN$ is a line bundle on $\CX$ with generic fiber $\CN_K\simeq \CO_{\CX_K}$ on $\CX_K$, and
	$f=-\log \|1\|_\CN$ via the metric $\|\cdot\|_\CN$ on the trivial bundle $\CO_X$ induced by $\CN$.
	
	In this case, the projection formula of the Deligne pairing gives a canonical isomorphism
	\begin{align*}
	&\ (\pi_\CX)_* \big\langle \CN,  \CL_1,\dots, \CL_n, \pi^*\CM_1,\dots, \pi^*\CM_d\big\rangle
	\\
	\lra &\ (\pi_\CY)_* \big\langle  \pi_*\pair{\CN,  \CL_1,\dots, \CL_n}, \CM_1,\dots, \CM_d\big\rangle.
	\end{align*}
	Here $\pi_\CX: \CX\to \Spec O_K$ and $\pi_\CY: \CY\to \Spec O_K$
	denote the structure morphisms.
	If $O_K$ is a discrete valuation ring, the projection formula is a special case of
	\cite[Prop. 5.2.3.b]{MG}.
	If $O_K$ is not discrete, then it is not noetherian, but the result can still be deduced from the noetherian case by Xia \cite[Prop. 3.7]{Xia}.
	
	Now let us unravel the information in terms of integration.
	For convenience, we will write the projection formula as $\CN_1\to \CN_2$, where $\CN_1$ and $ \CN_2$ are line bundles over $\Spec O_K$ denoting the two sides of the formula.
	By the fixed trivialization $\CN_K\simeq \CO_X$, the base changes $\CN_{1,K}$ and $\CN_{2,K}$ are canonically isomorphic to the trivial line bundle on $\Spec K$. Therefore, they induce metrics $\|\cdot\|_{\CN_1}$ and $\|\cdot\|_{\CN_2}$ on $K$.
	By \cite[Thm. 4.6.2]{YZ26}, we can compute the metric $\|1\|_{\CN_1}$ in terms of the metric $\|1\|_{\CN}$. It gives
	$$
	-\log\|1\|_{\CN_1}=\int_{X^\an} (-\log \|1\|_\CN)  c_1(\OL_1)\cdots c_1(\OL_n) c_1(\pi^*\OM_1)\cdots c_1(\pi^*\OM_d)
	$$
	and
	$$
	-\log\|1\|_{\CN_2}=\int_{Y^\an} h\, c_1(\OM_1)\cdots c_1(\OM_d).
	$$
	Here the function $h=-\log\|1\|_{\CN_3}$ on $Y^\an$ is induced by the line bundle $\CN_3= \pi_*\pair{\CN,  \CL_1,\dots, \CL_n}$ on $\CY$.
		By the compatibility of the Deligne pairing with the base change by $\Spec O_{\CH_y}\to Y$, we see that $h=\pi_*\big(f\, c_1(\OL_1)\cdots c_1(\OL_n)\big)$ on $Y^\an$.
	Then the model case follows from the equality $\|\cdot\|_{\CN_1}=\|\cdot\|_{\CN_2}$.
	
\medskip\noindent \emph{Case 2. projective case.} 
We assume that $\pi:X\to Y$ is projective and flat in the theorem.
 In this case, let $\ON$ be an adelic divisor on $X/O_K$ with trivial underlying divisor $N=0$.
	Then the projection formula of  the Deligne pairing still gives a canonical isomorphism
	\begin{align*}
	&\ (\pi_X)_* \big\langle \ON,  \OL_1,\dots, \OL_n, \pi^*\OM_1,\dots, \pi^*\OM_d\big\rangle
	\\
	\lra &\ (\pi_Y)_* \big\langle  \pi_*\pair{\ON,  \OL_1,\dots, \OL_n}, \OM_1,\dots, \OM_d\big\rangle.
	\end{align*}
	Here $\pi_X: X\to \Spec K$ and $\pi_Y: Y\to \Spec K$
	denote the structure morphisms.
	Here the Deligne pairing is defined similarly to that in \cite[Thm. 4.1.3]{YZ26},
	and the projection formula is the analogue of \cite[Lem. 4.6.1(1)]{YZ26}.
The treatment of \cite{YZ26} is based on the assumption that $Y$ is normal, but this can be achieved by replacing $Y$ by an open subvariety.
	
	By \cite[Thm. 4.6.2]{YZ26}, the equality of the metrics gives
	\begin{align*}
	&\ \int_{X^\an} f\, c_1(\OL_1)\cdots c_1(\OL_n) c_1(\pi^*\OM_1)\cdots c_1(\pi^*\OM_d)\\
	=&\ \int_{Y^\an} \pi_*\big(f\, c_1(\OL_1)\cdots c_1(\OL_n)\big) c_1(\OM_1)\cdots c_1(\OM_d).
	\end{align*}
	Here $f=-\log\|1\|_\ON$ is the function on $X^\an$ induced by $\ON$.
	To see that such functions can uniformly approximate any function in $C_c(X^\an)$, it suffices to take a projective model $X'$ of $X$ over $K$, and assume that $\ON$ is induced by a vertical line bundle on a projective model of $X'$ over $O_K$.
	Such functions are dense in $C(X'^\an)$ by Gubler \cite[Thm. 7.12]{Gub98} (cf. \cite[p. 643]{Yua08}).

\medskip\noindent \emph{Case 3. quasi-projective case.} 
In this part, we deduce the general case from the projective case via a limiting process.
By linearity, we may assume that $\OL=\OL_1=\cdots=\OL_n$ and $\OM=\OM_1=\cdots=\OM_d$. Write
\[
\OL=\lim_{i\to\infty}(X_i,\OL_i,\ell_i)
\]
as a limit of adelic line bundles on projective models $X_i$. Here by blowing up, we may assume that there is a projective (but not necessarily flat) morphism $\pi_i:X_i\to Y$ extending $\pi:X\to Y$. Then
\[
\int_{X^\an} f\, c_1(\OL)^n c_1(\pi^*\OM)^d=\lim_{i\to\infty} \int_{X_i^\an} f\, c_1(\OL_i)^n c_1(\pi^*\OM)^d.
\]
By generic flatness, there is a Zariski open subset $U_i\subset Y$ such that $\pi_i:X'_i\to U_i$ is flat. Here $X_i'=X_i\times_Y U_i$. Then by the projective case, we have
\[
\int_{X_{i}'^\an} f\, c_1(\OL_i)^n c_1(\pi^*\OM)^d=
\int_{U_i^\an} \pi_{i,*}\big(f\, c_1(\OL_i)^n\big) c_1(\OM)^d.
\]
Then it suffices to prove that
\[
\int_{X_{i}'^\an} f\, c_1(\OL_i)^n c_1(\pi^*\OM)^d=\int_{X_i^\an} f\, c_1(\OL_i)^n c_1(\pi^*\OM)^d
\]
and
\[
\int_{U_i^\an} \pi_{i,*}\big(f\, c_1(\OL_i)^n\big) c_1(\OM)^d=\int_{Y  ^\an} \pi_*\big(f\, c_1(\OL_i)^n\big) c_1(\OM)^d.
\]
For these two equalities, it suffices to prove a general result as follows. Let $W$ be a variety over $K$ of dimension $m$, and let $\OH\in\HPic(W/O_K)_\intt$. Then for a Zariski open subset $V\subset W$ and a compactly supported function $g$ on $W^\an$, we have
\[
\int_{W^\an}gc_1(\OH)^m=\int_{V^\an}gc_1(\OH)^m.
\]
In fact, by multilinearity, we may first assume that $\OH$ is nef, and thus $c_1(\OH)^m$ is a positive measure.  The equality
\[
\int_{W^\an}c_1(\OH)^m=\int_{V^\an}c_1(\OH)^m=\TH^m,
\]
proved in \cite[Thm. 1.2]{Guo25}, shows that $W^\an\setminus V^\an$ has measure zero.  Since $g$ is bounded, the desired equality follows.  
\end{proof}

An easy consequence of the proposition is as follows.

\begin{cor} \label{projection formula Dirac}
	Let $K$ be a non-archimedean field.
	Let $\pi:X\to Y$ be a quasi-projective and flat morphism of quasi-projective varieties over $K$.
	Denote $d=\dim Y$ and $n=\dim X-\dim Y$.
		Let $\OL_1,\dots, \OL_n$ be {integrable} adelic line bundles on $X/O_K$, and let
		$\OM_1,\dots, \OM_d$ be {integrable} adelic line bundles on $Y/O_K$.
	Assume that the Chambert-Loir measure
	$$
	c_1(\OM_1)\cdots c_1(\OM_d)=\sum_{i=1}^r a_i\delta_{y_i}
	$$
	for $y_1, \dots, y_r\in Y^\an$ and $a_1,\dots, a_r\in \RR$.
	Then the  Chambert-Loir measure
	\begin{align*}
	&\  c_1(\OL_1)\cdots c_1(\OL_n) c_1(\pi^*\OM_1)\cdots c_1(\pi^*\OM_d)\\
	=&\
	\sum_{i=1}^r a_i\ \rho_{i,*}
		\big( c_1(\OL_1|_{X_{y_i}})\cdots c_1(\OL_n|_{X_{y_i}})\big).
	\end{align*}
	Here $\rho_i: X_{y_i}^\an\to X^\an$ denotes the natural injection.
\end{cor}

\subsection{Continuity of fiberwise integration on the base}

In this subsection, we prove a continuity result for integrals on a family of Berkovich spaces. It is also a direct analogue of a similar result from complex analytic spaces.

We first recall the compatibility of two Deligne pairings proved in 
\cite[\S4.6.2]{YZ26}, and then prove a technical result on continuity of fiberwise integrals on Berkovich spaces.

Let $K$ be a non-archimedean field. Let $\pi:X\to Y$ be a projective and flat morphism of quasi-projective varieties over $K$. Denote $n=\dim X-\dim Y$.
Denote
$$
\widehat\Picc(X^\an)_\mathrm{mod}=\mathrm{Im}\left(\widehat\Picc(X/O_K)_\mathrm{mod}\longrightarrow\widehat\Picc(X^\an)\right),
$$
$$
\widehat\Picc(X^\an)_\intt=\mathrm{Im}\left(\widehat\Picc(X/O_K)_{\intt}\longrightarrow\widehat\Picc(X^\an)\right),
$$
where the arrows are the canonical injective analytification functors introduced in \cite[\S3.6]{YZ26}. Parallel to \cite[\S4.6.2]{YZ26}, we have two Deligne pairings
$$
\widehat\Picc(X/O_K)_{\intt}^{n+1}\longrightarrow\widehat\Picc(Y/O_K)_\intt
$$
and
$$
\widehat\Picc(X^\an)_\intt^{n+1}\longrightarrow\widehat\Picc(Y^\an)_\intt.
$$
The cited work actually introduces these two pairings only in the case that $Y=K$. However, the general case is similar and we sketch it in the following.

To define the first pairing, by taking limits, it suffices to construct a pairing
$$
\widehat\Picc(X/O_K)_\mathrm{mod}^{n+1}\longrightarrow\widehat\Picc(Y/O_K)_\mathrm{mod}.
$$
Let $\OL_1,\cdots,\OL_{n+1}\in\widehat\Picc(X/O_K)_\mathrm{mod}$ be model adelic line bundles, which are induced by line bundles $\CL_1,\cdots,\CL_{n+1}$ on a single projective model $\CX$ of $X$ over $O_K$. By blowing up $\CX$ along a center in the special fiber, we may extend $\pi$ to a projective and flat morphism $\pi:\CX\to\CY$, where $\CY$ is a projective model of $Y$.
Then we define the Deligne pairing
$\left\langle\OL_1,\cdots,\OL_{n+1}\right\rangle_{X/Y}$
to be 
the metrized line bundle induced by
$\left\langle\CL_1,\cdots,\CL_{n+1}\right\rangle_{\CX/\CY}.$

To define the second pairing, the quickest way is to transfer the first pairing to the Berkovich spaces via the analytification functor. 
We take this as the definition of the second pairing, and will consider the fiberwise metric of the pairing parallel to the complex case in \cite[\S4.2.2]{YZ26}. 

Let  $\OL_1,\cdots, \OL_{n+1}\in \widehat\Picc(X^\an)_\intt$ be integrable metrized line bundles. For a point $y\in Y^\an$, denote $\OL_{i,y}=\OL_i|_{X_y^\an}$. 
By \cite[\S4.6.2]{YZ26}, we already have the following Deligne pairing 
$$
\widehat\Picc(X_y^\an)_\intt^{n+1}\longrightarrow
\widehat\Picc(\CH_y^\an)_\intt,\quad
(\OL_{1,y},\cdots,\OL_{n+1,y})\longmapsto\left\langle\OL_{1,y},\cdots,\OL_{n+1,y}\right\rangle_{X_y^\an/\CH_y^\an}.
$$
Here by abuse of notation, we write $\widehat\Picc(\CH_y^\an)_\intt$ for $\widehat\Picc((\Spec \CH_y)^\an)_\intt$.
The metric of the Deligne pairing is defined in terms of fiberwise integration. 
For our purpose, we only need to know that for $L_1$  trivial, the metric of the meromorphic section $1$ at $y$ is given by the  integral
$$
\int_{X_y^\an}\left(-\log\|1\|_{\OL_1}\right)c_1(\OL_2)\cdots c_1(\OL_{n+1}).
$$

Now we claim that there is a canonical isometry 
$$
\left\langle\OL_{1,y},\cdots,\OL_{n+1,y}\right\rangle_{X_y/\CH_y}
\cong
\left.\left(\left\langle\OL_{1},\cdots,\OL_{n+1}\right\rangle_{X/Y}\right)\right|_{\CH_y}.
$$
To prove this, we convert it to the corresponding isomorphism for integrable adelic line bundles, and then we realize that it is given by the functoriality of the first Deligne pairing under the base change
$$
\Spec \CH_y\longrightarrow Y.
$$ 

Now we state the main lemma of this subsection.

\begin{lemma}\label{continuity lem}
	Let $K$ be a non-archimedean field.
	Let $\pi:X\to Y$ be a flat morphism of quasi-projective varieties over $K$.
	Assume $\dim Y=1$ and denote $n=\dim X-\dim Y$.
	Assume $\OL_1,\dots, \OL_n\in\HPic(X/O_K)_\intt$.
	Let $f$ be a continuous function on $X^\an$ which is either constant or compactly supported. Define a function
	$$
	I_f:Y^\an\longrightarrow\RR,\quad y\longmapsto
\int_{X_y^\an} f\, c_1(\OL_1|_{X_y})\cdots c_1(\OL_n|_{X_y}),
	$$
	where $X_y=X\times_Y \Spec \CH_y$ is the base change to the {completed residue field} $\CH_y$ of $y\in Y^\an$.
	Then the function
	$$
	I_f:Y^\an\longrightarrow\RR
	$$
	is continuous on $Y^\an$.
\end{lemma}

\begin{proof}
		By linearity, we may assume that $\OL=\OL_1=\cdots=\OL_n$ is strongly nef. Take a boundary divisor $(X_0,\OD_0)$. Write $\OL$ as a limit 
	\[
	\OL=\lim_{k\to\infty}(X_k,\OL_k,\ell_k),
	\]
	where 
	\[
	-\epsilon_i\OD_0 \leq \wh\div_{\OL_i\otimes\OL_j^{-1}}(1)\leq \epsilon_i\OD_0,\ \forall j>i\geq1.
	\]
	Here $\lim_{k\to\infty}\epsilon_k=0$. In addition, by blowing-up, we assume that there is a projective morphism $\pi_k:X_k\to Y$.
	Take a sequence of model functions $\{f_k\}_{k\geq1}$ on $X_0$ such that 
	\[
	\sup_{x\in X_0^\an}|f(x)-f_k(x)|<\epsilon_k.
	\]
	Take a bound $M$ for all $f$ and $f_k$. Then
	\[
	\left|
	\int_{X_y^\an} f\, c_1(\OL|_{X_y})^n-
	\int_{X_{k,y}^\an} f\, c_1(\OL_k|_{X_{k,y}})^n
	\right|
	\leq 
	\epsilon_k\cdot \OL_k|_{X_{k,y}}^n+M\left(\OL|_{X_{y}}^n-\OL_k|_{X_{k,y}}^n\right)
	\]
	
	Take the Zariski closure of $X$ in $\PP^{N}_Y$, which is denoted by $X'$. Denote the projection $\pi':X'\to Y$. Then $\pi'$ is projective and flat. Fix a projective model $X''$ of $X'$, which means $X''$ is a projective variety over $K$ containing $X'$ as a Zariski open subset. By \cite[Thm. 1.2]{Guo25}, 
	$$
	\int_{X_y^\an} c_1(\OL_1|_{X_y})\cdots c_1(\OL_n|_{X_y})=\TL_1|_{X_y}\cdots\TL_n|_{X_y}.
	$$
	Since $\pi$ is flat, the intersection numbers on the right hand side are equal for all $y\in Y^\an$. Take a sequence of model functions $\{f_n\}_{n\geq1}$ on $X''$ which approximates $f$ uniformly on $(X'')^\an$. Then the integration 
	$$
	\int_{X_y^\an} f_n\, c_1(\OL_1|_{X_y})\cdots c_1(\OL_n|_{X_y})\longrightarrow 
	\int_{X_y^\an} f\, c_1(\OL_1|_{X_y})\cdots c_1(\OL_n|_{X_y})
	$$
	uniformly for $y\in Y^\an$. It suffices to prove the continuity of $I$ assuming that $f$ is a model function on $(X'')^\an$.
	
	Take $\OL_0=(\CO_X,\|\cdot\|_f)\in\HPic(X^\an)$ defined by $\|1\|_f=\mathrm{e}^{-f}$. Then the metric of $\OL_0$ is a model metric. Denote
	$$
	\OM=(M,\|\cdot\|_\OM)=\left\langle\OL_{0},\cdots,\OL_{n}\right\rangle_{X^\an/Y^\an}.
	$$
	Then for the metric of the meromorphic section $1$ at $y$,
	$$
	\|1(y)\|_\OM=\int_{X_y^\an} f\, c_1(\OL_1|_{X_y})\cdots c_1(\OL_n|_{X_y}).
	$$
	Since the metric of $\left\langle\OL_{0},\cdots,\OL_{n}\right\rangle_{X^\an/Y^\an}$ on $Y^\an$ is integrable, hence continuous, we deduce the continuity of $I(y)$.
\end{proof}

\subsection{Strictly positive differential forms on curves}

In \cite{CLD}, Chambert-Loir and Ducros introduced a theory of smooth differential forms on Berkovich spaces over non-archimedean fields. 
This notion was generalized to weakly smooth differential forms, by Gubler--Jell--Rabinoff  \cite{GJR1}.
In \cite{GJR2}, Gubler--Jell--Rabinoff developed a theory of smooth forms on metric graphs, which can be transferred to Berkovich curves. The goal of this section is to review some results of 
\cite{GJR1,GJR2}, with an emphasis on integrals of differential forms.

\subsubsection{Differential forms on metric graphs}
 
Let $\Gamma$ be a metric graph (without {loops}) as in 
\cite[Def. 2.1.1]{GJR1}.
Then $\Gamma$ is a union of finitely many line segments, glued along their endpoints in some way.
Denote by $V(\Gamma)$ the set of vertices of $\Gamma$, and by $E(\Gamma)$ the set of (unoriented) edges of $\Gamma$.

In our paper, we always assume a metric graph has no weights everywhere, or we just think the weights are 1 for all edges. 
To connect to the notation of \cite{GJR2}, our metric graphs correspond to the unweighting of weighted metric graphs in \cite[Def. 2.1.2]{GJR2}.
 We do not consider boundaries as in \cite[Def. 2.1.4]{GJR2}, since they do not appear for smooth projective curves over non-archimedean fields.

We briefly recall the notions of smooth functions and smooth $(1,1)$-forms in this setting.
For every edge $e$ of $\Gamma$, denote by $l(e)$ the length of $e$. 
There is a \emph{coordinate function} (or \emph{parametrization}) $t_e$ on $e$, i.e., an isometry
$$
t_e : [0, l(e)] \longrightarrow e,
$$
There are exactly two choices of the coordinates, and each one is determined by a direction (or orientation) of $e$.
In the following, we will only specify the direction of $e$ if a statement depends on the choice of $t_e$.

A \emph{smooth function} (i.e. smooth $(0,0)$-form) on $\Gamma$ is a continuous function $f:\Gamma\to\RR$ which is smooth on every edge in the sense that $f_e=f\circ t_e:[0,l(e)]\to\RR$ is smooth, satisfying the following properties:
\begin{enumerate}[(a)]
	\item
	if $v\in \Gamma$ is a vertex of valency $1$ with outgoing edge $e$, then $f_e$ is constant in a neighborhood of $v$;
	\item
	if $v\in \Gamma$ is a vertex of valency $2$ with outgoing edges $e_1,e_2$, then for any $n\geq 0$,
	$$
	\left.\frac{\d^n f_{e_1}}{\d t_{e_1}^n}\right|_{t_{e_1}=0}
	=(-1)^n\left.\frac{\d^n f_{e_2}}{\d t_{e_2}^n}\right|_{t_{e_2}=0};
	$$	
	\item
	if $v\in \Gamma$ is a vertex of valency $r\geq3$ with outgoing edges $e_1,e_2,\cdots ,e_r$, then
	$$
		\sum_{i=1}^r\left.\frac{\d f_{e_i}}{\d t_{e_i}}\right|_{t_{e_i}=0}=0.
	$$
\end{enumerate}

A \emph{smooth $(1,1)$-form} on $\Gamma$ is a collection of data 
$\omega=(f_e\d' t_e\d'' t_e)_{e\in E(\Gamma)}$, where $f_e$ is a smooth function on $e$ for every edge $e$, satisfying the following properties:
\begin{enumerate}[(a)]
	\item
	if $v\in \Gamma$ is a vertex of valency $1$ with outgoing edge $e$, then $f_e=0$ in a neighborhood of $v$;
	\item
	if $v\in \Gamma$ is a vertex of valency $2$ with outgoing edges $e_1,e_2$, then for any $n\geq 0$,
	$$
\left.\frac{\d^n f_{e_1}}{\d t_{e_1}^n}\right|_{t_{e_1}=0}
	=(-1)^n\left.\frac{\d^n f_{e_2}}{\d t_{e_2}^n}\right|_{t_{e_2}=0}.
	$$
\end{enumerate}
For details, we refer to \cite[\S 2]{GJR2}.

For our purposes, for a collection of data $\omega=(f_e\d' t_e\d'' t_e)_{e\in E(\Gamma)}$ such that $f_e$ is smooth on $e$ for every edge $e$, we call it a \emph{piecewise smooth $(1,1)$-form} on $\Gamma$. For a piecewise smooth $(1,1)$-form $\omega$ satisfying condition $(b)$ above, we call it a \emph{piecewise smooth $(1,1)$-form with compatibility at vertices of valency $2$}. 

For a piecewise smooth $(1,1)$-form $\omega$, we say that $\omega$ is \emph{strictly positive}, if there exists a constant $c>0$ such that $f_e\geq c$ on every edge $e$. 

For a piecewise smooth $(1,1)$-form $\omega$, we define the \emph{integral} of $\omega$ as the sum of the integrals of $\omega$ on every edge. More precisely,
\[
\int_\Gamma\omega=\sum_{e\in E(\Gamma)}\int_e f_e\d' t_e\d'' t_e=\sum_{e\in E(\Gamma)}\int_{[0,l(e)]}f_e\d t.
\]
Here $\d t$ is the Lebesgue measure on $[0,l(e)]$.
This is the same as the integration in \cite{GJR2}.

\begin{remark}
	Note that because of condition $(a)$, a smooth $(1,1)$-form cannot be strictly positive if $\Gamma$ has a vertex of valency 1.
However, there is always a strictly positive piecewise smooth $(1,1)$-form on $\Gamma$.
\end{remark}

 \subsubsection{Differential forms on curves}

Let $K$ be a non-archimedean field.
Let $X$ be a smooth quasi-projective curve over $K$.
By an \emph{integral model} $\CX$ of $X$, we mean an integral scheme projective and flat over $O_K$ with an open immersion $X\to \CX_K$.
By a \emph{strictly semi-stable model} of $X$, we mean a regular integral model $\CX$ of $X$ whose special fiber has smooth irreducible components and has at worst ordinary double points as singularities.

Assume that $X$ has a strictly semi-stable model $\CX$, which can be done after replacing $K$ by its finite extension.
Denote by $S(\CX)$ the skeleton of $\CX$, which is the reduction graph of $\CX$.
By \cite[\S4]{Ber99} or \cite[\S5.3]{GJR2}, we have a canonical embedding $i_{\CX}:S(\CX)\to X^\an$ and a canonical contraction map $\tau_\CX: X^\an\to S(\CX)$. 
It is known that for each singular point $x\in\CX$, there is an open neighborhood $\CU$ of $x$, whose special fiber has exactly two irreducible components, and an \'etale morphism
\[
p:\CU\longrightarrow\Spec O_K[x,y]/(xy-a),\quad a\in O_K\setminus\{0\}.
\]
See \cite[Rmk. 2.2.9]{Thu} for this result. 
Then $S(\CX)$ can be endowed with a metric as follows. For an edge $e\in S(\CX)$, which corresponds to a singular point $\tilde x$ in the special fiber $\CX_s$, the length is defined as
$$
l(e)=\frac{1}{[\wt K(\tilde x):\wt K]}(-\log|a|).
$$

Suppose that $X$ is a smooth projective curve over $K$.
In \cite{GJR1}, Gubler--Jell--Rabinoff defined the set of \emph{weakly smooth $(p,q)$-forms} $A^{p,q}(X^\an)$ on $X^\an$, which refines the theory of smooth differential forms of Chambert-Loir--Ducros \cite{CLD}.
Denote by $A^{p,q}(S(\CX))$ the set of smooth $(p,q)$-forms on $S(\CX)$.
The following proposition is from \cite[\S 6.1]{GJR2}.

\begin{proposition}\label{pullback from skeleton}
	There is a natural pull-back $\tau_\CX^*:A^{p,q}(S(\CX))\to A^{p,q}(X^\an)$.
\end{proposition}

By the above proposition, for an element $\omega\in A^{1,1}(S(\CX))$, we may {define} an \emph{integral on $X^\an$} by
\[
\int_{X^\an}\tau_\CX^*\omega:=\int_{S(\CX)}\omega.
\]

\begin{proposition} \label{formula of laplace on graph}
	For any continuous function $f:S(\CX)\to\RR$ which is smooth on edges, denote by $\mathcal{O}_X(f)$ the trivial line bundle $\mathcal{O}_X$ endowed with the metric over $X^\an$ given by $\|1\|=e^{-\tau_\CX^* f}$. Then
	$$
	c_1\left(\mathcal{O}_X(f)\right)=-i_{\CX,*}(\Delta f)
	$$
	as measures over $X^\an$. Here
	$$
		\Delta f=-(f_e''\d' t_e\d'' t_e)_{e\in E(S(\CX))}-\sum_{v \in V(S(\CX))} \left(\sum_{e\textnormal{ outgoing} \atop\textnormal{edge of $v$}} \left.\frac{\d f}{\d t_e}\right|_{t_e=0}\right)\delta_v.
	$$
\end{proposition}

\begin{proof}
It is a standard result. For a proof, we refer to \cite[Prop. A.4]{Yua21}.
\end{proof}

In the theorem, the piecewise smooth $(1,1)$-form $(f_e\d' t_e\d'' t_e)_{e\in E(S(\CX))}$ is viewed as a measure on $S(\CX)$ by integration as above.

 \subsubsection{Monge--Amp\`ere equations on curves}

The first goal of this subsection is to prove the following result for a simple Monge--Amp\`ere equation. 

\begin{thm} \label{MA for measure on skeleton}
Let $K$ be a non-archimedean field. 
	Let $X$ be a smooth projective curve over $K$ with a strictly semi-stable model $\CX$.
	Let $L$ be a line bundle on $X$ with $\deg(L)>0$.
	Then for any smooth $(1,1)$-form $\mu$ on $S(\CX)$ satisfying
	$$
	\int_{S(\CX)}\mu=\deg(L),
	$$	
	there exists a weakly smooth metric $\|\cdot\|$ on $L$ such that
	$$
	c_1(L,\|\cdot\|)=i_{\CX,*}\mu.
	$$
	In addition, this metric is unique up to a multiplicative constant.
\end{thm}

\begin{proof}
	Denote $\Gamma=S(\CX)$, which is a metrized graph. Denote by $i=i_{\CX}:S(\CX) \to X^\an$ the injection.
	
	Firstly, we prove the existence. Write $\mu=(\mu_e\d' t_e\d'' t_e)_{e\in E(\Gamma)}$.	
	Take a continuous function $f$ on $S(\CX)$ such that
	
	\begin{enumerate}[(1)]
		\item
		$f$ is smooth on every edge $e\in E(\Gamma)$;
		
		\item
		$f$ is constant in a neighborhood of every vertex of valency $1$ in $\Gamma$;
		
		\item
		$f_e''=\mu_e$ for each $e\in E(\Gamma)$.
		
	\end{enumerate}
	This can be achieved by defining $f$ inductively on every edge. Since there is no condition on $f_e'$, we may modify $f_e$ by adding a polynomial of degree at most one so that the resulting functions glue together continuously.

	By Proposition \ref{formula of laplace on graph},
	$$
	c_1(\CO_X(f))=-i_*(\Delta f)=i_*\left(\mu+\sum_{v\in V(\Gamma)}c_v\delta_v\right),
	$$
	where  $c_v$ are constants. Note that for a vertex $v'$ of valency $1$ in $\Gamma'$, $c_{v'}=0$. In addition,
	$$
	\sum_{v\in V(\Gamma)}c_v=-\int_{S(\CX)}\mu=-\deg(L).
	$$
	By the local Hodge index theorem \cite[Thm. 2.1]{YZ17}, there is a  model $\CL_0\in\Pic(\CX)\otimes_\ZZ\RR$ on $\CX$ of $L$ inducing a model metric $\|\cdot\|_0$ of $L$ such that
	$$
	c_1(L,\|\cdot\|_0)=i_*\left(\sum_{v\in V(\Gamma)}-c_v\delta_v\right).
	$$
		To be specific, we can take an arbitrary model $\CL_0'$ of {$L$} and then modify the metric by a model function. 
		
	By twisting with $f$, the metric defined by $\|\cdot\|=\|\cdot\|_0e^{-\tau_\CX^*f}$ is a continuous metric of $L$. Then
	$$
	c_1(L,\|\cdot\|)=i_*\mu.
	$$
	In the following, we prove that the metric $\|\cdot\|$ of $L$ is weakly smooth.
	
		For an arbitrary point of $X^\an$, take an open affine neighborhood $\CU$ of its reduction in $\CX$ trivializing $\CL_0$. Denote by $U$ the generic fiber of $\CU$. Then $\CU$ is a strictly semistable model of $U^\an$ in the sense of \cite[Def. 5.2.3]{GJR2}. Take a local frame $s$ of $\CL_0$ on $\CU$. We still denote by $i$ the inclusion $S(\CU)\emb U^\an$. It suffices to prove that $-\log\|s\|_0+f\circ\tau_\CU$ is weakly smooth on $U^\an$. By Proposition \ref{pullback from skeleton}, it suffices to prove that $h:=-\log\|s\|_0+f|_{S(\CU)}$ is smooth on $S(\CU)$. Now we check conditions of smoothness. Firstly, $h$ is smooth on every edge. For a vertex of valency $1$ with outgoing edge $e$, $f$ is constant in a neighborhood of $v$. Since
	$$
		\left.\frac{\d (-\log\|s\|_0)}{\d t_e}\right|_{t_e=0}=c_v=0,
	$$
		the function $-\log\|s\|_0$ is also constant near $v$. Thus so is $h$. For each vertex of valency at least $2$ with outgoing edges $e_1,\cdots,e_r$, since
	$$
-\Delta(h)=\mu|_{S(\CU)}
	$$
	is a smooth $(1,1)$-form, we have
	$$
	\sum_{i=1}^r \left.\frac{d h}{d t_{e_i}}\right|_{t_{e_i}=0}=0,
	$$
		which implies condition $(c)$ of the smoothness. Finally, we only need to check the compatibility of $h_e^{(k)}$ at a vertex of valency $2$ for $k\geq2$. However, $-\log\|s\|_0$ is piecewise linear. Thus $h_e^{(k)}=f_e^{(k)}$ for $k\geq2$. The compatibility arises from the smoothness of $\mu=(f_e''\d' t_e\d'' t_e)_{e\in E(\Gamma')}$. This completes the proof of existence.
	
	For the uniqueness, suppose there are two weakly smooth metrics $\|\cdot\|_1$ and $\|\cdot\|_2$ satisfying the condition. Denote
	$$
	\overline{\CO}_X=(\CO_X,\|\cdot\|_\mathrm{con})=(L,\|\cdot\|_1)\otimes(L,\|\cdot\|_2)^\vee.
	$$
	Then
	$$
	c_1(\overline{\CO}_X)=\d'\d''(-\log\|1\|_\mathrm{con})=0.
	$$
	By \cite[Prop. 10.17]{GJR1},
	$$
	h_0:=-\log\|1\|_\mathrm{con}: X^\an\to\RR
	$$
		is a harmonic function on $X^\an$. 
Since $X^\an$ is compact and connected, the maximum principle for harmonic functions shows that $h_0$ is constant.
\end{proof}

A quick consequence of the theorem is the following positivity result, which will be used in our proof of the equidistribution theorem. 

\begin{corollary} \label{strictly positive measure on skeleton}
Let $K$ be a non-archimedean field. 
	Let $X$ be a smooth projective curve over $K$ with a strictly semi-stable model $\CX$ over $O_K$.
	Let $L$ be a line bundle on $X$ with $\deg(L)>0$. 
Then for any strictly positive piecewise smooth $(1,1)$-form $\mu$ on $S(\CX)$ satisfying
	\[
	\int_{S(\CX)}\mu\leq \deg(L),
	\]
	there exists a weakly smooth metric $\|\cdot\|$ of $L$ on $X^\an$ such that $c_1(L,\|\cdot\|)-i_{\CX,*}\mu$ is a positive measure on $X^\an$. 
\end{corollary}

\begin{proof}
If $S(\CX)$ consists of a single point,  the result is immediate.  We henceforth assume that $S(\CX)$ has an edge.

If the valency of every vertex of $S(\CX)$ is at least 3, then $\mu$ is 
automatically smooth on $S(\CX)$.
Otherwise, we take a blowing-up as follows. 

For every vertex $v$ of valency 1 or 2 in $S(\CX)$, corresponding to an irreducible component $C$ of the special fiber of $\CX$, 
	we take {two} distinct closed points of $C$ which do not lie on any other irreducible component. Take $\CX'\to\CX$ to be the blowing-up of $\CX$ along the set of these closed points. 
Denote the reduction graph $\Gamma=S(\CX)$ (resp. $\Gamma'=S(\CX')$). Denote $i=i_\CX$ (resp. $i'=i_\CX'$), and $\tau=\tau_\CX$ (resp. $\tau'=\tau_\CX'$). 
There are a canonical inclusion $i_0:\Gamma\emb \Gamma'$ and a canonical deformation-retraction $\tau_0:\Gamma'\to \Gamma$, which are compatible with $i,i',\tau,\tau'$. 

View $\Gamma$ as a subgraph of $\Gamma'$ by $i_0$.
View the push-forward $i_{0,*}\mu$ {as} a piecewise smooth $(1,1)$-form on $\Gamma'$, obtained from $\mu$ on $\Gamma$ by extension by zero to the extra edges of $\Gamma'$. By the construction of $\Gamma'$,  we see that $i_{0,*}\mu$ is a smooth $(1,1)$-form on $\Gamma'$. 

We can find a smooth $(1,1)$-form 
$\mu'$ on $\Gamma'$ such that 
$\mu'-i_{0,*}\mu$ is a positive measure and
	$$
	\int_{\Gamma'}\mu'=\deg(L).
	$$
	By Theorem \ref{MA for measure on skeleton}, there is a weakly smooth metric $\|\cdot\|$ of $L$ such that
	$$
	c_1(L,\|\cdot\|)=i'_{*}\mu'.
	$$
	This metric is the desired one.
\end{proof}

\subsubsection{Morphisms between curves} \label{subsec morphism}
Finally, we introduce a projection formula. The following paragraph is due to \cite[Def. 3.2.5, Lem 3.2.12, Lem. 3.6.1, Lem. 5.6.4, Prop. 5.6.8]{GJR2}.

In this subsubsection, as in the later application, we assume that $K$ is an algebraically closed non-archimedean field.

Let $\varphi:X'\to X$ be a morphism of smooth projective curves over $K$ of degree $\deg(\varphi)$. Suppose that $\varphi$ can be extended to a morphism $\varphi:\CX'\to\CX$ of strictly semi-stable models. 
By \cite[Lem. 5.6.4]{GJR2}, the  map $\varphi: X'^\an\to X^\an$ induces a map $\varphi_{\rm sk}:S(\CX')\to S(\CX)$. 
The map is \emph{piecewise linear} in the sense that there are subdivisions of 
$S(\CX')^*$ and $S(\CX)^*$ of $S(\CX')$ and $S(\CX)$, the induced map 
$\varphi:S(\CX')^*\to S(\CX)^*$ is linear on every edge.

As in \cite[Def. 3.2.1]{GJR2}, for any $e'\in E(S(\CX')^*)$ with 
$e=\varphi(e')\in E(S(\CX)^*)$, the \emph{expansion factor} of $\varphi$ at $e'\in E(S(\CX')^*)$ is 
$$
d_{e'}(\varphi)= \frac{l(e)}{l(e')} =\left| \frac{\d t_{e}}{\d t_{e'}} \right|.
$$
For a line segment $\ell'$ in $e'$ of positive length, we denote 
$$
d_{\ell'}(\varphi)=d_{e'}(\varphi)= \frac{l(\varphi_\sk(\ell'))}{l(\ell')}.
$$

The expansion factor satisfies the following global equality. 
\begin{lem}\label{total degree}
For a fixed edge $e$ of $S(\CX)^*$, 
\[
\deg(\varphi)=\sum_{e'\mapsto e}d_{e'}(\varphi).
\]
\end{lem}
\begin{proof}
See \cite[Prop. 5.6.8]{GJR2} and the paragraph after it, where the relevant degrees are introduced in \cite[Def. 3.2.2, Def. 3.2.5]{GJR2}.
\end{proof}

The map $\varphi$ induces a pull-back
\[
\varphi^*:A^{p,q}(S(\CX))\longrightarrow A^{p,q}(S(\CX')).
\]
We only present the case $p=q=1$, and we will pass the computation to the subdivisions $S(\CX')^*$ and $S(\CX)^*$ above.  
Then the pull-back is simply given by
\[
\varphi^*(f_e\d' t_e\d'' t_e)_{e\in E(S(\CX)^*)}
:=
(d_{e'}(\varphi)^2 (f_e\circ\varphi_{\rm sk})\d' t_{e'}\d'' t_{e'})_{e'\in E(S(\CX')^*)}.
\]

Now we have the following nice result, which also implies Lemma \ref{total degree}. 

\begin{lem}\label{total expansion}
Let $u'\in S(\CX')^*$ be a point, and denote $u=\varphi_\sk(u')$.
Let $e\in S(\CX)^*$ be an edge containing $u$.
Let $e_1',\dots, e_r'$ be all the edges of $S(\CX')^*$ containing $u'$ and mapped bijectively to $e$ via $\varphi_\sk$. 
Then 
$$
[\CH(u'):\CH(u)]=\sum_{j=1}^r d_{e_j'}(\varphi)
=\sum_{j=1}^r \frac{l(e)}{l(e_j')}.
$$
Here $\CH(u')$ and $\CH(u)$ are the completed residue fields in the Berkovich spaces.
\end{lem}

\begin{proof}
Consider the additive valuation $\val: K^\times \to\RR$ given by $\val(a)=-\log|a|$. 
The value group $\val(K^\times) \subset \RR$ and the $\QQ$-vector space 
$R=\val(K^\times)\otimes_\ZZ\QQ \subset \RR$.  
By \cite[Lem. 5.6.4]{GJR2}, the lengths of all edges of $S(\CX')^*$ and $S(\CX)^*$ belong to $R$. 

We first assume that $u$ is $R$-rational on $e$ in the sense that the two line segments from $u$ to the two endpoints of $e$ belong to $R$. 
Then $u$ is $R$-rational on every $e_j'$. 
By \cite[Lem. 5.6.4]{GJR2}, up to replacing $K$ by a finite extension, and blowing up the corresponding base changes of $\CX$ and $\CX'$, we can assume that $\varphi_\sk: S(\CX')\to S(\CX)$ is already linear on every edge, and that 
$u'$ and $u$ are vertices of $S(\CX')^*$ and $S(\CX)^*$. 
We further assume that all nodes of $\CX'$ and $\CX$ are rational points over the residue field $\wt K$. 

Assume that $u$ corresponds to the irreducible component $C$ of the special fiber $\CX_{\wt K}$ of $\CX$,  and the edge $e$ corresponds to a node $P$ of $\CX_{\wt K}$ contained in $C$.
Assume that $u'$ corresponds to the irreducible component $C'$ of the special fiber $\CX'_{\wt K}$ of $\CX'$, and the edge $e_j'$ corresponds to a node $P_j'$ of $\CX_{\wt K}'$ contained in $C'$.  
By \cite[Rmk. 5.6.7]{GJR2}, which is originally due to \cite[Prop. 2.2.27(iv)]{Thu}, the expansion factor 
$d_{e_j'}(\varphi)$ is equal to the ramification index $e(P_j'/P)$ of $P_j'$ above $P$ under the finite morphism $C'\to C$. 
It follows that 
$$
\sum_{j=1}^r d_{e_j'}(\varphi)
=\sum_{P_j'} e(P_j'/P)
=\deg(C'/C)
=[\wt\CH(u'):\wt\CH(u)]
=[\CH(u'):\CH(u)].
$$
Here the third equality holds as the residue fields $\wt\CH(u')$ and $\wt\CH(u)$ are exactly the function fields of $C'$ and $C$, and the fourth equality holds as the extension $\CH(u')/\CH(u)$ is unramified. 
This proves the lemma in the case where 
$u$ is  $R$-rational on $e$.

It remains to treat the case where $u$ is not  $R$-rational  on $e$. 
Then $u$ and $u'$ are points of type 3 on the Berkovich spaces, and they are not vertices. 
Hence $e$ is the only edge containing $u$,   and $e'=e_1'$ (with $r=1$) is the
only edge containing $u'$. 
Then $e$ and $e'$ correspond to nodes $P\in \CX_{\wt K}$ and $P'\in \CX'_{\wt K}$. 
Assume that the equation of $\CX$ at $P$ (resp. $\CX'$ at $P'$) is \'etale locally given by $t_1t_2=a$ with $a\in O_K$ (resp. $t_1't_2'=a'$ with $a'\in O_K$). 
By the definition of the skeletons in \cite[\S5.3, (5.3.1)]{GJR2}, we have lengths 
$l(e)=\val(a)$ and $l(e')=\val(a')$.
Moreover, by the loc. cit., the norm $|\cdot |_u$ corresponding to $u\in e$ is given by 
the norm 
$$
\left| \sum a_n t_1^n\right|_u=\max |a_n| e^{-n \delta},
$$
where $\delta\notin R$ is the length of the line segment from $u$ to the endpoint of $e$ corresponding to the divisor $t_2=0$. 
We have a similar description for $|\cdot |_{u'}$. 
Then the value group $\log|\CH(u)^\times|$ is generated by $\log  |K^\times|$ and 
$\ZZ\delta$ in $\RR$. 
Note that the residue fields $\wt\CH(u')=\wt K$ and $\wt\CH(u)=\wt K$ for points of type 3. 
Then we have 
\begin{align*}
[\CH(u'):\CH(u)]
&=[|\CH(u')^\times|:|\CH(u)^\times|]\\
&=[\log|K^\times|+\ZZ\delta':\log|K^\times|+\ZZ\delta]\\
&=\delta/\delta'=l(e)/l(e')=d_{e_j'}(\varphi).
\end{align*}
This proves the result in the case where $u$ is not $R$-rational on $e$.
\end{proof}

\subsection{Lipschitz properties for effective 0-cycles}
\label{sec lipschitz}

In this subsection, we list some notions and results on Lipschitz properties of functions on effective 0-cycles. 
The setting of this section is independent of the global setting of this paper, and it will be used only in \S\ref{sec proof holder} and \S\ref{sec proof distance}. 

\subsubsection{Distances, effective 0-cycles  and transport plans}
\label{subsec transport plan}

Let $(\Gamma, \dist)$ be a \emph{metric space}, i.e. a set $\Gamma$ with a function 
$\dist:\Gamma^2\to \RR_{\geq0}$ satisfying the following properties. 
\begin{enumerate}[(1)]
\item (positivity)
$\dist(x,y)\geq 0$ for any $x,y\in \Gamma$, and the equality holds if and only if 
$x=y.$

\item (symmetry)
$\dist(x,y)=\dist(y,x)$ for any $x,y\in \Gamma$. 

\item (triangle inequality)
$\dist(x,z)\le \dist(x,y)+\dist(y,z)$ for any $x,y,z\in \Gamma$. 
\end{enumerate}
The function $\dist:\Gamma^2\to \RR$ is called a \emph{distance function}. 
 
By an \emph{effective 0-cycle} on $\Gamma$, we mean a formal linear combination 
$$
\alpha=\sum_{i=1}^r a_i x_i, \quad x_i\in \Gamma, \ a_i\in \RR_{\geq 0}.
$$
The \emph{degree} of $\alpha$ is defined by 
$$
\deg(\alpha):=\sum_{i=1}^r a_i.
$$
For any function $f:\Gamma\to \RR$, we define 
$$
f(\alpha):=\sum_{i=1}^r a_i f(x_i).
$$

Let $\beta=\sum_{j=1}^s b_j y_j$ be another effective 0-cycle on $\Gamma$ with $\deg(\alpha)=\deg(\beta)$.
A \emph{transport plan}  from $\alpha$ to $\beta$ is an $r\times s$ matrix $M=(m_{i,j})$ with entries $m_{i,j}\geq 0$ and with relations
$$
\sum_{j=1}^s m_{i,j}=a_i, \quad \sum_{i=1}^r m_{i,j}=b_j. 
$$ 
In terms of matrix multiplication, the relations can be written as
$$
M\cdot (1,\dots, 1)_s^T=(a_1,\dots, a_r)^T, \quad 
(1,\dots, 1)_r \cdot M =(b_1,\dots, b_s).
$$
Here $(1,\dots, 1)_s$ denotes the row vector of $s$ components, and $(1,\dots, 1)_s^T$ denotes the transpose of the row vector into a column vector.  

The \emph{$M$-distance} between $\alpha$ and $\beta$ is defined as 
$$
\dist(\alpha, \beta, M)
:= \sum_{i,j} m_{i,j} \dist(x_i,y_j). 
$$
The \emph{distance} between $\alpha$ and $\beta$ is defined as 
$$
\dist(\alpha, \beta)
:=\inf_{M}  \dist(\alpha, \beta, M),
$$
where the infimum is over all {transport plans} $M$.
If $\alpha$ and $\beta$ are points of $\Gamma$, this recovers the original distance function. 

The above setting gives a natural finite-dimensional linear programming problem.  Namely,
suppose that a mass \(a_i\) is located at the point
\(x_i\) (for every \(i=1,\dots,r\)), and the goal is to redistribute the mass so that a mass \(b_j\) is
located at the point \(y_j\)  (for every \(j=1,\dots,s\)).  A transport plan \(M=(m_{i,j})\) aims to carry mass \(m_{i,j}\) from \(x_i\) to \(y_j\), so it is a collection of paths for the transportation.
The $M$-distance $\dist(\alpha, \beta, M)$
is the total transportation cost of this plan. 
 The distance
\(\dist(\alpha,\beta)\) is the minimal total transportation cost from
\(\alpha\) to \(\beta\).  With this interpretation, many results introduced later are intuitive.

We can also define the push-forward of effective 0-cycles. 
Namely, let $\pi:\Gamma\to \Gamma'$ be a map of metric spaces.  
Let $\alpha=\sum_{i=1}^r a_i x_i$ be an effective 0-cycle on $\Gamma$.
Then the \emph{push-forward} of $\alpha$ is
$$
\pi_*\alpha = \sum_{i=1}^r a_i \pi(x_i),
$$ 
viewed as an effective 0-cycle on $\Gamma'$.
It is easy to have $\deg(\pi_*\alpha)=\deg(\alpha)$. 

Let $\beta$ be another effective 0-cycle on $\Gamma$ with 
$\deg(\beta)=\deg(\alpha)$.  
Let $M=(m_{i,j})$ be a transport plan from $\alpha$ to $\beta$. 
Then $M=(m_{i,j})$ induces a natural transport plan from $\pi_*\alpha$ to $\pi_*\beta$, which we denote by $M$ by abuse of notation.

\subsubsection{Triangle inequalities}
\label{subsec triangle inequality}

Let $\alpha, \beta,\gamma$ be effective 0-cycles of equal degrees on a metric space $\Gamma$. 
Let
\(M=(m_{ij})\) be a transport plan from \(\alpha\) to \(\beta\), and let
\(N=(n_{jk})\) be a transport plan from \(\beta\) to \(\gamma\).  
There is a canonical composition $P=M\circ N$, which is a transport plan $P=(p_{ik})$ from $\alpha$ to $\gamma$ given by
\[
p_{ik}=\sum_{j}\frac{m_{ij}n_{jk}}{b_j}.
\]
Here we write
\[
\alpha=\sum_{i=1}^r a_i x_i,\qquad
\beta=\sum_{j=1}^s b_j y_j,\qquad
\gamma=\sum_{k=1}^t c_k z_k.
\]
And we have assumed that every $b_j>0$, by removing the terms with $b_j=0$. 
In term of matrix multiplication, we have 
$$
P=M\cdot  \mathrm{diag}(b_1^{-1},\dots, b_s^{-1})\cdot N,
$$
where the matrix in the middle is the diagonal matrix. 
By matrix multiplication, it is easy to check that $P$ is a transport plan from $\alpha$ to $\gamma$. 
Now we have the following triangle inequalities.

\begin{lem}[triangle inequality]
\label{triangle inequality}
We have
$$
\dist(\alpha, \gamma, M\circ N)
\leq
\dist(\alpha, \beta, M)
+
\dist(\beta, \gamma, N)
$$
and
$$
\dist(\alpha, \gamma)
\leq
\dist(\alpha, \beta)
+
\dist(\beta, \gamma).
$$
\end{lem}

\begin{proof}
The first inequality {follows} from
\[
\begin{aligned}
\sum_{i,k}p_{ik}\dist(x_i,z_k)
&=
\sum_{i,j,k}\frac{m_{ij}n_{jk}}{b_j}
\dist(x_i,z_k) \\
&\le
\sum_{i,j,k}\frac{m_{ij}n_{jk}}{b_j}
\bigl(\dist(x_i,y_j)+\dist(y_j,z_k)\bigr) \\
&=
\sum_{i,j}m_{ij}\dist(x_i,y_j)
+
\sum_{j,k}n_{jk}\dist(y_j,z_k).
\end{aligned}
\]
Taking the infimum over all transport plans \(M\) and \(N\), the first inequality implies the second inequality.
\end{proof}

\subsubsection{Lipschitz property}

Now we have some easy results on Lipschitz properties. 
The following one concerns Lipschitz property for effective 0-cycles. 

\begin{lem}[Lipschitz property: cycles] \label{lipschitz1}
Let $\Gamma$ be a metric space. 
Let $f:\Gamma\to \RR$ be a function satisfying the {Lipschitz property} in the sense that there is a constant $C$ such that 
$$
|f(x)-f(y)|\le C\cdot\dist(x,y),
 \quad \forall x,y\in \Gamma. 
$$
Then for any effective \(0\)-cycles \(\alpha,\beta\)
of the same degree, 
\[
 |f(\alpha)-f(\beta)|
 \le C\cdot \dist(\alpha,\beta). 
\]
\end{lem}

\begin{proof}
Write
\[
\alpha=\sum_i a_i x_i,
\qquad
\beta=\sum_j b_j y_j.
\]
Let $M=(m_{i,j})$ be a transport plan from $\alpha$ to $\beta$. 
By  definition, we have
\[
\begin{aligned}
f(\alpha)-f(\beta)
&=\sum_i a_i f(x_i)-\sum_j b_j f(y_j) \\
&=\sum_{i,j}m_{ij}f(x_i)-\sum_{i,j}m_{ij}f(y_j) \\
&=\sum_{i,j}m_{ij}\bigl(f(x_i)-f(y_j)\bigr).
\end{aligned}
\]
It follows that
\[
\begin{aligned}
|f(\alpha)-f(\beta)|
\le \sum_{i,j}m_{ij}|f(x_i)-f(y_j)| 
\le C\sum_{i,j}m_{ij}\dist(x_i,y_j)
=C\cdot\dist(\alpha,\beta, M).
\end{aligned}
\]
Taking the infimum over all transport plans \(M\), we obtain
the result. 
\end{proof}

\subsubsection{Product space}

Let $(\Gamma_1,\dist_1), \dots, (\Gamma_N,\dist_N)$ be metric spaces. 
Endow the product $X=\Gamma_1\times\cdots\times\Gamma_N$ 
with the \emph{box distance} given by 
$$
\dist((x_1,\cdots,x_N), (y_1,\cdots,y_N)):=\dist_1(x_1,y_1)+\cdots+\dist_N(x_N,y_N).
$$

Let $\alpha, \beta$ be effective 0-cycles of equal degrees on $X$. 
The box distance induces a \emph{distance} 
$$
\dist(\alpha, \beta)
=\inf_{M}  \dist(\alpha, \beta, M)
=\inf_{M} 
\sum_{i=1}^N \dist_i(p_{i,*}\alpha, p_{i,*}\beta,M),
$$
where the infimum is over all {transport plans} $M$ 
from $\alpha$ to $\beta$. 
Here $p_i:X\to \Gamma_i$ is the $i$-th projection, and $M$ is also viewed as a transport plan from $p_{i,*}\alpha$ to $p_{i,*}\beta$; the corresponding transportation cost is denoted by $\dist_i(p_{i,*}\alpha,p_{i,*}\beta,M)$.

We have the following quick Lipschitz property for product functions.

\begin{lem}[Lipschitz property: product] \label{lipschitz2}
For $i=1,\dots, N$, let $f_i$ be a bounded function on $\Gamma_i$ satisfying the {Lipschitz property}. 
Define a function 
\[
f:\Gamma_1\times\cdots\times\Gamma_N\lra \RR, \quad
(x_1,\dots,x_N)\longmapsto f_1(x_1)\cdots f_N(x_N).
\]
Then $f$ satisfies the {Lipschitz property} on $\Gamma_1\times\cdots\times\Gamma_N$ under the box distance. 
\end{lem}

\begin{proof}
This is a consequence of the identity
$$
\begin{aligned}
f(x)-f(y)
&=\prod_{i=1}^N f_i(x_i)-\prod_{i=1}^N f_i(y_i) \\
&=\sum_{i=1}^N
\Big(\prod_{j<i} f_j(x_j)
\Big)
\Bigl(f_i(x_i)-f_i(y_i)\Bigr)
\Big(\prod_{j>i} f_j(y_j)
\Big).
\end{aligned}
$$
\end{proof}

\section[Equidistribution at fully transcendental valuations]
{Equidistribution at fully transcendental valuations} \label{equi on NS fibers section}

The goal of this section is to prove Theorem \ref{equi ft}.
The idea of proof has already been sketched in \S\ref{sec idea}, and can be also seen in the title of each subsection in the following.

\subsection{Step 1: Equidistribution at divisorial valuations}

In this subsection, we prove the equidistribution theorem at divisorial valuations. 
The remaining subsections are to extend the equidistribution theorem from divisorial valuations to all fully transcendental valuations by approximation.

\subsubsection{Density lemma} \label{section density lemma 1}

To prove the equidistribution theorem at divisorial valuations, we will need an easy density lemma on continuous functions on Berkovich spaces.
We will start with the classical Stone--Weierstrass theorem. 

Let $M$ be a locally compact Hausdorff topological space. Denote by $C(M)$ the ring of real-valued continuous functions on $M$. Denote by $C_c(M)$ the subset of $C(M)$ consisting of the compactly supported functions. The \emph{uniform topology} on $C(M)$ is the metric topology induced by the supremum norm.

Let $A$ be a \emph{subalgebra} of $C(M)$, i.e., a $\QQ$-vector subspace closed under multiplication. 
We say that $A$ \emph{vanishes nowhere} if for every $x$ in $M$, there is some $f\in A$ such that $f(x) \neq 0$.
We say that $A$ \emph{separates points} if for every two different points $x,y\in M$, there is some $f\in A$ such that $f(x)\neq f(y)$.
We have the following version of Stone--Weierstrass theorem.

\begin{theorem}[Stone--Weierstrass] \label{S-W for compactly support}
Suppose that $M$ is a locally compact Hausdorff space and that $A$ is a $\QQ$-subalgebra of $C_c(M)$. Then $A$ is dense in $C_c(M)$ under the uniform topology if and only if it vanishes nowhere and separates points.
\end{theorem}

Now we are ready to prove the following density lemma.

\begin{lemma}\label{density fiber}
	Let $K$ be a non-archimedean field. Let $\pi: X\to Y$ be a flat morphism of quasi-projective varieties over $K$. Let $y\in Y^\an$ be any point. Then the restriction map 
	$$
	C_c(X^\an) \lra C_c(X_y^\an), \quad f \longmapsto f|_{X_y^\an}
	$$
has a dense image in $C_c(X_y^\an)$. 
\end{lemma}
\begin{proof}
Denote by $A$ the image of the restriction map. 
It is obvious that $A$ is a subalgebra of $C_c(X_y^\an)$ and that $A$ vanishes nowhere. By Theorem \ref{S-W for compactly support}, it suffices to check that $A$ separates points. Let $x_1,x_2\in X_y^\an$ be distinct. Then we can view $x_1,x_2$ as points in $X^\an$ naturally. Take an open analytic domain $U$ of $x_1$ in $X^\an$ such that $x_2\notin U$. Take a compactly supported function $h:U\to\RR$ such that $h(x_1)=1$. Then we extend $h$ to a function $h:X^\an\to\RR$ by setting $h(z)=0$ for all $z\notin U$. Then $h|_{X_y^\an}\in A$ satisfies $h(x_1)=1$ and $h(x_2)=0$. This completes the proof.
\end{proof}

\subsubsection{Lifting model divisors}

We also need the following basic result on adelic divisors. 

\begin{lemma}\label{mod surj lem}
Let $F$ be a finitely generated field over $\ZZ$. Then the functorial map 
\[
\wh\Pic(F/\ZZ)_\mathrm{mod}\longrightarrow\wt\Pic(F/\QQ)_\mathrm{mod}
\]
is surjective.
\end{lemma}
\begin{proof}
It suffices to prove that the functorial map
$$
\wh\Div(F/\ZZ)_\mathrm{mod}\longrightarrow\wt\Div(F/\QQ)_\mathrm{mod}
$$
is surjective.
Fix an element $\TD\in \Div(F/\QQ)_\mathrm{mod}$, and we want to lift it to 
$\wh\Div(F/\ZZ)_\mathrm{mod}$.
By definition, there exist a projective variety $Y$ over $\QQ$ with function field $F$ and a Cartier divisor $D_1$ on $Y$ representing $\TD$.
By linearity, it suffices to prove the existence of a preimage when $D_1$ is effective.

Let $\CY_1$ be a projective model of $Y$ over $\Spec \ZZ$.
Let $\CD_1$
be the Zariski closure of $D_1$ in $\CY_1$, viewed as a subscheme of $\CY_1$.
Let $\CY_2$ be the blowing-up of $\CY_1$ along $\CD_1$, and let $\CD_2$ denote the exceptional divisor, which is a Cartier divisor on $\CY_2$.
By base change, $\CY_{2,\QQ}$ is the blowing-up of $\CY_{1, \QQ}$ along 
$\CD_{1,\QQ}$. 
Since $\CD_{1,\QQ}=D_1$ is a Cartier divisor on $\CY_{1, \QQ}=Y$, we have 
$\CY_{2,\QQ}= Y$ and $\CD_{2,\QQ}=D_1$.
In other words, the pair $(\CY_2, \CD_2)$ extends the pair $(Y,D_1)$. 
Extend $\CD_2$ to an arithmetic divisor $\CDD_2$ on $\CY_2$. The pair 
$(\CY_2, \CDD_2)$ gives a  preimage of $\TD$.
\end{proof}

\subsubsection{Equidistribution at divisorial valuations}

Now we are ready to prove the equidistribution theorem (Theorem \ref{equi ft}) in the case that $v\in\CM(F/k)^\ft$ is a divisorial valuation 
(cf. Definition \ref{def fully tran}). For the sake of readers, we state the result here in a slightly different setting for convenience of the proof.

\begin{theorem}[equidistribution at divisorial valuations] 
\label{equi divisorial}
	\kkk
	Let $F$ be a finitely generated field over $k$.
	Let $X$ be a geometrically integral quasi-projective variety over $F$.
	Let $\overline L$ be a nef adelic line bundle on $X/k$ such that $\deg_{\wt L}(X/F)>0$.
	Let $\{x_m\}_m$ be a generic sequence  in $X(\overline F)$ which is numerically small with respect to $\OL$.
	Then for every divisorial valuation $v\in\CM(F/k)^\ft$,
	the Galois orbit of $\{x_m\}_m$ is equidistributed in $X_{v}^\an$ for $\d\mu_{\overline L,v}$.
\end{theorem}

\begin{proof}
If $k=\ZZ$, take $K=\QQ$ and $B=\Spec \ZZ$. 
Then $v$ is a divisorial point of $\CM(F/K)_{v_0}$ for the restriction $v_0=v|_K$.
We first assume that the restriction $v_0=v|_K$ is non-trivial, and will treat the trivial case in the end.

If $k$ is a field, as $v$ is divisorial, there is an intermediate field $K$ of 
$F/k$ with $\trd(K/k)=1$, such that the restriction $v_0=v|_K$ is non-trivial on $K$, and $v$ is a fully transcendental point of $\CM(F/K)_{v_0}$.
We can further assume that $K$ is purely transcendental over $k$, which can be obtained by replacing $K$ by $k(t)$ for some $t\in K$ transcendental over $k$. 
Take $B=\PP^1_k$ to be a projective line with function field $K$. 

In both cases, let $\SY$ be a normal projective variety over $K$ with function field $F$. Denote $d=\dim \SY=\trd(F/K)$. 
Take a projective model $\SX$ of $X$ over $K$ such that there is a morphism $\pi:\SX\to \SY$ of projective varieties over $K$ with generic fiber $X\to \Spec {F}$.

By Lemma \ref{prop of bijective map between ft points}, 
$v$ is given by a divisorial point on $\SY^\an_{v_0}$, which we still denote by $v$ by abuse of notations. 
By Theorem \ref{div pt equiv}, we can find a normal (projective) integral model $\CY$ of $\SY$ over $B$ such that $v$ corresponds to an irreducible component $D$ of $\CY_{v_0}$. Here $\CY_{v_0}$ is the fiber of $\CY$ over the closed point $v_0$ of $B$ representing the place $v_0$ of $K$. 

By possibly blowing-up $\CY$ along a center disjoint with $\CY_{v_0}$, we can find a generically finite morphism 
$\varphi:\CY\to\PP^d_B$ which is finite along $D$. 
To construct such a morphism, take a very ample line bundle $\CA$ on 
$\CY$, take global sections $s_1,\dots, s_{d+1}$ of $\CA$ successively such that 
$\div(s_1)\cap\cdots \cap \div(s_{d+1})\cap \CY_{v_0}$ is empty. 
Then these sections define a rational map $\CY\dashrightarrow \PP^d_B$, which is a morphism in a neighborhood of $\CY_{v_0}$.
Blowing-up $\CY$ along $\div(s_1)\cap\cdots \cap \div(s_{d+1})$, we obtain the desired morphism.

Denote by $h:\PP_B^d\to\PP_B^d$ the square map on the coordinates. By Tate's limiting argument, there is an invariant adelic line bundle $\overline{\CO}(1)\in\HPic(\PP^d_B/k)_\nef$ such that
	$$
	h^*\overline{\CO}(1)=2\cdot\overline{\CO}(1).
	$$
Moreover, $\overline{\CO}(1)$ satisfies the Moriwaki condition. Denote $\OH=\varphi^*\overline{\CO}(1)$. Then $\OH$ also satisfies the Moriwaki condition. Note that
	$$
c_1(\OH)_{v_0}^d=\sum_{i=1}^r a_{i}\delta_{v_i},
	$$
		where $v=v_1, v_2,\dots,v_r$ are the divisorial points in the support of this measure, and every coefficient $a_i>0$. In particular, $a_1>0$ since $\varphi$ is finite along $D$.

Now we are ready to apply Theorem \ref{equi relative}.
As a consequence, on $\SX_{v_0}^\an$, we have the weak convergence 
	$$
		\frac{1}{\deg(x_m)} \delta_{{\triangle(x_m),v_0}}\, c_1(\pi^*\OH)_{v_0}^{d}
		\lra \frac{1}{\deg_{\wt L}(X/F)} c_1(\OL)_{v_0}^{n}c_1(\pi^*\OH)_{v_0}^{d}.
	$$
	By Corollary \ref{projection formula Dirac}, we have the weak convergence on 
	$\SX^\an_{v_0}$,
	$$
		\sum_{i=1}^r a_i  \mu_{x_m,v_i}\,
	\lra \frac{1}{\deg_{\wt L}(X/F)} \sum_{i=1}^r a_i\ \rho_{i,*}
	\big( c_1(\OL|_{\SX_{v_i}})^n\big).
	$$
	Here $\rho_i: \SX_{v_i}^\an\to {\SX_{v_0}^\an}$ denotes the natural injection.
	Note that $\SX_{v_1}^\an=X_{v}^\an$.
	
To separate $v_1$ from the weak convergence, it suffices to prove that the space 
$$\{
f|_{\SX_{v_1}^{\an}}: f\in C_c(\SX_{v_0}^{\an}), \ f|_{\SX_{v_i}^{\an}}=0, \ \forall\, i\geq 2
\}$$
is uniformly dense in $C_c(\SX_{v_1}^\an)$.
	By multiplying functions of $C_c(\SX_{v_0}^{\an})$ by the pullback of a continuous function $g$ on $\SY^\an$
	with  $g(v_1)=1$ and $g(v_i)=0$ for each $i\geq 2$, 
	 it suffices to prove that the space 
$$\{
f|_{\SX_{v_1}^{\an}}: f\in C_c(\SX_{v_0}^{\an})\}$$
is uniformly dense in $C_c(\SX_{v_1}^\an)$.
This has been established in Lemma \ref{density fiber}.
Then the equidistribution on $X_{v}^\an$ follows.

It remains to treat the case that $k=\ZZ$ and $v_0=v|_K$ 
is trivial on $K=\QQ$. We will convert this case to the case $k=\QQ$ (with trivial valuation) of the theorem.
For this, it suffices to check the implication of the smallness condition in these two different settings. 
By assumption, we have 
$$
\frac{1}{\deg (x_m)}\OH_1\cdots\OH_{d}\cdot \OL \cdot {z_m} 
\lra \frac{1}{(n+1)\deg_{\wt L} (X)}\OH_1\cdots\OH_{d}\cdot\OL^{n+1}
$$
for any nef model adelic line bundles $\OH_1,\dots, \OH_d\in \wh\Pic(F/\ZZ)$. 
Here $z_m$ denotes the closed point of $X$ corresponding to $x_m$. 
Denote by $\TL_K$ the image of $\OL$ in $\wt\Pic(X/K)$, and denote by 
$\TH_{i,K}$ the image of $\OH_i$ in $\wt\Pic(F/K)$. 
Let $\ON$ be an element of $\wh\Pic(K/\ZZ)$ of degree 1, and 
take $\OH_d$ to be the image of $\ON$ under the pull-back map  
$\wh\Pic(K/\ZZ)\lra \wh\Pic(F/\ZZ)$.
Then the above convergence becomes 
$$
\frac{1}{\deg (x_m)}\TH_{1,K}\cdots\TH_{d-1,K}\cdot\TL_K \cdot {z_m} 
\lra \frac{1}{(n+1)\deg_{\wt L} (X)}\TH_{1,K}\cdots\TH_{d-1,K}\cdot\TL_K^{n+1}
$$
By Lemma \ref{mod surj lem}, this yields a smallness condition for all nef model line bundles
$\TH_{i,K}\in \wt\Pic(F/\QQ)_\mathrm{mod}$.
Therefore, we reduce the situation to the case $k=\QQ$, at the cost of replacing the original smallness condition by a weaker one.
However, in the above argument of the case $k=\QQ$, each $\TH_{i,K}$ is only regarded as an element of $\wt\Pic(F/\QQ)_\mathrm{mod}$.
This completes the proof.
\end{proof}

\subsection{Step 2: Reduce to an approximation theorem}

The goal of this subsection is to state an approximation theorem (cf. Theorem \ref{approximation1}), by which we will see that our equidistribution theorem at divisorial valuations (Theorem \ref{equi divisorial}) implies our main theorem (Theorem \ref{equi ft}) immediately.

\subsubsection{The approximation theorem}
To introduce the approximation theorem, we first establish some notation.

Let $K$ be a non-archimedean field. 
Let $F$ be a finitely generated field over $K$ of {transcendence degree} $d\geq1$. 
We will reduce general $d\geq1$ to $d=1$, but here we {set up} notation for general $d$ to be used in the reduction process. 

Let $X$ be a geometrically integral quasi-projective variety of dimension $n$ over $F$. Let $\ol L$ be a nef adelic line bundle on $X/O_K$ such that 
$\deg_{\wt L}(X/F)>0$.
Here the geometric part $\wt L$ is the image of $\OL$ in $\TPic(X/F)$. 
To avoid confusion, we will denote by 
$\TL_K$ {the} image of $\OL$ in $\TPic(X/K)$. 

As before, for a valuation $v\in\CM(F/K)$, denote by 
$X_v^\an$ the Berkovich space of the variety $X_v=X\times_{F} F_v$ over the completion $F_v$. 
Denote by $\OL|_{X_v}$ the image of $\OL$ under the functorial map
$$
\wh\Pic(X/O_K) \lra \wh\Pic(X_v/O_{F_v}). 
$$

To define the functorial map, it suffices to do it for model line bundles. Let $(\CX,\CL)$ be a projective model of $(X,L)$ over $O_K$.
Let $\CY$ be a projective model of $\Spec F$ over $O_K$. 
Up to blowing-up, we can assume that the morphism $X\to \Spec F$ extends to a morphism $\CX\to \CY$ over $O_K$.
By the valuative criterion, the composition $\Spec F_v\to \Spec F\to \CY$ over $\Spec O_K$ extends to a unique morphism $\Spec O_{F_v}\to \CY$.
Now we take the base change of $(\CX,\CL)$ by $\Spec O_{F_v}\to \CY$.
The resulting pair is a projective model of $(X_v, L|_{X_v})$ over $O_{F_v}$, which is the image of $(\CX,\CL)$ under the functorial map. 

Return to the adelic line bundle  $\OL|_{X_v}$ on $X_v/O_{F_v}$. 
Denote by $c_1(\OL)_v^n=c_1(\OL|_{X_v})^n$ the Chambert-Loir measure on $X_v^\an$. 
We still normalize it by 
$$\d\mu_{\OL,v}:=\frac{1}{\deg_{\wt L}(X/F)}c_1(\OL)_v^n.$$
It is a probability measure by \cite[Thm. 1.2]{Guo25}.

For a closed point $x$ of $X$, the base change $x_{F_v}=x\times_FF_v$ is a natural closed subscheme of $X_v$, but it is not necessarily reduced or irreducible. 
This causes lots of notational trouble, but does not bring any essential difficulty to the problem. Because of this, we will start with a series of definitions extending our equidistribution of closed points to effective 0-cycles. 

By a \emph{0-cycle} $\alpha$ of $X$, we mean a formal linear combination 
$$
\alpha=\sum_{i=1}^r a_i [x_i], \quad a_i\in \RR,
$$
where $x_1,\dots, x_r\in X$ are distinct closed points.
The \emph{support} $|\alpha|$ of $\alpha$ is the set of $x_i$ with $a_i\neq 0$. 
The \emph{degree} of $\alpha$ is just 
$$
\deg(\alpha)=\sum_{i=1}^r a_i\deg(x_i).
$$ 
We say that $\alpha$ is \emph{effective} if $r\geq1$ and every coefficient $a_i>0$.

Any $0$-dimensional closed subscheme $Z$ of $X$ defines an effective 0-cycle by
$$
[Z]:=\sum_{x\in Z} \mathrm{mult}_x(Z) [x],
$$
where the multiplicity $\mathrm{mult}_x(Z)$ is the length of the the local ring $\CO_{Z,x}$
as an $\CO_{Z,x}$-module. 

The definitions actually work for any variety over any field. 
For any closed point $z$ of $X_v$, the classical point associated to 
$z$ defines the measure $\deg(z)\delta_z$ on $X_v^\an$.
By linear combination, this definition extends to all 0-cycles on $X_v$. 

Let $\alpha$ be an effective 0-cycle of $X$ as above. 
Then the \emph{base change} $\alpha_{F_v}$ is an effective 0-cycle on $X_v=X_{F_v}$ defined by 
$$
\alpha_{F_v}=\sum_{i=1}^r a_i [x_{i,F_v}]. 
$$
Here $x_{i,F_v}=x_{i}\times_F F_v$ is viewed as a 0-dimensional closed subscheme of $X_{F_v}$ via the morphism $x_{i,F_v}\to X_{F_v}$. 
Then we have a probability measure on $X_v^\an$ by  
$$
\mu_{\alpha, v}:=\frac{1}{\deg(\alpha)}\delta_{\alpha,v}, 
$$
where the Dirac measure $\delta_{\alpha,v}=\delta_{\alpha_{F_v}}$
is defined by linear combination as above. 

If the transcendence degree $d=\trd(F/K)=1$,  
a sequence $\{\alpha_m\}_{m\geq1}$ of effective 0-cycles of $X$ is said to \emph{have bounded $\TL_K$-height over $K$} if  the height
$$
h_{\TL_K}(\alpha_m)
=
\frac{1}{\deg (\alpha_m)}\TL_K\cdot {\alpha_m}$$
is bounded as $m\to \infty.$

For general $d=\trd(F/K)\geq1$, a sequence $\{\alpha_m\}_{m\geq1}$ of effective 0-cycles of $X$ is said to \emph{have bounded $\TL_K$-height over $K$} if for some nef and big adelic line bundle
$$
\TH \in\widetilde{\Pic}(F/K)_{\mathrm{nef}},
$$
viewed as an adelic line bundle on $X/K$ by pull-back, the height
$$
h_{\TL_K}^{\TH}(\alpha_m)
=
\frac{1}{\deg (\alpha_m)}\TH^{d-1}\cdot\TL_K\cdot {\alpha_m}$$
is bounded as $m\to \infty.$
Here the geometric part $\TL_K$ is the image of $\OL$ in $\TPic(X/K)$, and the intersection number is defined by linearity. 
This implies that for all
$$
\TH_1,\cdots,\TH_{d-1}\in\widetilde{\Pic}(F/K)_{\mathrm{int}},
$$
the height
$$
h_{\TL_K}^{\TH_1,\dots,\TH_{d-1}}(\alpha_m)
=
\frac{1}{\deg (\alpha_m)}\TH_1\cdots\TH_{d-1}\cdot\TL_K\cdot {\alpha_m}$$
is bounded as $m\to \infty.$
In fact, by linearity, it suffices to assume that every $\TH_i$ is nef. 
By the arithmetic bigness theorem (cf. \cite[Thm. 5.2.2]{YZ26}), there is a positive integer $a$ such that $a\TH-\TH_i$ is big for every $i$.  
Replacing $\TH$ by $a\TH$, we can assume that $\TH-\TH_i$ is big for every $i$. 
Now the boundedness follows from the inequality 
$$
0\leq \TH_1\cdots\TH_{d-1}\cdot\TL_K\cdot {\alpha_m}
\leq \TH\cdot\TH_2\cdots\TH_{d-1}\cdot\TL_K\cdot {\alpha_m}
\leq \TH^2\cdot\TH_3\cdots\TH_{d-1}\cdot\TL_K\cdot {\alpha_m}
\leq \cdots
\leq \TH^{d-1}\cdot\TL_K\cdot {\alpha_m}.
$$

The sequence $\{\alpha_m\}_{m\geq1}$ is said to be \emph{generic} if every closed subvariety $Z\subsetneq X$ {meets the support $|\alpha_m|$ for only finitely many $m$}. 
There is a more relaxed definition of ``genericity'' in terms of probability, but we do not need it here. 

Finally, our approximation theorem is as follows. 

\begin{theorem}[approximation]
\label{approximation1}
Let $K$ be an algebraically closed non-archimedean field. 
Let $F$ be a finitely generated field extension over $K$ of transcendence degree $1$. 
Let $X$ be a geometrically integral quasi-projective variety of dimension $n$ over $F$. Let $\ol L$ be a nef adelic line bundle on $X/O_K$ such that $\deg_{\wt L}(X/F)>0$.
Let $\{\alpha_m\}_{m\geq1}$ be a generic sequence of effective 0-cycles of $X$ with bounded $\TL_K$-heights over $K$.  If the weak convergence
	$$\mu_{\alpha_m,v}\longrightarrow \d\mu_{\OL,v}$$
holds for all divisorial valuations $v\in \CM(F/K)$, then it holds for all fully transcendental valuations $v\in\CM(F/K)^\ft$.
\end{theorem}

In the following, we first extend the approximation theorem (for $\trd(F/K)=1$) to a version of it for $\trd(F/K)\geq 1$, and then prove that the latter implies the equidistribution theorem \ref{equi ft}.

\subsubsection{Lower the transcendence degree}

The goal here {is to extend} Theorem \ref{approximation1} to a similar statement for $\trd(F/K)\geq 1$ with some {minor revisions}. More precisely, we prove that the theorem implies the following general case.

\begin{theorem}[approximation: high degree]
 \label{approximation high}
Let $K$ be a complete non-archimedean field. 
Let $F$ be a finitely generated field over $K$ of transcendence degree $d\geq1$. 
Let $X$ be a geometrically integral quasi-projective variety of dimension $n$ over $F$. Let $\ol L$ be a nef adelic line bundle on $X/O_K$ such that $\deg_{\wt L}(X/F)>0$.
Let $\{\alpha_m\}_{m\geq1}$ be a generic sequence of effective 0-cycles of $X$ with bounded $\TL_K$-heights over $K$. If the weak convergence
	$$\mu_{\alpha_m,v_1}\longrightarrow \d\mu_{\OL,v_1}$$
holds for all divisorial valuations $v_1\in \CM(F_1/K)$ on all finite extensions $F_1$ of $F$, then it holds for all fully transcendental valuations $v_1\in\CM(F_1/K)^\ft$ on all finite extensions $F_1$ of $F$.
\end{theorem}

Note that the weak convergence  
$$\mu_{\alpha_m,v_1}\longrightarrow \d\mu_{\OL,v_1}$$
is on the Berkovich space $X_{v_1}^\an=X_{F_{1,v_1}}^\an$ associated to the base change $X_{F_{1,v_1}}$ of $X$ to  $F_{1,v_1}$. 
Similarly, the measure $\mu_{\alpha_m,v_1}$ is associated to the cycle 
$\alpha_{m,F_{1,v_1}}$ on $X_{F_{1,v_1}}$, and the measure $\d\mu_{\OL,v_1}$ is associated to the adelic line bundle $\OL_{F_1}$ on $X_{F_1}$.

For application, we only need the case $F_1=F$, but to prove the theorem, we need to find a suitable $F_1$ to get a semistable model.

Note that Theorem \ref{approximation high} differs from Theorem \ref{approximation1} by three requirements.
Namely, Theorem \ref{approximation1} assumes $\trd(F/K)= 1$, assumes that $K$ is algebraically closed, and only considers equidistribution at $v\in \CM(F/K)$ (instead of $v_1\in \CM(F_1/K)$).

\paragraph*{Lower the transcendence degree.}
To reduce Theorem \ref{approximation high} to Theorem \ref{approximation1}, we first reduce Theorem \ref{approximation high} from $\trd(F/K)>1$ to  $\trd(F/K)=1$. 

Assume the setting of Theorem \ref{approximation high}. 
By the condition, we have the equidistribution (i.e. weak convergence) for all divisorial valuations $v_1\in \CM(F_1/K)$. We need to prove that the equidistribution holds for all fully transcendental valuations
$v_1\in \CM(F_1/K)^\ft$.
By base change, it suffices to prove the equidistribution for all fully transcendental valuations  
$v\in \CM(F/K)^\ft$.

Assume that $d=\trd(F/K)>1$. 
We assume that the {result holds} for $d-1$ in the following, and the goal is to deduce the result for $d$. 
	
Before the reduction process, we introduce a transition setting. 
Let $v\in \CM(F/K)^\ft$ be fully transcendental as above. 
Let $E$ be an intermediate field of $F/K$ with $\trd(E/K)=1$.  
Denote by $v_0=v|_E$ the restriction, and denote by $K'=E_{v_0}$ the completion. 
Denote by $F'$ the fraction field of the image of $F\otimes_E K' \to F_v$, and denote $v'=v|_{F'}$. 
As in the proof of Theorem \ref{equi ft2} by Theorem \ref{approximation high}, we have $F'_{v'}=F_v$ and that 
$F'$ is finitely generated over $K'$ with $\trd(F'/K')=\trd(F/E)=d-1$. 

Denote $X'=X\times_F F'$, and $\OL'=\OL|_{X'}$. 
We plan to obtain the equidistribution for $(K',F',X',v')$ by the induction hypothesis.
Denote by $\alpha_{m}'=\alpha_{m,F'}$ the base change, which is an effective 0-cycle on $X'$. 
Via $F'_{v'}=F_v$, we have 
$$X_{v'}'^\an=X_v^\an, \quad
\alpha_{m,v'}'=\alpha_{m,v}, \quad
\d\mu_{\overline L',v'}=\d\mu_{\overline L,v}.$$  
Note that $\alpha_{m}'$ has bounded $\TL'_{K'}$-height over $K'$, which can be proved as in the proof of Theorem \ref{equi ft2} by Theorem \ref{approximation high}. 
This finishes our description of the transition setting. 

Now we return to the proof of the equidistribution at every $v\in \CM(F/K)^\ft$. 
Denote by $\tilde{d}(v)=\trd(\wt F_v/\wt K)$ the extension between the residue fields. 
By Proposition \ref{prop of ttrd}, we have $\tilde{d}(v)\leq d$. 
If $\tilde{d}(v)=d$, then $v$ is a divisorial point, and the equidistribution at $v$ is our assumption. 

We first prove the equidistribution for the case $\tilde{d}(v)>0$.  
By assumption, there is an element $\tilde t\in \wt F_v$ transcendental over $\wt K$. Take $t\in F\cap O_{F_v}$ to be a preimage of $\tilde t$. 
Take $E=K(t)$ to be the intermediate field of $F/K$ to fit the transition setting. 
This gives a datum $(K',F',X',v')$. 
By the choice of $E$, we have $\trd(\wt K'/\wt K)=1$, and thus 
$v_0$ is divisorial in $\CM(E/K)$; moreover, $v'$ is fully transcendental in $\CM(F'/K')$ by Lemma~\ref{sub of fully trans}. 
To prove the equidistribution at $v'$, by the result for $d-1$, it suffices to check that the equidistribution holds at every divisorial $w'\in \CM(F'_1/K')$ on a finite extension $F'_1$ of $F'$. 
To ease the notation, we assume that $F_1'=F_1$; the general case follows a similar strategy.

Denote $w=w'|_F\in \CM(F/K)$. We have $F'_{w'}=F_w$ as before. 
Then 
$$\tilde d(w)=\trd(\wt F_w/\wt K)=\trd(\wt{F'_{w'}}/\wt K')+1=d.$$
Thus $w$ is divisorial in $\CM(F/K)$. 
By assumption, the equidistribution holds at 
$w\in\CM(F/K)$, and thus it holds at $w'\in \CM(F'/K')$. 
Therefore, the result for $d-1$ implies the equidistribution at $v\in\CM(F/K)$ with $\tilde{d}(v)>0$.  

Now we prove the equidistribution at all $v\in \CM(F/K)^\ft$. 
Take an arbitrary intermediate field $E$ of $F/K$ as in the transition setting.
This gives the datum $(K',F',X',v')$. 
We claim that $v'$ is fully transcendental in $\CM(F'/K')$, and $v_0=v|_E$ is fully transcendental in $\CM(E/K)$. In fact, by Proposition \ref{prop of ttrd}, we have
\begin{align*}
d=\ttrd(F'_{v'}/K)
\leq &\ \ttrd(F'_{v'}/K')+
\ttrd(K'/K) \\
\leq &\ \trd(F'/K')+
\trd(E/K)
=d.
\end{align*}
This forces the inequalities to be equalities, and thus $v'$ and $v_0$ are fully transcendental. 

Now the situation is similar to the  case $\tilde{d}(v)>0$. 
In fact, to prove the equidistribution at $v'$, by the result for $d-1$, it suffices to check that the equidistribution holds at every divisorial $w'\in \CM(F_1'/K')$. Assume $F_1'=F'$ again. 
As above, denote $w=w'|_F\in \CM(F/K)$. We have $F'_{w'}=F_w$ as before. 
Then 
$$\tilde d(w)=\trd(\wt F_w/\wt K)\geq \trd(\wt{F'_{w'}}/\wt K')=d-1>0.$$
By Proposition \ref{extension by divisorial}, $w$ is fully transcendental over $K$.
By the above case, the equidistribution holds at 
$w\in\CM(F/K)$. 
By the  result for $d-1$, we have the equidistribution at $v\in\CM(F/K)$.

\paragraph*{Take algebraic closure.}
In the above, we have already reduced Theorem \ref{approximation high} to the case 
$\trd(F/K)=1$. Now we further reduce it to $K$ being algebraically closed and $F_1=F$.  

Return to the setting of Theorem \ref{approximation high}. Denote by $K'=\wh{\ol K}$ the completion of the algebraic closure of $K$. Denote by $F'= FK'$ the compositum, which has a dense subfield $F \ol K$. 
Note that under the natural map $\CM(F'/K') \to \CM(F_1/K)$, the image of a fully transcendental point is fully transcendental, and a pre-image of a fully transcendental point is fully transcendental. A similar results holds for divisorial points.

We first claim that the following are equivalent:
\begin{enumerate}[(1)]
\item 
the equidistribution $\mu_{\alpha_m,v_1}\to \d\mu_{\OL,v_1}$ over 
$X_{F_{1,v_1}}^\an$
at all $v_1\in \CM(F_1/K)^\ft$ for all finite extensions $F_1$ of $F$;
\item
the equidistribution $\mu_{\alpha_m,v'}\to \d\mu_{\OL,v'}$ 
over $X_{F_{v'}'}^\an$
at all $v'\in \CM(F'/K')^\ft$. 
\end{enumerate}
The equivalence can be proved by the  method of the ``Equivalence'' in  \cite[p. 638]{Yua08}. 
In fact, (2) implies (1) by push-forward of the measures via the natural map $X_{F_{v'}'}^\an\to X_{F_{1,v_1}}^\an$ if $v'$ extends $v_1$.
To have that (1) implies (2), we need to apply the Stone--Weierstrass theorem. 

The above equivalence also holds for divisorial points. 
Therefore, the theorem for $(K, F, F_1)$ is implied by that for $(K', F', F')$.  
This reduces Theorem \ref{approximation high} to Theorem \ref{approximation1}.

\subsubsection{Equidistribution by approximation}

In the above, we have already proved that Theorem \ref{approximation1} implies Theorem \ref{approximation high}. Here we prove that Theorem \ref{approximation high} and Theorem \ref{equi divisorial} imply Theorem \ref{equi ft}. 
We actually prove the following variant, which implies Theorem \ref{equi ft}
by Lemma \ref{equivalence}.

\begin{thm}[equidistribution] \label{equi ft2}
	\kkk
	Let $F$ be a finitely generated field over $k$.
	Let $X$ be a geometrically integral quasi-projective variety over $F$.
	Let $\overline L$ be a nef adelic line bundle on $X/k$ such that 
	$\deg_{\wt L}(X/F)>0$.
	Let $\{x_m\}_m$ be a generic sequence  in $X(\overline F)$  which is numerically small with respect to $\OL$.
	
	If $k=\ZZ$, for every fully transcendental valuation $v\in\CM(F/k)$, the Galois orbit of $\{x_m\}_m$ is equidistributed in $X_v^\an$ for $\d\mu_{\overline L,v}$.
	
        If $k$ is a field, for every valuation $v\in\CM(F/k)$ which is non-trivial and fully transcendental over an intermediate field $K$ of $F/k$ with $\trd(K/k)=1$, the Galois orbit of $\{x_m\}_m$ is equidistributed in $X_v^\an$ for $\d\mu_{\overline L,v}$.
\end{thm}

\begin{proof}[Proof of Theorem \ref{equi ft2} by Theorem \ref{approximation high}]
We start with the setting of Theorem \ref{equi ft2}. 

If $k=\ZZ$, take $K=\QQ$, and denote $v_0=v|_K$. 
If $v_0$ is trivial in this case, then as in the proof of Theorem \ref{equi divisorial}, we can convert the problem to the case $k=\QQ$.
In the following, we assume that $v_0$ is non-trivial on $K$ in this case. 

Hence, in both cases of $k$, the extension
$F/K$ is finitely generated, and the restriction $v_0=v|_K$ is a non-trivial non-archimedean valuation on $K$. 
By assumption, $v\in \CM(F/K)_{v_0}^\ft$ is fully transcendental.

Denote $K'=K_{v_0}$. We will consider the base change of our data via $K\to K_{v_0}$ to transfer them to the local setting over $K'$ to fit the setting of Theorem \ref{approximation high}. 
A complication is that $F\otimes_K K_{v_0}$ is not necessarily a field; however, the image of $F\otimes_K K_{v_0}\to F_v$ is an integral domain $R$, and the fraction field $F'$ of $R$ is a subfield of $F_v$. As $R$ is a finitely generated ring over $K_{v_0}$, its fraction field $F'$ is a finitely generated field over $K_{v_0}$.  
Moreover, $F'$ is endowed with the valuation $v'=v|_{F'}$, and the completion $F'_{v'}=F_v$ as $F'$ contains $F$. 

We claim that $\trd(F'/K')=\trd(F/K)$. 
In fact, $F$ contains a subfield $F_0$ such that $F/F_0$ is a finite extension and $F_0/K$ is a purely transcendental extension. 
Denote $w_0=v|_{F_0}$. 
Then $R=\Im(F\otimes_K K_{v_0}\to F_v)$ is finite over $R_0=\Im(F_0\otimes_K K_{v_0}\to (F_0)_{w_0})$, and thus 
 $F'=\Frac(R)$ is finite over $F_0'=\Frac(R_0)$. 
It suffices to prove $\trd(F_0'/K')=\trd(F_0/K)$. 
This can be checked directly by writing $F_0=K(t_1,\dots, t_d)$.

Denote by $X'=X\times_F F'$ the base change.
Denote by $\OL'=\OL|_{X'}$ the image of $\OL$ under the composition 
$$\HPic(X/k)\lra \HPic(X\times_K K_{v_0}/O_{K_{v_0}})\lra \HPic(X'/O_{K'}).$$  
The compositions are constructed via fiber products of integral models as before, and we omit them. 

Denote by $z_m$ the closed point of $X$ corresponding to $x_m$. 
Denote by $\alpha_m'$ the 0-cycle on $X'$ associated to the base change $z_{m,F'}$, which is still generic. 

Finally, $(K', F', X', \OL', \alpha_m')$ fits the setting of Theorem \ref{approximation high}. 
By Theorem \ref{equi divisorial}, the weak convergence
$$\mu_{\alpha_m',v'}\longrightarrow \d\mu_{\OL',v'}$$
holds for all divisorial valuations $v'\in \CM(F'/K')$.
To prove Theorem \ref{equi ft2} for the fully transcendental $v\in \CM(F'/K')$, by Theorem \ref{approximation high}, 
it suffices to prove that $\{\alpha_m'\}_m$ has bounded $\TL'_{K'}$-heights over $K'$.
Here $\TL'_{K'}$ is the image of $\OL'$ in $\TPic(X'/K')$. 

Denote by $d=\trd(F/K)$. 
By the definition of small sequence, for each $\OH\in \HPic(F/k)_{\nef}$, we have
\[
h_\OL^{\OH}(x_m)=\frac{1}{\deg(x_m)}\OL\cdot\OH^d\cdot z_m \longrightarrow h^{\OH}_\OL(X),\quad m\to\infty.
\]
By the elementary polynomial identity
\[
{d!\cdot t_1\cdots t_d=\sum\limits_{I\subset\{1,2,\cdots,d\}}(-1)^{d-\# I}\left(\sum\limits_{i\in I}t_i\right)^d},
\]
we have that for any  $\OH_1,\cdots,\OH_{d}\in \HPic(F/k)_{\nef}$,
\[
\frac{1}{\deg(x_m)}\OL\cdot\OH_1\cdots\OH_{d}\cdot z_m
\longrightarrow 
\frac{1}{(n+1)\deg_{\wt L}(X/F)}\OL^{\dim X+1}\cdot\OH_1\cdots\OH_d.
\]
In particular, the left-hand side is bounded. 

Take $\OH_d$ to be an adelic line bundle on $K/k$ with $\deg(\OH_d)=1$, and we also view it as an adelic line bundle on $F/k$ via pull-back. 
Then we have 
$$
\frac{1}{\deg(x_m)}\OL\cdot\OH_1\cdots\OH_{d}\cdot z_m
=\frac{1}{\deg(x_m)}\TL_K\cdot\TH_1\cdots\TH_{d-1}\cdot z_m,
$$
where
$\TL_K$ and $\TH_i$ are the images of $\OL$ and $\OH_i$ in 
$\TPic(X/K)$ and $\TPic(F/K)$, respectively. 
This identity can be checked for model adelic line bundles and then take an approximation to get the general case. 

Consider the functorial map 
$$
\TPic(F/K) \lra \TPic(F'/K').
$$
Denote by $\TH_i'$ the image of $\OH_i$ in 
$\TPic(F'/K')$. 
Note that $\TL_{K'}'$ defined above is the image of $\TL_K$ in 
$\TPic(F'/K')$. 
The property $\trd(F'/K')=\trd(F/K)$
implies that projective models of $F'/K'$ have the same dimension as those of $F/K$. 
As a consequence, there is a constant $c>0$ such that 
$$
\frac{1}{\deg(x_m)}\TL_K\cdot\TH_1\cdots\TH_{d-1}\cdot z_m
=c\cdot \frac{1}{\deg(\alpha_m')}\TL_{K'}'\cdot\TH_1'\cdots\TH_{d-1}'\cdot \alpha_m'.
$$
In particular, the right-hand side is bounded. 
Moreover, if we take the original $\OH_1,\cdots, \OH_{d-1}$ to be equal and big on $F/k$, then 
$\TH_1,\dots, \TH_{d-1}$ are equal and big on $F/K$, and thus 
$\TH_1',\dots, \TH_{d-1}'$ are equal and big on $F'/K'$.
This proves that the sequence has bounded $\TL_{K'}'$-height. 
\end{proof}

\subsection{Step 3: Choose special test functions}

In the above, we have reduced our main theorem (Theorem \ref{equi ft}) to the approximation theorem (Theorem \ref{approximation1}). 
In this subsection, we prove that in the weak convergence in Theorem \ref{approximation1}, we only need to consider some special test functions obtained by pull-back from $\PP_K^{1,\an}$. 

Let $K$ be a non-archimedean field. 
Let $X$ be a projective variety over $K$. 
Recall that a \emph{model function} $f:X^\an\to \RR$ is a function of the form 
$-a\log \|1\|_\CL: X^\an \to \RR$, where $a$ is a rational number, and
$\|\cdot\|_\CL$ is the metric of $\CO_X$ on $X^\an$ induced by a projective model 
$(\CX, \CL)$ of $(X, \CO_X)$ over $O_K$. 
Note that model functions are continuous. 

Let $U$ be a quasi-projective variety over $K$. Recall that $C(U^\an)$ denotes the ring of continuous $\RR$-valued functions, and 
 $C_c(U^\an)$ denotes the ring of continuous and compactly supported $\RR$-valued functions. 
 
A \emph{model function} on $U^\an$ is a function of the form $\iota^*f =f\circ \iota$ for an open immersion $\iota:U\hookrightarrow X$ into a projective compactification $X$ over $K$ and a model function $f$ on $X^\an$. 
Denote by
$C^{\m}(U^\an)$ the $\QQ$-vector space of model functions on $U^\an$, 
and denote by
$C_c^{\m}(U^\an)$ the $\QQ$-vector space of compactly supported model functions on $U^\an$.

Denote by 
$C_{P}^{\m}(U^\an)_{\rm pure}$
the set of functions on $U^\an$ of the form
$$
(\psi_1^*f_1)(\psi_2^*f_2)\cdots (\psi_N^*f_N),
$$
where $N\geq 1$ is an integer, every $\psi_i: U\to\PP^1$ is a $K$-morphism, and every $f_i: \PP^{1,\an}\to \RR$ is a model function. 
Denote 
$$C_{P,c}^{\m}(U^\an)_{\rm pure}= C_{P}^{\m}(U^\an)_{\rm pure} \cap C_c(U^\an).
$$
Denote by $C_{P,c}^{\m}(U^\an)$ the $\QQ$-vector subspace of $C(U^\an)$ generated by $C_{P,c}^{\m}(U^\an)_{\rm pure}$.
Then  $C_{P,c}^{\m}(U^\an)$ is actually a $\QQ$-subalgebra of $C(U^\an)$. 

Note that we require $\psi_i: U\to\PP^1$ to be an algebraic $K$-morphism, and if $U$ is projective, there might be few non-constant choices of $\psi_i$. However, 
if $U$ is affine, there are plenty of choices.  
The main result of this subsection is the following density theorem. 

\begin{thm}[special test function] \label{test function}
Let $U$ be an affine variety over a non-archimedean field $K$. Then 
$C_{P,c}^{\m}(U^\an)$ is dense in $C_c(U^\an)$ under the uniform topology.
\end{thm}

\subsubsection{Preliminary results}

We list two preliminary results before proving Theorem \ref{test function}. 

Let $X$ be a projective variety over a non-archimedean field. By the theorem of Gubler \cite[Thm. 7.12]{Gub98} (cf. \cite[p. 643]{Yua08}),
 $C^{\m}(X^\an)$ is dense in $C(X^\an)$ under the uniform topology.
The following result extends it to the quasi-projective setting. 

\begin{thm} \label{density quasi-projective}
Let $U$ be a quasi-projective variety over a non-archimedean field $K$. 
Then any $f\in C_c(U^\an)$ is a uniform limit of a sequence $\{f_m\}_m$
in $C^{\m}(U^\an)$ with $\supp(f_m)\subset \supp(f)$.  
As a consequence, $C^{\m}_c(U^\an)$ is dense in $C_c(U^\an)$ under the uniform topology.
\end{thm}

\begin{proof}
Fix $f\in C_c(U^\an)$. The goal is to approximate it by elements of $C^{\m}(U^\an)$ with smaller supports. Write  
$$f= f_+-f_-, \quad f_+=\frac{1}{2}(|f|+f), \quad f_-=\frac{1}{2}(|f|-f).$$
It suffices to approximate $f_+$ and $f_-$ separately. In other words, we can assume that $f\geq 0$ everywhere. 

Let $X$ be a projective compactification of $U$ over $K$. View $U^\an$ as an open subspace of $X^\an$. 
View $f$ as a continuous function on $X^\an$ by extension by 0.  
We will prove that $f$ can be approximated by elements of 
$C^{\m}(X^\an)\cap C_c(U^\an)$, where the intersection is taken in $C(U^\an)$. 

By Gubler's theorem in the projective case, there is a sequence $\{f_m\}_m$ in $C^{\m}(X^\an)$ approximating $f$ uniformly. 
Then there is a sequence $\{\epsilon_m\}_m$ of positive real numbers converging to 0 such that $|f-f_m|< \epsilon_m$ everywhere. 
By \cite[Lem. 7.8]{Gub98}, 
$$f_m'=\max\{f_m -2\epsilon_m,\ 0 \}$$
is still a model function on $X^\an$. 
Moreover, $f_m'$ converges uniformly to $\max\{f,0\}=f$ by the assumption $f\geq 0$. 
It remains to check the support of $f'_m$. 
In fact, if $x\in \supp(f_m')$, then $f_m(x) \geq2\epsilon_m$, and thus $f(x) \geq\epsilon_m$. 
It follows that $\supp(f_m')$ is contained in $\supp(f)$. This finishes the proof. 
\end{proof}

In the following, we introduce some notation and results on product spaces. 
Let $K$ be a non-archimedean field. 
For every positive integer $N$, let $\AA^{N,\an}$ (resp. $\PP^{1,\an}$, 
$(\PP^1)^{N,\an}$) be the Berkovich spaces of $\AA^N_K$ 
(resp. $\PP^1_K$, $(\PP^1_K)^N$) over $K$. 
Let $(\AA^{1,\an})^N$ (resp. $(\PP^{1,\an})^N$) be the $N$-th power of $\AA^{1,\an}$ (resp. $\PP^{1,\an}$) as topological spaces, which are not necessarily Berkovich spaces. Note that the $i$-th projection $p_i: {\AA^N}\to \AA^1$ induces a morphism 
$$
p_i:\AA^{N,\an} \lra \AA^{1,\an}.
$$
The product gives a \emph{contraction map} 
$$
\tau=(p_1,\dots, p_N):\AA^{N,\an} \lra (\AA^{1,\an})^N.
$$
The map is continuous and surjective, but  it is usually not injective for $N>1$. 
A similar result holds for $(\PP^{1,\an})^N$. 

\begin{lem}\label{properness}
The canonical map $\tau:\AA^{N,\an} \to (\AA^{1,\an})^N$ is proper in the sense that preimages of compact subsets are compact.  
\end{lem}

\begin{proof}
Denote by $x_1,\dots, x_N$ the standard coordinate functions of $\AA^{N}$. 
Denote by 
$$D_N(r)=\CM(K\{r^{-1}T_1,\cdots, r^{-1}T_N\})$$
 the standard closed ball of radius $r>0$ in $\AA^{N,\an}$. 
Note that $\AA^{N,\an}$ is covered by $D_N(r)$, and $(\AA^{1,\an})^N$  
is covered by $(D_1(r))^N$, as $r$ varies. 

Let $S$ be a compact subset of $(\AA^{1,\an})^N$, and we need to prove that 
$\tau^{-1}(S)$ is compact in $\AA^{N,\an}$.
We can assume that $S$ is contained in $(D_1(r))^N$ for some fixed $r$.
As a consequence, $\tau^{-1}(S)$ is contained in $\tau^{-1}((D_1(r))^N)=D_N(r)$. 
As $S$ is closed, $\tau^{-1}(S)$ is closed in $D_N(r)$. 
As $D_N(r)$ is compact, {$\tau^{-1}(S)$} is compact.  
 \end{proof}

\subsubsection{Proof of the density theorem}

Now we are ready to prove the density theorem. 

\begin{proof}[Proof of Theorem \ref{test function}]
We first pass to an affine version of $C_{P,c}^{\m}(U^\an)$. 
Denote by 
$C_{A}(U^\an)_{\rm pure}$
the set of functions on $U^\an$ of the form
$$
(\psi_1^*f_1)\cdots (\psi_N^*f_N),
$$
where $N\geq 1$ is an integer, every $\psi_i: U\to\AA^1$ is a $K$-morphism, and every $f_i\in C_c(\AA^{1,\an})$ is continuous and compactly supported. 
Denote 
$$
C_{A,c}(U^\an)_{\rm pure}= C_{A}(U^\an)_{\rm pure} \cap C_c(U^\an).
$$
Denote by $C_{A,c}(U^\an)$ the $\QQ$-vector subspace of $C_c(U^\an)$ generated by $C_{A,c}(U^\an)_{\rm pure}$, which is actually a $\QQ$-subalgebra of $C_c(U^\an)$. 

By Theorem \ref{density quasi-projective}, every element of $C_c(\AA^{1,\an})$ can be approximated uniformly by elements of $C^{\m}_c(\AA^{1,\an})$ of smaller supports. 
It follows that every element of $C_{A,c}(U^\an)_{\rm pure}$ can be approximated  
uniformly by elements of $C_{P,c}^{\m}(U^\an)_{\rm pure}$. 
Therefore, it suffices to prove that $C_{A,c}(U^\an)$ is dense in $C_c(U^\an)$ under the uniform topology.
It suffices to check the conditions in the Stone--Weierstrass theorem in Theorem \ref{S-W for compactly support}.
	
First, $C_{A,c}(U^\an)$ vanishes nowhere; i.e., for any $x\in U^\an$, there exists $g\in C_{A,c}(U^\an)$ such that $g(x)\neq 0$. 
Take a closed immersion $\iota:U\emb\AA^N$, which induces a closed immersion $\iota:U^\an\emb\AA_K^{N,\an}$. Compose it with the contraction map 
$\tau: \AA^{N,\an}\to (\AA^{1,\an})^N.$
There exist functions 
$$f=(p_1^* f_1)\cdots (p_N^*f_N) \in C_c((\AA^{1,\an})^N), \quad 
f_i\in {C_c(\AA^{1,\an})},$$
such that $f((\tau\circ \iota)(x))\neq 0$. 
Then $g=(\tau\circ \iota)^* f$ satisfies the requirement. 
In fact, by Lemma \ref{properness}, $\tau$ is proper, and thus $\tau\circ \iota$ is also proper. It follows that $(\tau\circ \iota)^* f$ is still compactly supported, and thus lies in $C_{A,c}(U^\an)$.

Second, $C_{A,c}(U^\an)$ separates points. 
Let $x, y$ be two different points in $U^\an$, and we need to find $g\in C_{A,c}(U^\an)$ with $g(x)\neq g(y)$. 
Note that $U=\Spec R$ is affine. 
By definition, there exists $h\in R$ such that $|h|(x)\neq |h|(y)$. 
Take a closed immersion $\iota:U\emb\AA^N$. By taking the product with 
$h:U\to \AA^1$ (and replacing $N$ by $N+1$), we can assume that the composition of $\iota:U\emb\AA^N$ with the first projection $p_1:\AA^N\to \AA^1$
is equal to $h:U\to \AA^1$. 

As above, consider the composition of $\iota:U^\an\emb\AA_K^{N,\an}$ with the contraction map 
$\tau: \AA^{N,\an}\to (\AA^{1,\an})^N.$
We claim that there exist functions 
$$f=(p_1^* f_1)\cdots (p_N^*f_N) \in C_c((\AA^{1,\an})^N), \quad 
f_i\in {C_c(\AA^{1,\an})},$$
such that $f((\tau\circ \iota)(x))\neq f((\tau\circ \iota)(y))$. 
In fact, denote by $x_i, y_i$ the images of $x,y$ under the composition 
$$U^\an\stackrel{\iota}\lra \AA^{N,\an} \stackrel{\tau}\lra (\AA^{1,\an})^N \stackrel{p_i}\lra \AA^{1,\an}.$$
The condition $|h|(x)\neq |h|(y)$ implies that $x_1\neq y_1$ in $\AA^{1,\an}$. 
Then we can make $f_1(x_1)\neq f_1(y_1)$, and make
$f_i(x_i)= f_i(y_i)=1$ for every $i\neq 1$. 
Then $g=(\tau\circ \iota)^* f$ satisfies the requirement. 
\end{proof}

\subsection{Step 4: Reduce to a uniform H\"older property}

Previously, we have reduced our main theorem (Theorem \ref{equi ft}) to an approximation theorem (Theorem \ref{approximation1}).
In this subsection, we introduce a uniform H\"older property, and see how it implies the approximation theorem. 

\subsubsection{Distance function}

To start with, we first introduce a canonical \emph{distance function} on a Berkovich curve.
Let $\Gamma$ be a metrized graph. The distance $\dist_\Gamma(x,y)$ of two points $x,y\in\Gamma$ is the infimum of the lengths of paths from $x$ to $y$. 

\begin{definition} \label{defn distance}
Let $X$ be a smooth projective curve over an algebraically closed non-archimedean field $K$. For two points $x,y\in X^\an$, we define their \emph{canonical distance}
as
\[
\dist(x,y)=\sup_{\CX}\dist_{S(\CX)}(\tau(x),\tau(y)).
\]
Here $\CX$ runs through all strictly semi-stable models of $X$ over $O_K$, and $\tau:X^\an\to S(\CX)$ is the retraction map.
\end{definition}
It is easy to see that if $x,y\in S(\CX)$ for some $\CX$, then
\[
\dist(x,y)=\dist_{S(\CX)}(x,y).
\]
The following theorem concerns the finiteness of the canonical distance function.

\begin{thm} \label{finite distance}
\begin{enumerate}[(1)]
\item If $x,y$ are points of types $2,3,4$, then $\dist(x,y)<\infty$.
\item
If $x$ is a point of type $2,3,4$, and $\{x_m\}_m$ is a sequence of points of type $2,3,4$ convergent to $y$, then 
\[
\lim_{m\to\infty}\dist(x_m,x)=0.
\]
\end{enumerate}

\end{thm}
\begin{proof}
By Baker--Payne--Rabinoff \cite[Cor. 5.7]{BPR}, the distance function extends to a finite-valued metric on $\mathbf H(X^{\an})=X^{\an}\setminus X(K)$, the set of points of type $2,3,4$. 
\end{proof}

\subsubsection{Uniform H\"older property}

To prove Theorem \ref{approximation1}, the key is the following H\"older property. We will first state the theorem, and then explain the notation. 

\begin{thm}[uniform H\"older property]\label{holder}
Let $K$ be an algebraically closed non-archimedean field. 
Let $Y$ be a smooth projective curve over $K$. 
Let $\pi:U\to Y$ be a flat $K$-morphism from a quasi-projective variety $U$ over $K$. 
Denote by $F$ the function field of $Y$, and denote by $X\to \Spec F$ the generic fiber of $\pi$, which is assumed to be geometrically integral. 
Let $\TL\in \wt\Pic(U/K)$ be a big adelic line bundle on $U/K$.
Let $g\in C_{P,c}^\m(U^\an)$ be a test function on $U^\an$. 
Then there exist a Zariski open subset $X'$ of $X$ and positive constants $A_0, A_1$ such that for any closed point $x\in X'$ and any two points $v, w\in Y^\an$ of type $2,3,4$, we have
	$$
	\left|\frac{1}{\deg(x)} g(x_{v})-\frac{1}{\deg(x)} g(x_{w})\right|
	\leq
	A_1\cdot\sqrt{h_{\TL}(x)+A_0}\cdot\sqrt{\dist(v,w)}.
	$$
\end{thm}

We explain the notation as follows. 
As $F$ is a function field of one variable over $K$, we have a height function 
$$
h_\TL: |X|_0\lra \RR, \quad x\longmapsto \frac{1}{\deg(x)} \wh\deg(\TL|_x),
$$
where $|X|_0$ denotes the set of closed points of $X$. 

Now we explain the term $g(x_{v})$.  
For any closed point $x\in X$ and any valuation $v\in \CM(F/K)$, the base change $x_v=x_{F_v}=x\times_F F_v$ is a closed subscheme of $X_{F_v}$, and thus defines an effective $0$-cycle $[x_v]$ on $X_{F_v}$. It is also viewed as a linear combination of classical points of $X_v^\an$. 
If $[x_v]=\sum_{i}a_i y_i$ for closed points $y_i\in X_{F_v}$, then
$$
g(x_{v})=\sum_{i}a_i \deg(y_i)\, g(y_i).  
$$
By the notation of Theorem \ref{approximation high}, we simply have 
$$
\frac{1}{\deg(x)} g(x_{v})=\int_{X_v^\an} g \mu_{x,v}.  
$$

Note that the statement of the theorem also implies that $h_{\TL}(x)+A_0\geq 0$
for $x\in X'$. 
The theorem implies that the function 
$$
v\longmapsto \frac{1}{\deg(x)} g(x_{v})
$$
of $v\in Y^\an$ is H\"older continuous if $v$ is of type $2,3,4$. Moreover, 
if $h_{\TL}(x)$ is bounded, then the function is uniformly continuous.

Now let us see how Theorem \ref{holder} implies Theorem \ref{approximation1}. 

\begin{proof}[Proof of Theorem \ref{approximation1} by Theorem \ref{holder}]
Let $(K, F, X, \OL, \TL, \alpha_m)$ be as in Theorem \ref{approximation1}. 
Assume that the weak convergence
	$$\mu_{\alpha_m,v}\longrightarrow \d\mu_{\OL,v}$$
holds for all divisorial valuations $v\in \CM(F/K)$, and we need to prove that it holds for all fully transcendental valuations $v\in\CM(F/K)^\ft$.

By a partition of unity, we may replace $X$ by {an open} subvariety, and thus assume that $X$ is {affine}. In this process, we might need to remove the parts of $\alpha_m$ that are supported at the complement of the subvariety.
This does not affect the results since $\alpha_m$ is generic. 
Hence, in what follows, we assume that $X$ is an affine variety over $F$. Moreover, it suffices to verify the equidistribution condition for functions in $C_c(X_v^\an)$.

We may take an affine variety $U$ over $K$ and a flat morphism $\pi:U\to Y$ with generic fiber $X\to \Spec F$. By shrinking $U$ if necessary, we may assume that $\OL\in\HPic(U/O_K)$ and thus the geometric part $\TL\in \wh\Pic(U/K)$.

Fix $v\in\CM(F/K)^\ft$. We need to prove  
$$
\lim_{m\to\infty}\frac{1}{\deg(\alpha_m)} f(\alpha_{m,v})= \int_{X_{v}^\an}f\, \d\mu_{\OL, v}
$$
for every $f\in C_c(X_v^\an)$.
Note that $X_v^\an=U_v^\an$ is the fiber of $U^\an\to Y^\an$ above $v\in Y^\an$. 
By Lemma \ref{density fiber}, we can assume that the test function $f$ can be extended to a test function $\wt f\in C_c(U^\an)$.
By Theorem \ref{test function}, it suffices to check that for all test functions $g\in C_{P,c}^\m(U^\an)$, we have
$$
\lim_{m\to\infty}\frac{1}{\deg(\alpha_m)} g(\alpha_{m,v})= \int_{X_{v}^\an}g\, \d\mu_{\OL, v}.
$$
	
To approximate our point, we need the density of divisorial points in Berkovich spaces. For this well-known fact, we refer to \cite[Cor. 2.4]{BFJ15} or \cite[Cor. 4.5]{Poi13} (see also \cite[Prop. A.5]{GM}). The case of Berkovich curves can be seen from the skeleton. 
	
By assumption, $v$ is fully transcendental, so $v$ is a point of type $2,3,4$. There is a sequence of divisorial points in $Y^\an$, such that $v_j\to v$ when $j\to\infty$. 
By Lemma \ref{finite distance}, we have 
\[
\lim_{j\to\infty}\dist(v_j,v)=0.
\]
Note that this is the reason why our proof does not work for type $1$ points. 

By the equidistribution at $v_j$, we have
$$
\lim_{m\to\infty}\frac{1}{\deg(\alpha_m)}g(\alpha_{m,v_j})= \int_{X_{v_j}^\an}g\, \d\mu_{\OL, v_j}. 
$$
	Now set $j\to\infty$. By Lemma \ref{continuity lem}, we have 
$$
\int_{U_{v_j}^\an} g c_1(\OL)^n_{v_j}\longrightarrow \int_{U_{v}^\an} g c_1(\OL)^n_{v}$$
and thus
$$
 \int_{X_{v_j}^\an}g\, \d\mu_{\OL, v_j} \lra  \int_{X_{v}^\an}g\, \d\mu_{\OL, v}. 
	$$
Thus it suffices to prove that
	$$
	\lim_{j\to\infty}\frac{1}{\deg(\alpha_m)}g(\alpha_{m,v_j})=
	\frac{1}{\deg(\alpha_m)}g(\alpha_{m,v})
	$$
uniformly on $m$. 
	
By the H\"older property in Theorem \ref{holder},
 there exists a constant $A_1$ such that
for any closed point $x\in X'$,  
and  for any two points $v$ and $w$ of type $2,3,4$,
	$$
	\left|\frac{1}{\deg(x)}g(x_{v})-\frac{1}{\deg(x)}g(x_{w})\right|
	\leq
	A_1\cdot\sqrt{h_{\TL}(x)+A_0}\cdot\sqrt{\dist(v,w)}.
	$$
Here $X'$ is an open subvariety of $X$, but we can assume that $X'=X$ by shrinking $X$. 

If every $\alpha_m$ is a closed point of $X'$, then the H\"older property implies the uniform convergence by the assumption of the boundedness of the $\TL$-heights. 

In general, it suffices to generalize the H\"older property to effective 0-cycles $\alpha$ on $X$. In fact, write $\alpha=\sum_{i=1}^r a_i x_i$ with closed points $x_i\in X$. We have
	\begin{align*}
	&\ \ \ \ \frac{1}{\deg(\alpha)}\left|g(\alpha_{v})-g(\alpha_{w})\right|\\
	&\leq \sum_{i=1}^r\frac{a_i \deg(x_i)}{\deg(\alpha)}\left|\frac{1}{\deg(x_i)}g(x_{i,v})-\frac{1}{\deg(x_i)}g(x_{i,w})\right|\\
	&\leq A_1\sqrt{\dist(v,w)}\cdot \sum_{i=1}^r\frac{a_i \deg(x_i)}{\deg(\alpha)}\sqrt{h_{\TL}(x_i)+A_0}\\
	&\leq A_1\sqrt{\dist(v,w)}\cdot\left(\sum_{i=1}^r\frac{a_i\deg(x_i)}{\deg(\alpha)}\right)^{\frac{1}{2}}\left(\sum_{i=1}^r\frac{a_i \deg(x_i)}{\deg(\alpha)}(h_{\TL}(x_i)+A_0)\right)^{\frac{1}{2}}\\
	&=A_1\sqrt{\dist(v,w)}\cdot\sqrt{h_\TL(\alpha)+A_0}.
		\end{align*}
This proves the H\"older properties for effective 0-cycles. 
Hence, we have proved Theorem \ref{approximation1} by Theorem \ref{holder}. 
\end{proof}

\subsection{Step 5: Reduce to a distance inequality}
\label{sec proof holder}

In the above, we have reduced our main theorem (Theorem \ref{equi ft}) to the H\"older property (Theorem \ref{holder}). 
In this subsection, we reduce the H\"older property
to a distance inequality (Theorem \ref{distance}).

\subsubsection{Distance inequality}
To state the distance inequality, we start with a series of notation. 

Let $K$ be an algebraically closed non-archimedean field. 
Let $Y$ be a smooth projective curve over $K$. 
Let $\pi:U\to Y$ be a flat $K$-morphism from a projective variety $U$ over $K$. 
Denote by $F$ the function field of $Y$, and denote by $X\to \Spec F$ the generic fiber of $\pi$, which is assumed to be geometrically integral. 
Let $\TL\in \wt\Pic(U/K)$ be a big adelic line bundle on $U/K$.
For $i=1,\dots,N$, let $\psi_i:U\to\PP^1$ be surjective $K$-morphisms. 

This notation is the same as that in Theorem \ref{holder}, except that we require $U$ to be projective for convenience, which can be obtained by a compactification and a blowing-up. 
Note that we do not need the test function $g$ here, but we need to choose integral models and work on skeletons in the following. 

Let $x\in X$ be a closed point, which will vary later. 
Denote by $Y'$ the normalization of the Zariski closure $\triangle(x)$ of $x$ in $U$, which is a smooth projective curve over $K$. Denote by $\pi':Y'\to Y$ the composition $Y'\to \triangle(x)\to U \to Y$, which is a finite morphism of degree $\deg(\pi')=\deg(x)$.

Let $\CP_i$ be a strictly semi-stable model  of $\PP^1_K$ over $O_K$.
Let $\CY$ be a strictly semi-stable model of $Y$ over $O_K$.
Take a strictly semi-stable model $\CY'$ of $Y'$ over $O_K$ such that $\psi_i':Y'\to\PP^1$ extends to a morphism $\psi_i':\CY'\to\CP_i$ and $\pi':Y'\to Y$ extends to a morphism $\pi':\CY'\to\CY$.

Denote the skeletons
$$\Gamma'=S(\CY'),\quad
\Gamma=S(\CY), \quad
\Gamma_i=S(\CP_i), \quad
\Gamma^{(N)}=\Gamma_1\times \cdots \times \Gamma_N.$$
 By \cite[Lem. 5.6.4]{GJR2},  the induced maps 
 \[
 \pi_\sk':\Gamma'\longrightarrow\Gamma, \quad
\psi_{i,\sk}':\Gamma'\longrightarrow\Gamma_i
\]
are piecewise linear on every edge.

Let $v \in \Gamma$ be a point, also viewed as a point of $Y^\an$. 
The normalization map $Y'\to \triangle(x)$ is an isomorphism outside finitely many closed points, and thus the induced map $Y'^\an\to\triangle(x)^\an$ is an isomorphism outside finitely many points of type 1. In particular, it is an isomorphism above $v\in Y^\an$, since $v$ is of type $2,3$. 
Similarly, outside finitely many closed points, the finite morphism $Y'\to Y$ is the composition of an unramified morphism with a relative Frobenius morphism, and $v$ is different from classical points related to the images of these finitely many closed points.  

Recall the notion of effective 0-cycles in \S\ref{sec lipschitz}. 
Denote by $F'$ the function field of $Y'$. 
Recall that $x_v=x\times_F F_v=\Spec(F'\otimes_F F_v)$ defines an effective $0$-cycle $[x_v]$ on $X_v^\an=U_v^\an$. 
We can also {view} it as an effective $0$-cycle on 
$\triangle(x)^\an$ or $Y'^\an$, as its support lies in $\triangle(x)^\an$.  
By Lemma \ref{separable}, $x_v=\Spec(F'\otimes_F F_v)$ is always a reduced scheme, and thus the multiplicity of every point of $x_v$ is 1. 
Therefore, we can write 
$$
[x_v]=\sum_{v'\in \triangle(x)^\an,\ \pi(v')=v} v'=\sum_{v'\in Y'^\an,\ \pi'(v')=v} v'.
$$

By push-forward via $\tau_{Y'}: Y'^\an \to \Gamma'$, 
we define an effective $0$-cycle on $\Gamma'$ as the push-forward
$$
[x_v]_{\Gamma'}= \sum_{v'\in Y'^\an,\ \pi'(v')=v} [\CH(v'):\CH(v)]\ \tau_{Y'}(v').
$$
Note that we have put the degree $[\CH(v'):\CH(v)]$ as the weight before $v'$.
This gives
$$
\deg([x_v]_{\Gamma'})
= \sum_{v'\in Y'^\an,\ \pi'(v')=v} [\CH(v'):\CH(v)]
=\deg(\pi':Y'\to Y)
=\deg(x). 
$$
By further push-forward via $\psi_{i}': Y' \to \PP^1$, 
we define an effective $0$-cycle on $\Gamma_i$ as 
$$
[x_v]_{\Gamma_i}= \sum_{v'\in Y'^\an,\ \pi'(v')=v} [\CH(v'):\CH(v)]\ {\psi'_{i,\sk}(\tau_{Y'}(v'))}.
$$
More precisely, the final point in this formula is $\tau_{\CP_i}(\psi_i'(v'))=\psi'_{i,\sk}(\tau_{Y'}(v'))\in\Gamma_i$.  We also define the effective cycle on the product skeleton by
\[
[x_v]_{\Gamma^{(N)}}=
\sum_{v'\in Y'^\an,\ \pi'(v')=v}[\CH(v'):\CH(v)]
\big(\psi'_{1,\sk}(\tau_{Y'}(v')),\ldots,
\psi'_{N,\sk}(\tau_{Y'}(v'))\big).
\]

Note that $\Gamma, \Gamma', \Gamma_i$ are metrized graphs, and thus induce canonical distance functions. Then $\Gamma^{(N)}$ has the box distance induced from $\Gamma_i$. Recall the notion of distance between two effective 0-cycles in \S\ref{sec lipschitz}. 
Finally, we have the following distance inequality. 

\begin{thm}[distance inequality] \label{distance}
There exists a Zariski open subset $X'$ of $X$ and positive constants $A_0, A_2$ such that for any closed point $x\in X'$, and for any two points $v, w\in \Gamma$, 
we have
$$
\frac{1}{\deg(x)} \dist_{\Gamma^{(N)}}([x_v]_{\Gamma^{(N)}},[x_w]_{\Gamma^{(N)}})
\leq A_2\cdot\sqrt{h_{\TL}(x)+A_0}\cdot \sqrt{\dist_\Gamma(v,w)}.
$$
Moreover, the datum $(X', A_0, A_2)$ is independent of the choices of the integral models $(\CY, \CY')$ of $(Y, Y')$.
\end{thm}

\subsubsection{From the distance inequality to the H\"older property}

It is easy to prove the H\"older property by the distance inequality. 

\begin{proof}[Proof of Theorem \ref{holder} by Theorem \ref{distance}]
The conclusion of Theorem \ref{holder}, which we need to prove here,  is the inequality
	$$
	\left|\frac{1}{\deg(x)} g(x_{v})-\frac{1}{\deg(x)} g(x_{w})\right|
	\leq
	A_1\cdot\sqrt{h_{\TL}(x)+A_0}\cdot\sqrt{\dist(v,w)}.
	$$
By linear combination, we can assume $g\in C_{P,c}^\m(U^\an)_{\rm pure}$, and thus we can write
$$g=(\psi_1^* f_1) \cdots (\psi_N^* f_N),$$
 where
$\psi_i: U \to \PP^1_K$ is a $K$-morphism and $f_i$ is a model function on 
$\PP^{1,\an}$. 
Assume that $U$ is projective by compactification and blowing-up. 

As in Theorem \ref{distance}, 
take $Y'$ to be the normalization of $\triangle(x)$, 
take integral models $\CY, \CY', \CP_i$ , and take 
the skeletons $\Gamma, \Gamma', \Gamma_i, \Gamma^{(N)}$. 

We further assume that the model function $f_i$ of $\PP^{1,\an}$ is realized on the integral model $\CP_i$ by choosing $\CP_i$ more carefully. 
We will fix $\CP_i$, but we will vary $\CY$ and $\CY'$. 
View $f_i$ as a function on $\Gamma_i=S(\CP_i)$ by restriction. This gives a function 
$$f=(p_1^*f_1) \cdots (p_N^* f_N)$$ 
on $\Gamma^{(N)}=\Gamma_1\times\cdots \times \Gamma_N$. 

Let $v,w\in \Gamma$ be two points,  viewed as points of $Y^\an$. 
By definition, we have
$$
g(x_{v})=f([x_v]_{\Gamma^{(N)}}), \quad g(x_{w})=f([x_w]_{\Gamma^{(N)}}).
$$
Note that $f_i$ is piecewise linear on $\Gamma_i$, and thus is bounded and satisfies the Lipschitz property. 
By Lemma \ref{lipschitz2} and Lemma \ref{lipschitz1}, $f$ satisfies the Lipschitz property on $\Gamma^{(N)}$ and the resulting evaluation on effective cycles of fixed degree satisfies the corresponding transport-distance bound. 
Then there exists $C$ such that
$$
|f([x_v]_{\Gamma^{(N)}})-f([x_w]_{\Gamma^{(N)}})| \leq C\cdot 
\dist_{\Gamma^{(N)}}([x_v]_{\Gamma^{(N)}},[x_w]_{\Gamma^{(N)}}). 
$$
It follows that Theorem \ref{distance} implies Theorem \ref{holder}
for any $v, w\in \Gamma$.

It is easy to extend Theorem \ref{holder}
for all $v, w\in \Gamma$ to all $v, w\in Y^\an$ of types 2,3,4.
In fact,  in the above proved case for all $v, w\in \Gamma$, the  constants $(A_0, A_1)$ are independent of the choices of 
$(\CY, \CY', v,w)$. 
Since any pair of points of types 2,3 can lie in the skeleton for some $\CY$, we see that the theorem with the same constants $(A_0, A_1)$  actually holds for all $v, w\in Y^\an$ of types 2,3.
Note that $g(x_v)$ is continuous in $v\in Y^\an$, which can be proved directly or viewed as a special case of Theorem \ref{continuity lem} of relative dimension 0. 
Approximating type-$4$ points by type-$2$ points in the canonical distance and using continuity, the theorem extends to all $v, w\in Y^\an$ of types 2,3,4.
\end{proof}

\subsection{Step 6: Prove the distance inequality}
\label{sec proof distance}

In the above, we have reduced our main theorem (Theorem \ref{equi ft}) to a distance inequality (Theorem \ref{distance}). 
In this subsection, we prove the distance inequality. 
The source of the distance inequality comes from a height inequality of Yuan--Zhang \cite{YZ26} and an estimate of partial heights of Xie--Yuan \cite{XY}. 
Note the situation is much easier to explain if $\deg(x)=1$, so we will first treat this case to illustrate the idea.

\subsubsection{Height inequality}
Assume the setting of Theorem \ref{distance}.

Note that $\TL$ is big on the generic fiber $X\to\Spec F$. Apply the height inequality of \cite[Thm. 5.3.7]{YZ26}. Then there is a Zariski dense open subset $X'$ of $X$, together with constants $A_0, A_3>0$, 
 such that for all closed points $x\in X'$,
$$
h_{\psi_i^*\CO(1)}(x)\leq A_3(h_{\TL}(x)+A_0),\quad \forall\, 1\leq i\leq N.
$$
Note that this is the \emph{only part} we need to pass to a Zariski open subset $X'$ of $X$.

For edges $e'\in E(\Gamma')$, $e\in E(\Gamma)$, and $\bar e\in E(\Gamma_i)$, we will use $t_{e'}$, $t_e$, and $t_{\bar e}$ to denote  their coordinate functions respectively.

Take the differential form 
$$\omega_0:=(\d'  t_{\bar e}\wedge\d'' t_{\bar e})_{\bar e\in E(\Gamma_i)}
=\sum_{\bar e\in E(\Gamma_i)} \d' t_{\bar e}\wedge\d''t_{\bar e}$$ 
on $\Gamma_i$. It corresponds to the Lebesgue measure on every edge.
It is piecewise smooth and strictly positive. 
By Corollary \ref{strictly positive measure on skeleton}, there exists a weakly smooth metric $\|\cdot\|$ of $\CO(1)$ on $\PP^1_K$ such that 
$c_1(\overline{\CO(1)})-\epsilon\omega_0$ is strictly positive on $\Gamma_i$ for some $\epsilon>0$. 
Then
$$
\epsilon \int_{\Gamma'}\psi_i'^*\omega_0
\leq \int_{\Gamma'}\psi_i'^*c_1(\overline{\CO(1)})
\leq \int_{(Y')^\an}\psi_i'^*c_1(\overline{\CO(1)})
=\deg_{Y'}(\psi_i'^*\CO(1))
=\deg(x)\cdot h_{\psi_i^*\CO(1)}(x).
$$
Note that the map $\Gamma'\to \Gamma_{i}$ is linear on edges, and the pull-back $\psi_i'^*\omega_0$ is defined on every edge as recalled in \S\ref{subsec morphism}. Then the first inequality holds by checking at every edge.

Combining with the height inequality, we have a height inequality
$$
\frac{1}{\deg(x)} \int_{\Gamma'}\psi_i'^*\omega_0
\leq \epsilon^{-1}A_3(h_{\TL}(x)+A_0)
$$
for all {closed points} $x \in X'$.

\subsubsection{Special case: degree 1}

To illustrate the idea, we first consider the case $\deg(x)=1$. 
Then $Y'\to Y$ is an isomorphism, and $Y'$ is a section of $\pi:U\to Y$.
We also assume that $\CY'\to \CY$ is also an isomorphism, which can be achieved by replacing $\CY$ by $\CY'$. 
Then $\pi'_\sk:\Gamma'\to \Gamma$ is an isomorphism. By subdivision, we can assume that $\psi'_{i,\sk}:\Gamma'\to \Gamma_i$ is linear on every edge of $\Gamma'$. 

In this case, $[x_v]_{\Gamma'}$ and $[x_w]_{\Gamma'}$ are just points of 
$\Gamma'$.
Thus we simply have 
$$
\dist([x_v]_{\Gamma^{(N)}},[x_w]_{\Gamma^{(N)}})= \sum_{i=1}^N
\dist([x_v]_{\Gamma_i},[x_w]_{\Gamma_i}). 
$$
It suffices to bound each $\dist([x_v]_{\Gamma_i},[x_w]_{\Gamma_i})$.
 
Let $\ell$ be a path from $v$ to $w$ in $\Gamma$ with length $l(\ell)=\dist(v,w)$.  
Assume that the path $\ell$, together with the direction, is given by a union of 
$\ell_1, \cdots, \ell_m$, where $\ell_j$ is a directed closed line segment contained in an edge $e_j$ of $\Gamma$, and $\ell_j\cap \ell_{j+1}$ is a vertex of $e_j$. 
This gives 
$$
\dist(v,w)=l(\ell)=\sum_{j=1}^m l(\ell_j).
$$

Denote by $\ell_j'$ the preimage of $\ell_j$ in $\Gamma'$.  
Then $\ell'=\ell_1'\cup\cdots \cup \ell_m'$ is a path from $[x_v]_{\Gamma'}$ to $[x_w]_{\Gamma'}$. 
This gives 
$$
\dist([x_v]_{\Gamma'},[x_w]_{\Gamma'})=l(\ell')=\sum_{j=1}^m l(\ell_j').
$$
Each term is equal to the corresponding term in  the above expression for $\dist(v,w)$.

Denote by $\bar\ell_j$ the image of $\ell_j'$ in $\Gamma_i$.
This gives a path $\bar\ell=\bar\ell_1\cup\cdots \cup \bar\ell_m$
from $[x_v]_{\Gamma_i}$ to $[x_w]_{\Gamma_i}$. 
Each $\bar\ell_j$ is contained in an edge $\bar e_j$ of $\Gamma_i$. 
Then we have 
$$
\dist([x_v]_{\Gamma_i},[x_w]_{\Gamma_i})
\leq l(\bar \ell)=\sum_{j=1}^m l(\bar \ell_j)=\sum_{j=1}^m d_{e_j'}(\psi_i') l(\ell_j').
$$
Here the degree 
$$d_{e_j'}(\psi_i')=\frac{l(\bar e_j)}{l(e_j')}=\frac{l(\bar \ell_j)}{l(\ell_j')}
=\left|\frac{\d t_{\bar e_j}}{\d t_{e_j'}}\right|$$
 is the expansion factor of the map $e_j'\to \bar e_j$, as recalled in \S\ref{subsec morphism}.

Recall that the differential form 
$$\omega_0=\sum_{\bar e\in E(\Gamma_i)} \d' t_{\bar e}\wedge\d''t_{\bar e}.$$ 
As described in \S\ref{subsec morphism}, the pull-back 
$$\psi_i'^*\omega_0=\sum_{e'\in E(\Gamma')}
d_{e'}(\psi_i')^2 \d't_{e'}\wedge\d'' t_{e'}.$$ 
For a line segment $\rho'$ in an edge $e'$ of $\Gamma'$, we have
$$
\int_{\rho'} \psi_i'^*\omega_0= \int_{\rho'} d_{e'}(\psi_i')^2 \d' t_{e'}\wedge\d'' t_{e'}= 
d_{e'}(\psi_i')^2 l({\rho'}). 
$$
Then we have 
 $$
 \int_{\Gamma'}\psi_i'^*\omega_0
 \geq  \int_{\ell'}\psi_i'^*\omega_0
= \sum_{j=1}^m   \int_{\ell_j'}\psi_i'^*\omega_0
= \sum_{j=1}^m   d_{e_j'}(\psi_i')^2l(\ell_j').$$

As we have seen above, the integral $\displaystyle\int_{Y'^\an} c_1(\overline\CO(1))$ gives 
the height of $x$ (up to the degree), which integrates over the whole space $Y'^\an$.
The integral $\displaystyle\int_{\ell'}\psi_i'^*\omega_0$ integrates on a substantially smaller subset $\ell'$ of $Y'^\an$, and it is a \emph{partial height} of $x$ from the philosophy of Xie--Yuan \cite{XY}. 

Combining the integral inequality and the height inequality, we have
 $$
\sum_{j=1}^m   d_{e_j'}(\psi_i')^2l(\ell_j')
\leq \epsilon^{-1}A_3(h_{\TL}(x)+A_0).
$$
Recall that we have  
$$
\dist(v,w)=\sum_{j=1}^m l(\ell_j'),
$$
and
$$
\dist([x_v]_{\Gamma_i},[x_w]_{\Gamma_i})
\leq \sum_{j=1}^m d_{e_j'}(\psi_i') l(\ell_j').
$$
Then the Cauchy--Schwarz inequality gives 
$$
\dist([x_v]_{\Gamma_i},[x_w]_{\Gamma_i})^2
\leq \left(\sum_{j=1}^m  l(\ell_j') \right)
\left( \sum_{j=1}^m d_{e_j'}(\psi_i')^2 l(\ell_j') \right)
\leq \dist(v,w)\cdot \epsilon^{-1}A_3(h_{\TL}(x)+A_0).
$$
This proves Theorem \ref{distance} in the case $\deg(x)=1$.

\subsubsection{General case}

Now we prove the general case $\deg(x)\geq1$.
We need the notions of distances and transport plans between effective 0-cycles in \S\ref{sec lipschitz}, as well as compositions of transport plans in \S\ref{subsec transport plan} and push-forward of transport plans in \S\ref{subsec triangle inequality}.

We will choose a canonical transport plan $M$ from $[x_v]_{\Gamma'}$ to 
$[x_w]_{\Gamma'}$, which  induces a transport plan $M$ from $[x_v]_{\Gamma_i}$ to $[x_w]_{\Gamma_i}$ by push-forward. 
From \S\ref{sec lipschitz},  we have
$$
\dist([x_v]_{\Gamma^{(N)}},[x_w]_{\Gamma^{(N)}})
\leq \dist([x_v]_{\Gamma^{(N)}},[x_w]_{\Gamma^{(N)}},M)
= \sum_{i=1}^N
\dist([x_v]_{\Gamma_i},[x_w]_{\Gamma_i}, M). 
$$
Then it is reduced to prove 
$$
\frac{1}{\deg(x)}
\dist([x_v]_{\Gamma_i},[x_w]_{\Gamma_i}, M) 
\leq A_2\cdot\sqrt{h_{\TL}(x)+A_0}\cdot \sqrt{\dist(v,w)}.
$$

In the following proof, most efforts are devoted to clarify the transport plan $M$
from $[x_v]_{\Gamma'}$ to  $[x_w]_{\Gamma'}$.
Once we have the transport plan, the other parts of the proof are very similar to the case of $\deg(x)=1$.

Let $\ell$ be a (directed) path from $v$ to $w$ in $\Gamma$ with length $l(\ell)=\dist(v,w)$.
Recall that $\pi_{\sk}':\Gamma'\to \Gamma$ is linear on every edge. 
Then we can write the path $\ell$  as a union of 
$\ell_1, \cdots, \ell_m$, where $\ell_j$ is a directed closed line segment with positive length contained in an edge $e_j$ of $\Gamma$, and $\ell_j\cap \ell_{j+1}$ is a point of $e_j$,
such that no points of $\pi_{\sk}'(V(\Gamma))$ lie in the interior of $\ell_j$. 

An edge of $\Gamma'$ is \emph{exceptional} if its image in $\Gamma$ is a point. 
{Removing the interiors of all exceptional edges of $\Gamma'$ and all isolated vertices of the remaining graph, we obtain} a subgraph $\Gamma''$ of $\Gamma'$ and endow it with the induced metric.
The induced map $\pi_{\sk}'':\Gamma''\to \Gamma$ has no exceptional edges. 
We will see that $\Gamma''$ contains all the information from $\Gamma'$ for our purpose. 

By the choices, we can write the preimage  (for every $j$) 
$$\pi_{\sk}''^{-1}(\ell_j)=\cup_{k=1}^{n_j} \ell_{j,k}',$$ 
where every $\ell_{j,k}'$ is a closed line segment of an edge of $\Gamma''$ and is mapped bijectively to $\ell_j$.
The interiors of different $\ell_{j,k}'$ do not intersect with each other.  
Then $\ell_{j,k}'$ has a direction induced from that of $\ell_j$. 

For a directed line segment $\beta$ in a graph, denote by $\beta^-$ (resp. $\beta^+$) the left endpoint (resp. right endpoint) of $\beta$. 

We introduce two effective 0-cycles on $\Gamma'$ by 
$$
\alpha_{j}^-=\sum_{k=1}^{n_j} [\CH((\ell_{j,k}')^-):\CH(\ell_j^-)]\cdot (\ell_{j,k}')^-, \qquad 
\alpha_{j}^+=\sum_{k=1}^{n_j} [\CH((\ell_{j,k}')^+):\CH(\ell_j^+)]\cdot (\ell_{j,k}')^+. 
$$
Here $[\CH((\ell_{j,k}')^-):\CH(\ell_j^-)]$ and $[\CH((\ell_{j,k}')^+):\CH(\ell_j^+)]$ are the degrees between the valuation fields via $Y'^\an\to Y^\an$.
In the terminology introduced right before Theorem \ref{distance}, we  have 
$$\alpha_{j}^-=[x_{\ell_j^-}]_{\Gamma'}, \qquad
\alpha_{j}^+=[x_{\ell_j^+}]_{\Gamma'}.$$
In particular, we simply have 
$$
\alpha_{1}^-=[x_v]_{\Gamma'}, \quad 
\alpha_{m}^+=[x_w]_{\Gamma'}, \quad
\alpha_{j}^+=\alpha_{j+1}^-, \quad j=1,\dots, m-1. 
$$

The preimage $\pi_{\sk}''^{-1}(\ell_j)$, which is the union of $\ell_{j,k}'$ with $k=1,\dots, n_j$, induces a transport plan $M_j$ from $\alpha_{j}^-$ to $\alpha_{j}^+$. 
In fact, via the path $\ell_{j,k}'$, the transport plan carries mass 
$$d_{\ell_{j,k}'}(\pi')=l(\ell_{j})/l(\ell_{j,k}')$$ 
from $(\ell_{j,k}')^-$
to $(\ell_{j,k}')^+$. 
In other words, $M_j$ is an $n_j\times n_j$ diagonal matrix with diagonal entries $d_{\ell_{j,k}'}(\pi')$. Here repeated endpoints are temporarily kept as separate atoms; combining equal atoms gives a transport plan in the earlier convention that the support points are distinct.

To see that this indeed gives a transport plan $M_j$ from $\alpha_{j}^-$ to $\alpha_{j}^+$, we need to check  
$$
\sum_{k=1}^{n_j} d_{\ell_{j,k}'}(\pi') \cdot (\ell_{j,k}')^-
=\sum_{k=1}^{n_j} [\CH((\ell_{j,k}')^-):\CH(\ell_j^-)] \cdot (\ell_{j,k}')^-
$$
and a similar equality for the right endpoints.
This is equivalent to, for a fixed point $u=(\ell_{j,k}')^-$ with a fixed $j$, 
$$
\sum_{ (\ell_{j,k}')^-=u} d_{\ell_{j,k}'}(\pi') 
= [\CH(u):\CH(\ell_j^-)],
$$
where the left-hand side is
a summation through all $k$ such that  
$\ell_{j,k}'$ contains $u$ as a left endpoint. 
This is a special case of Lemma \ref{total expansion}.

Therefore, we have a canonical transport plan $M_j$ from $\alpha_{j}^-$ to $\alpha_{j}^+$. 
By composition, we have a canonical  
transport plan $M=M_1\circ \cdots \circ M_{m}$ from $[x_v]_{\Gamma'}=\alpha_{1}^-$ to $[x_w]_{\Gamma'}=\alpha_{m}^+$.
By push-forward via $\psi_{i, \sk}':\Gamma'\to \Gamma_i$, we have a transport plan
$M_j$ from $\bar\alpha_{j}^-=\psi_{i, \sk}'(\alpha_{j}^-)$ to $\bar\alpha_{j}^+=\psi_{i, \sk}'(\alpha_{j}^+)$, and a transport plan  
$M=M_1\circ \cdots \circ M_{m}$ from $[x_v]_{\Gamma_i}$ to $[x_w]_{\Gamma_i}$.
The transport plan on $\Gamma_i$ corresponds to a union of paths
$\bar \ell_{j,k}=\psi_{i, \sk}'(\ell_{j,k}')$ carrying mass $d_{\ell_{j,k}'}(\pi')$ for $j=1,\dots, m$ and $k=1,\dots, n_j$. 

With the transport plan, we can finally estimate the distances. 
By the triangle inequality in Lemma \ref{triangle inequality}, we have 
$$\begin{aligned}
\dist([x_v]_{\Gamma_i}, [x_w]_{\Gamma_i}, M)
\leq & \sum_{j=1}^{m} \dist(\bar\alpha_{j}^-, \bar\alpha_{j}^+, M_j)
\\
= & \sum_{j=1}^{m}  \sum_{k=1}^{n_j}
d_{\ell_{j,k}'}(\pi')\cdot \dist(\bar \ell_{j,k}^-, \bar \ell_{j,k}^+)
\\
\leq& \sum_{j=1}^{m} \sum_{k=1}^{n_j} d_{\ell_{j,k}'}(\pi')\cdot  l(\bar \ell_{j,k}).\end{aligned}
$$
We need to bound the right-hand side. 
As in the case of $\deg(x)=1$, this is given by two other relations via the Cauchy--Schwarz inequality. 

The first relation is   
 $$ \frac{1}{\deg(x)}
\sum_{j=1}^{m} \sum_{k=1}^{n_j}   d_{\ell_{j,k}'}(\pi')\cdot l(\ell_j)
 = \dist(v,w).$$
Here we have used the total degree 
$$
\sum_{k=1}^{n_j}   d_{\ell_{j,k}'}(\pi') 
 = \deg(\pi') = \deg(x)
$$
from Lemma \ref{total degree}.

The second relation is
 $$
\frac{1}{\deg(x)}
\sum_{j=1}^{m} \sum_{k=1}^{n_j}   d_{\ell_{j,k}'}(\psi_i')^2 l(\ell_{j,k}')
\leq \epsilon^{-1}A_3(h_{\TL}(x)+A_0).
$$
In fact, as in the case of $\deg(x)=1$, the height inequality gives
 $$
 \int_{\Gamma'}\psi_i'^*\omega_0
 \geq \sum_{j=1}^{m} \sum_{k=1}^{n_j} \int_{\ell_{j,k}'}\psi_i'^*\omega_0
= \sum_{j=1}^{m} \sum_{k=1}^{n_j}   d_{\ell_{j,k}'}(\psi_i')^2 l(\ell_{j,k}').$$

With these two relations, the Cauchy--Schwarz inequality gives 
$$
\begin{aligned}
\frac{1}{\deg(x)^2}
\dist([x_v]_{\Gamma_i},[x_w]_{\Gamma_i}, M)^2
\leq&\ \frac{1}{\deg(x)^2}
 \left(\sum_{j,k} d_{\ell_{j,k}'}(\pi')\cdot  l(\bar \ell_{j,k}) \right)^2\\
\leq &\ \frac{1}{\deg(x)^2}
 \left(\sum_{j,k}  d_{\ell_{j,k}'}(\pi')\cdot l(\ell_j) \right)
\left( \sum_{j,k} d_{\ell_{j,k}'}(\psi_i')^2\cdot l(\ell_{j,k}')\right) \\
\leq &\ \dist(v,w)\cdot \epsilon^{-1}A_3(h_{\TL}(x)+A_0).
\end{aligned}
$$
Here we have used the simple equality
$$
\left( d_{\ell_{j,k}'}(\pi')\cdot  l(\bar \ell_{j,k}) \right)^2
= \left(  d_{\ell_{j,k}'}(\pi')\cdot l(\ell_j) \right)
\left( d_{\ell_{j,k}'}(\psi_i')^2\cdot l(\ell_{j,k}')\right)
$$
obtained from 
$$
d_{\ell_{j,k}'}(\pi')= \frac{l(\ell_j)}{l(\ell_{j,k}')},\quad
d_{\ell_{j,k}'}(\psi_i')=\frac{l(\bar\ell_{j,k})}{l(\ell_{j,k}')}.
$$
This proves Theorem \ref{distance} for general $\deg(x)$. 
Then the proof of Theorem \ref{equi ft} is complete.

\section{Counterexamples for fiberwise equidistribution}\label{counterexample about fiberwise equi}

The goal of this section is to construct the counterexample to Conjecture \ref{equi conj} stated in Theorem \ref{thm counterexample super intro}.

\begin{theorem}[Theorem \ref{thm counterexample super intro}] 
 \label{thm counterexample super}
Denote $k = \mathbb{Z}$, $F = \mathbb{Q}(t)$, and $X = \mathbb{P}^1_F$.
Let $f:X\to X$ be the square map.
Let $L$ be the tautological line bundle $\mathcal{O}(1)$ on $X$, and let $\overline{L}$ be the $f$-invariant adelic line bundle on $X/\ZZ$ extending $L$. 
Then there exists a generic sequence $\{x_m\}_m$ of $X(\overline{F})$, numerically small with respect to $\overline{L}$, such that for any valuation $v$ of $F$ satisfying the condition that $F_v$ is a finite extension of $\QQ_w$ for the restriction $w=v|_\QQ$,  
 the Galois orbit of $\{x_m\}_m$ is not equidistributed in $X_v^{\mathrm{an}}$ for any probability measure on $X_v^{\mathrm{an}}$. 
\end{theorem}

Throughout this section, we further take $Y=\PP^1_\QQ$ and $\CX=\PP^1_{Y}$, and view the natural map $\pi:\CX\to Y$ as a projective model of $X\to \Spec F$ over $\QQ$. 
The square map $f:X\to X$ extends to the relative square map $f:\CX\to \CX$ over $Y$ by base change. The adelic line bundle $\OL$ on $X/\ZZ$ extends to an adelic line bundle on $\CX/\ZZ$ with $f^*\OL\simeq 2\OL$, which we still denote by $\OL$.

\subsection{Degeneracy on algebraic fibers}

In the above, we have taken $v$ as a valuation of $F$, which can be viewed {as a point} of $Y_w^\an$ by the injection $\CM(F/\QQ_w)\to Y_w^\an$. 
Here $\CX_y^\an$ is the fiber of $\CX^\an_w\to Y^\an_w$ over $y$.

In this subsection, as a short {digression}, we consider a rather special and independent situation, which concerns equidistribution of the Galois orbit of $x_m$ on $\CX_y^\an$ for a closed point $y$ of $Y_w=Y_{\QQ_w}$ for a place $w$ of $\QQ$. 
It turns out that in this case, the reduction $\triangle(x_m)_y$ can fail to be generic on the fiber $\CX_y$, where we denote by $\triangle(x_m)$ the multi-section of $\CX\to Y$ corresponding to $x_m$. 
In fact, we have the following example due to Junyi Xie.

\begin{example}[Xie]
Fix a closed point $y_0\in Y$. Take a section $x_0\in \CX(Y)$ which is not $0$ or $\infty$ on the generic fiber $X$, but is $0$ on the special fiber $\CX_{y_0}$. Then $f^{-n}(x_0)|_\eta$ is a generic small sequence in $X(\ol F)$ with respect to $\OL$. However, the specialization of $f^{-n}(x_0)$ at $y_0$ is constantly $0$, thus it is not generic.
As a consequence, the sequence has no equidistribution on the fiber above $y_0$. 
\end{example}

\subsection{Construct the counterexample}\label{con_arc}

In this subsection, we prove Theorem \ref{thm counterexample super} by {giving} an explicit construction. 
We only need to consider $v$ such that $w=v|_\QQ$ is the standard absolute value.  We will prove the following theorem, which is slightly stronger by taking into account the points $v$ outside the image of $\CM(F/\QQ_w)\to Y_w^\an$.
We still resume the setting of Theorem \ref{thm counterexample super} and the paragraph right after it. 

\begin{theorem} \label{thm counterexample supersuper}
There exists a generic sequence $\{x_m\}_m$ of $X(\overline{F})$, numerically small with respect to $\overline{L}$, such that for every place $w$ of $\QQ$ and every point $y\in Y^\an_w$ which is given by a point in $Y(\ol\QQ_w)$, the Galois orbit of $\{x_m\}_m$ is not equidistributed in $\CX_y^{\mathrm{an}}$ with respect to any measure.
\end{theorem}

We will first prove the following two special cases of the above theorem, and then merge the special cases to prove the above theorem.

\begin{theorem}[archimedean case] \label{thm counterexample for arch fibers}
Let $w=\infty$ be the archimedean place of $\QQ$. Then there exists a generic sequence $\{x_m\}_m$ in $X(\overline{F})$, numerically small with respect to $\OL$, such that for every point $y\in Y(\CC)$, the Galois orbit of $\{x_m\}_m$ is not equidistributed in $\CX_y^{\an}$ with respect to any measure.
\end{theorem}

\begin{theorem}[non-archimedean case] \label{thm counterexample for nonarch fibers}
Let $w$ be a finite place of $\QQ$. Then there exists a generic sequence $\{x_m\}_m$ in $X(\overline{F})$, numerically small with respect to $\OL$, such that for every point $y$ in the image of $Y(\ol\QQ_w)\to Y^\an_w$, the Galois orbit of $\{x_m\}_m$ is not equidistributed in $\CX_y^{\an}$ with respect to any measure.
\end{theorem}

\subsubsection{Global estimates}

To prepare for the proof of Theorem \ref{thm counterexample supersuper}, we first make some global estimates. 

For a polynomial $p(t)\in\BQ[t]$, let $x_p$ be the point in $X(F)$ with homogeneous coordinate $[x_0:x_1]=[1:p(t)]$. Our idea is to choose a suitable sequence of polynomials $\{p_m\}_m$ and a corresponding sequence of integers $\{n_m\}_m$, and then define $x_m$ to be the $\overline{F}$-point
$$
x_m = \left[1: (p_m)^{1/n_m}\right]
$$
in order to construct the desired sequence $\{x_m\}_m$.

The algebraic dynamical system $(\CX, f, L)$ over $Y$ is the base change of the algebraic dynamical system $(\PP^1_\QQ, f_0, L_0)$ over $\QQ$, where $f_0$ is the square map on $\PP^1_\QQ$ and $L_0=\CO(1)$. 
Let $\psi:\CX=\PP^1_\QQ\times_\QQ Y\to\PP^1_\QQ$ be the first projection.  Then $\OL=\psi^*\OL_0$, where $\OL_0$ is the $f_0$-invariant extension of $L_0$ on $\PP^1_\QQ/\ZZ$.
Denote by $M_\QQ$ the set of places of $\QQ$, where each place is endowed with the usual absolute value. 
It is known that $\OL_0$ is linearly equivalent to the adelic divisor 
$$
\OD_0= (D_0, (\log^+|z|_w)_{w\in M_\QQ}). 
$$
Here $z=z_1/z_0$ is the affine coordinate of $\PP^1_\QQ$, and $D_0$ is the divisor associated to the point $\infty=[0:1]$ of $\PP^1_\QQ$.
It follows that $\OL$ is linearly equivalent to the adelic divisor 
$$
\OD= (D, (\log^+|z|_w)_{w\in M_\QQ}).
$$
Here we still take $z=z_1/z_0$ as the affine coordinate of $\PP^1_Y$, and $D=\psi^*D_0$ is the divisor associated to infinity section of $\PP^1_Y$.

Note that $x_p$ is a closed point of degree 1 on $X$. Let $\triangle_p$ be the Zariski closure of $x_p$ in $\CX$, which is also viewed as a section of $\CX\to Y$. 
The vector-valued height $\triangle_p^* \OL$, introduced by Yuan--Zhang in \cite[\S5.3.1]{YZ26}, is an adelic line bundle on $Y/\ZZ$. In the current situation,  
$\triangle_p^* \OL$
 is linearly equivalent to 
$$
\triangle_p^* \OD= (\deg(p) [0:1], \ (\log^+|p(y)|_w)_{w\in M_\QQ})
$$
on $Y=\PP^1_\QQ$.
Here we take  $y=y_1/y_0$ as the affine coordinate of $Y=\PP^1_\QQ$, and $[0:1]$ is the infinity of $Y$.
This can be obtained by viewing $\triangle_p$ as the graph of the morphism $Y\to \PP_\QQ^1$ given by $y\longmapsto p(y)$. 

Denote by $\triangle_t$ for $\triangle_p$ with $p(t)=t$.
We have 
$$
\triangle_p^* \OD
-\deg(p)\triangle_t^* \OD= (0, \ g_p)= (0, \ (g_{p,w})_{w\in M_\QQ}),
$$
where the Green function 
$$
g_{p,w}(y)=\log^+\left|p(y)\right|_w-\deg(p)\log^+\left|y\right|_w
$$
is actually continuous on the Berkovich space $Y_w^\an$. 
Moreover, $g_{p,w}=0$ for all but finitely many $w\in M_\QQ$. 

We further denote the supremum norm
\[
\left|g_{p,w}\right|_{\sup}=\sup_{y\in Y_w^\an}\left|g_{p,w}(y)\right|
=\sup_{y\in Y(\ol\QQ_w)}\left|g_{p,w}(y)\right|.
\]

\begin{lemma}\label{lem bound height of x}
For any $\OH\in\HPic(F/\ZZ)_{\QQ,\nef}$, there exists a constant $C_\OH$ such that for all nonzero $p(t)\in\BQ[t]$,
\[
h^\OH_\OL(x_p)\leq \deg(p)h^\OH_\OL(x_t)+C_\OH \sum_{w\in M_\BQ}\left|g_{p,w}\right|_{\sup}.
\]
\end{lemma}

\begin{proof}
Let $U$ be a non-empty open subset of $Y$ such that $\OH\in\HPic(U/\ZZ)$.
We have
\[
h^\OH_\OL(x_p)-\deg(p)h^\OH_\OL(x_t)
=\OH\cdot\left(0,g_p\right)
=\sum_{w\in M_\BQ}\int_{U_w^{\an}}g_{p,w}c_1(\OH)
\leq (\deg\widetilde H)\cdot \sum_{w\in M_\BQ}\left|g_{p,w}\right|_{\sup}.
\]
\end{proof}

For positive integers $r,s$, take
\begin{equation*}
P_{r,s}:=\left\{\sum\limits_{i=1}^r a_it^{2^i-1} \Big|a_i\in\BZ, \left|a_i\right|\leq s\right\}\subset \QQ[t].
\end{equation*}

\begin{lem}\label{poly_bd}
	For any nonzero $p\in P_{r,s}$, we have
	$$\sum_{w\in M_\BQ}\left|g_{p,w}\right|_{\sup}\leq \log8+3\log s+2\log r.$$
\end{lem}
\begin{proof}
	We can assume that $\deg(p)=2^r-1$; otherwise, the problem is reduced to smaller $r$.  
	Let $a=a_r$ be the absolute value of the coefficient of the leading term of $p$.
We assume that $y\in\ol\QQ_w$ in the following.

We first assume that $w$ is archimedean. 
Denote $R=(4r^2s^2a^{-2})^\frac{1}{2^r-1}>1$. 
	If $|y|\geq R$, then 
	$$|p(y)|\geq |a| R^{2^r-1}-(r-1)s R^{2^{r-1}-1} 
	=R^{2^{r-1}-1}( |a| R^{2^{r-1}}-(r-1)s)>1.  $$
Then
$$
	g_{p,w}(y)
	= \log\left|p(y)\right|-\deg(p)\log\left|y\right| 
	=\log\left|y^{-\deg p}p(y)\right|.
	$$
	Since
	$$
	\left|y^{-\deg p}p(y)-a\right|
	\leq\sum_{i=1}^{r-1}\left|a_iy^{2^i-2^r}\right|\leq rs\cdot(4r^2s^2a^{-2})^{-\frac{1}{2}}\leq\frac{a}{2},
	$$
	we have 
	$$\left|g_{p,w}(y)\right|\leq \log 2a\leq\log 2+\log s.$$
	If $\left|y\right|<R$,
	$$
	\deg(p)\log^+\left|y\right|\leq \log4+2\log s+2\log r-2\log a.
	$$
	Since
	$$
	\left|p(y)\right|
	\leq\sum_{i=1}^{r-1}\left|a_i(4r^2s^2a^{-2})^{\frac{2^i-1}{2^r-1}}\right|+4r^2s^2a^{-1}
	\leq rs\cdot(2rsa^{-1})+4r^2s^2a^{-1}\leq 8r^2s^2a^{-1},
	$$
	we get
	$$
	\log^+\left|p(y)\right|\leq \log8+2\log s+2\log r-\log a.
	$$
	Thus
	\[
	\left|g_{p,w}(y)\right|\leq\max\{\deg(p)\log^+\left|y\right|,\log^+\left|p(y)\right|\}\leq \log8+2\log s+2\log r-\log a.
	\]
	Note that the above inequality also holds for the first case. Thus it holds for archimedean $w$ and all $y\in\ol\QQ_w$.
	
	Now we assume that $w$ is non-archimedean. 
	If $\left|y\right|_w>\left|a^{-2}\right|_w^\frac{1}{\deg p}$, we have
	$$
	\log^+\left|p(y)\right|_w=\log\left|a\cdot y^{\deg(p)}\right|_w.
	$$
	Thus, $\left|g_{p,w}(y)\right|\leq-\log\left|a\right|_w$.
	If $\left|y\right|_w\leq \left|a^{-2}\right|_w^\frac{1}{\deg p}$, we have $\deg(p)\log^+\left|y\right|_w\leq -2\log\left|a\right|_w$. Moreover,
	\[
	\left|p(y)\right|_w
	\leq \max\left\{\max_{1\leq i\leq r-1}\left|a_iy^{2^i-1}\right|_w,\ |a^{-1}|_w\right\}
	\leq \max\left\{\left|y^{\frac{\deg p}{2}}\right|_w,\ |a^{-1}|_w\right\}
	\leq \left|a^{-1}\right|_w.
	\]
	Thus,
	\[
	\left|g_{p,w}(y)\right|\leq\max\{\deg(p)\log^+\left|y\right|,\log^+\left|p(y)\right|\}\leq-2\log\left|a\right|_w.
	\]
	
In conclusion,
	\begin{align*}
	\sum\limits_{w\in M_\BQ}\left|g_{p,w}\right|_{\sup}&\leq \log8+2\log s+2\log r-\log a+\sum\limits_{w\in M_{\BQ,f}}-2\log\left|a\right|_w\\
	&=\log8+2\log s+2\log r+\log a\\
	&\leq \log8+3\log s+2\log r.
	\end{align*}
\end{proof}

\subsubsection{Archimedean case}

Here we construct a sequence for Theorem \ref{thm counterexample for arch fibers}. 
We first have the following easy Minkowski type of result. 

\begin{lem}\label{sm_poly}
	For each $y\in \BC$, $r\geq 3$, $s\geq 2$, there is a nonzero polynomial $p_{y,r,s}$ in $P_{r,s}$ such that $\left|p_{y,r,s}(y)\right|<4r(1+\left|y\right|^{2^r})s^{1-\frac r2}$.
\end{lem}

\begin{proof}
	Consider the set
	\begin{equation*}
	P_{r,s}^+:=\Big\{\sum\limits_{i=1}^r a_it^{2^i-1} \Big|a_i\in\BZ, 0\leq a_i \leq s\Big\}
	\end{equation*}
	There are $(s+1)^r$ elements in $P_{r,s}^+$. For each $p\in P_{r,s}^+$, $\left|p(y)\right|\leq rs(1+\left|y\right|^{2^r})$.  Consider the open disks
	\begin{equation*}
	\Big\{D\left(p(y), 2rs(1+\left|y\right|^{2^r})(s+1)^{-\frac r2}\right)\Big|p\in P^+_{r,s}\Big\}.
	\end{equation*}
	Since $(s+1)^{-\frac r2}<\frac12$, the closure of all these disks are contained in the open disk $D_0:=D(0,2rs(1+\left|y\right|^{2^r}))$. On the other hand, the sum of the area of these disks is $4\pi r^2s^2(1+\left|y\right|^{2^r})^2$, which is the area of $D_0$. Thus, these disks cannot be disjoint. This implies that there exist $p_1\neq p_2$ in $P^+_{r,s}$ such that
	\[
	\left|p_1(y)-p_2(y)\right|<4rs(1+\left|y\right|^{2^r})(s+1)^{-\frac r2}<4r(1+\left|y\right|^{2^r})s^{1-\frac r2}.
	\]
	Let $p_{y,r,s}=p_1-p_2\in P_{r,s}$, which is nonzero. {This completes} the proof.
\end{proof}

\begin{prop}\label{prop arch sequence}
Let $N$ be a positive integer. Denote by $\wt Q_N$ the set of pairs of $(p,n)$,  where $p$ is a non-constant polynomial in $\QQ[t]$, and $n$ is a positive integer satisfying
	\begin{enumerate}[(1)]
		\item $n\geq N\deg(p)$,
		\item $n\geq N\sum_{w\in M_\BQ}\left|g_{p,w}\right|_{\sup}$.
	\end{enumerate}
Then there is a finite subset $Q_N$ of $\wt Q_N$ such that for each $y\in\overline{D(0,N)}$, there is a pair $(p,n)$ in $Q_N$ with
	$$
	\left|p(y)\right|<2^{-Nn}.
	$$
\end{prop}
\begin{proof}
	Take $r=8N^2+2$. For each $s\geq 2$ and each $y\in\CC$, let $p_{y,r,s}$ be the polynomial constructed in Lemma \ref{sm_poly}. Set 
	$$n_{r,s}=\lceil N(\log8+3\log s+2\log r)\rceil.$$
	 First note that for $s$ sufficiently large (independent of the choice of $y$), the pair $(p_{y,{r},s},n_{r,s})$ satisfies conditions $(1)$ and $(2)$. In fact, $(1)$ follows from $\deg(p)\leq {2^r-1}$, and $(2)$ follows from Lemma \ref{poly_bd}.
	
	For any $y\in\CC$,
	$$
	\limsup\limits_{s\to\infty} \left|p_{y,r,s}(y)\right|^\frac{1}{n_{r,s}}
	\leq \lim\limits_{s\to\infty} \left(4r(1+\left|y\right|^{2^r})s^{-4N^2}\right)^\frac{1}{n_{r,s}}
	=\lim\limits_{s\to\infty}s^{\frac{-4N^2}{3N\log s}}
	=\mathrm{e}^{-\frac{4N}{3}}
	<2^{-N}.
	$$
	Thus, for $s$ sufficiently large, $\left|p_{y,r,s}(y)\right|<2^{-Nn_{r,s}}$. Then for each point $y\in\overline{D(0,N)}$, we have a pair $(p_{y,{r},s},n_{r,s})$ satisfying the condition that $\left|p_{y,{r},s}(y)\right|<2^{-Nn_{r,s}}$. However, if the property $\left|p(y)\right|<2^{-Nn}$ holds for $y$, then it holds in an open neighborhood of $y$. Thus the proposition holds by the compactness of $\overline{D(0,N)}$.
\end{proof}

Now we are ready to construct the sequence 
in Theorem \ref{thm counterexample for arch fibers}.

\begin{proof}[Proof of Theorem \ref{thm counterexample for arch fibers}]
For each positive integer $N$, for each $(p,n)\in Q_N$, let $x$ be the $\ol F$-point $[1:p^{1/n}]$ in $X(\ol F)$. Then by Lemma \ref{lem bound height of x}, for $\OH\in\HPic(F/\BZ)_{\QQ,\nef}$,
$$
h^\OH_\OL(x)=\frac{1}{n}h^\OH_\OL(x_p)
\leq \frac1n\left(\deg(p)h^\OH_\OL(x_t)+C_\OH\sum\limits_{v\in M_\BQ}\sup\left|g_{p,v}\right|\right)
\leq \frac{1}{N}\left(h^\OH_\OL(x_t)+C_\OH\right).$$
	
Now consider the sequence $\{[1: (p_m)^{1/n_m}]\}_m$ formed by listing all elements in $Q_1$, then those in $Q_2$, and so on. Define $x_{2m-1}$ to be the $\overline{F}$-point $[1: (p_m)^{1/n_m}]$, and define $x_{2m}$ to be the $\overline{F}$-point $[1:2^\frac1m]$ in $X(\overline{F})$. This sequence is $h^\OH_\OL$-small by the computations above. Moreover, the sequence is {obviously generic}.
	
Fix a point $y\in Y(\BC)$. It is easy to see that the Galois orbits of $\{x_{2m}\}_m$ are equidistributed in $\CX_y^\an$ with respect to the Haar measure on the unit circle. 

On the other hand, when $y=\infty$, the Galois orbits of $\{x_{2m-1}\}_m$ are equidistributed in $\CX_y^\an$ with respect to the Dirac measure of the infinity, since their reductions on $\CX^\an_\infty$ are all infinity. 

When $y\neq \infty$, by the construction of the set $Q_N$, there is a subsequence 
$\{x_{2m_k-1}\}_k$ of $\{x_{2m-1}\}_m$ such that
\[
\lim_{k\to\infty}|x_{2m_k-1,y}|=0,
\]
where $x_{2m_k-1,y}$ is the reduction of $x_{2m_k-1}$ on the fiber $\CX_y^\an$. Thus the Galois orbits of the subsequence are equidistributed in $\CX_y^\an$ with respect to the Dirac measure of the origin. 

In summary, Galois orbits of $\{x_m\}_m$ are not equidistributed in $\CX_y^\an$ with respect to any measure for any $y$.
\end{proof}

\subsubsection{Non-archimedean case}

Now we prove Theorem \ref{thm counterexample for nonarch fibers} by applying a similar method for Theorem \ref{thm counterexample for arch fibers}. The key point is to prove analogues of Lemma \ref{sm_poly} and Proposition \ref{prop arch sequence}.

In the following, let $w$ be a finite place of $\QQ$ corresponding to a prime number $q$.

\begin{lem}\label{sm_poly-nonarch1}
For each $y\in \ol\QQ_w$ of degree $d$ over $\QQ_w$, $r\geq 3$, $s\geq 2$, there is a nonzero polynomial $p_{y,r,s}$ in $P_{r,s}$ such that $\left|p_{y,r,s}(y)\right|_w\leq q^2\left(1+|y|_w^{2^r}\right)s^{-\frac{r}{d}}$.
\end{lem}
\begin{proof}
Denote $K_w=\QQ_w(y)$.
Let $\varpi$ be the uniformizer of $O_{K_w}$.
Let $e=\log_{|\varpi|_w}|{q}|_w$ and $f=[O_{K_w}/{\varpi} O_{K_w}:\FF_q]$.
Denote
\begin{equation*}
P_{r,s}^+:=\Big\{\sum\limits_{i=1}^r a_it^{2^i-1} \Big|a_i\in\BZ, 0\leq a_i \leq s\Big\}.
\end{equation*}
Note that for $p\in P_{r,s}^+$,
\[
\left|p(y)\right|_w\leq\max_{1\leq i\leq r}|y|_w^{2^i-1}\leq 1+|y|_w^{2^r}\leq |\varpi|_w^c,
\]
where $c=\lfloor\log_{|\varpi|_w}(1+|y|_w^{2^r})\rfloor$. Let $c'=\lfloor\log_{q^f}((s+1)^r-1)\rfloor$. Then for all $p\in P_{r,s}^+$, $p(y)\in D(0,|\varpi|_w^c)$. 
There are $(s+1)^r\geq q^{fc'}+1$ polynomials in $P_{r,s}^+$, whereas the ball $D(0,|\varpi|_w^c)$ has only $q^{fc'}$ residue classes modulo $\varpi^{c+c'}$. Then there must be $p_1\neq p_2$ in $P^+_{r,s}$ such that
$$\left|p_1(y)-p_2(y)\right|_w
\leq|\varpi|_w^{c+c'}
\leq \left(1+|y|_w^{2^r}\right)q^{\frac{2}{e}} \left(q^{-\frac{1}{e}}\right)^{\frac{\log((s+1)^r-1)}{f\log q}} 
\leq \left(1+|y|_w^{2^r}\right)q^2  s^{-\frac{r}{d}}.$$
Set $p_{y,r,s}=p_1-p_2\in P_{r,s}$. {This completes} the proof.
\end{proof}

Since there are only countably many finite {extensions} of $\QQ_w$ in $\ol\QQ_w$, we enumerate them as
\[
K_1,K_2,\cdots,K_n,\cdots.
\]

\begin{prop}\label{prop nonarch sequence}
Let $N$ be a positive integer. Denote by $\wt Q_N$ the set of pairs of $(p,n)$,  where $p$ is a non-constant polynomial in $\QQ[t]$, and $n$ is a positive integer satisfying
	\begin{enumerate}[(1)]
		\item $n\geq N\deg(p)$,
		\item $n\geq N\sum_{w\in M_\BQ}\left|g_{p,w}\right|_{\sup}$.
	\end{enumerate}
Then there is a finite subset $Q_N$ of $\wt Q_N$ such that for each $y\in K_1\cdots K_N$ with $|y|_w\leq N$, there is a pair $(p,n)$ in $Q_N$ with
	$$
	\left|p(y)\right|_w<2^{-Nn}.
	$$
\end{prop}
\begin{proof}
Set $D_N=[K_1\cdots K_N:\QQ_w]$ and take $r=3D_NN^2$. For each $s\geq 2$ and each $y\in\ol\QQ_w$, let $p_{y,r,s}$ be the polynomial constructed in {the above lemma}. Set
	$$n_{r,s}=\lceil N(\log8+3\log s+2\log r)\rceil.$$
For $s$ sufficiently large (independent of the choice of $y$), the pair 
$(p_{y,{r},s},n_{r,s})$ satisfies conditions $(1)$ and $(2)$. In fact, $(1)$ follows from $\deg(p)\leq {2^r-1}$, and $(2)$ follows from Lemma \ref{poly_bd}.
	
	For any $y\in K_1\cdots K_N$,
$$
	\limsup\limits_{s\to\infty} \left|p_{y,r,s}(y)\right|_w^\frac{1}{n_{r,s}}
	\leq \lim\limits_{s\to\infty} \left(q^2\left(1+|y|_w^{2^r}\right)s^\frac{-r}{\deg y}\right)^\frac{1}{n_{r,s}}
	\leq\lim\limits_{s\to\infty}s^{\frac{-3D_NN^2}{\deg(y)N(\log8+3\log s+2\log r)}}
	\leq\mathrm{e}^{-N}.
$$
	Thus, for $s$ sufficiently large, $\left|p_{y,r,s}(y)\right|_w<2^{-Nn_{r,s}}$. Then for each point
	\[
	y\in\overline{D(0,N)}\subset  K_1\cdots K_N,
	\]
	we have a pair $(p_{y,{r},s},n_{r,s})$ satisfying the condition $\left|p_{y,{r},s}(y)\right|_w<2^{-Nn_{r,s}}$. Thus the proof is finished by compactness as in the proof of Proposition \ref{prop arch sequence}.
\end{proof}

Now we can also prove Theorem \ref{thm counterexample for nonarch fibers}. With Proposition \ref{prop nonarch sequence}, the proof proceeds exactly as in the case of Theorem \ref{thm counterexample for arch fibers}. We therefore give only a brief outline.

\begin{proof}[Proof of Theorem \ref{thm counterexample for nonarch fibers}]

Consider the sequence $\{[1: (p_m)^{1/n_m}]\}_m$ formed by listing all elements in $Q_1$, then those in $Q_2$, and so on. For $x=[1:p^\frac{1}{n}]$ such that $(p,n)\in Q_N$, we have
\[
h^\OH_\OL(x)\leq \frac{1}{N}\left(h^\OH_\OL(x_t)+C_\OH\right).
\]
Define $x_{2m-1}$ to be the $\overline{F}$-point $[1: (p_m)^{1/n_m}]$, and define $x_{2m}$ to be the $\overline{F}$-point $[1:2^\frac1m]$ in $X(\overline{F})$. Then this sequence is generic and $h_\OL^\OH$-small.

Fix a classical point $y$ in $Y_w^\an$  corresponding to a point in $Y(\ol\QQ_w)$. Then the Galois orbits of $\{x_{2m}\}_m$ are equidistributed in $\CX_y^\an$ with respect to the Dirac measure of the Gauss point. 

However, when $y=\infty$, the Galois orbits of $\{x_{2m-1}\}_m$ are equidistributed in $\CX_y^\an$ with respect to the Dirac measure of the infinity. 

When $y\neq \infty$, by the construction of the set $Q_N$, there is a subsequence $\{x_{2m_k-1}\}_k$ of $\{x_{2m-1}\}_m$ such that
\[
\lim_{k\to\infty}|x_{2m_k-1,y}|_w=0,
\]
where $x_{2m_k-1,y}$ is the reduction of $x_{2m_k-1}$ on the fiber $\CX_y^\an$. Thus the Galois orbits of the subsequence are equidistributed in $\CX_y^\an$ with respect to the Dirac measure of the origin. 

In summary, Galois orbits of $\{x_m\}_m$ are not equidistributed in $\CX_y^\an$ with respect to any measure for any classical point $y$.
\end{proof}

\subsubsection{Combination}

Now we are ready to prove the strongest counterexample in Theorem \ref{thm counterexample supersuper}.

\begin{proof}[Proof of Theorem \ref{thm counterexample supersuper}]
We will construct the required sequence $\{x_m\}_m$ by combining the above constructions.
First, for $m\geq 1$, take $x_{2m}=[1:2^\frac{1}{m}]\in X(\ol F)$ as before.
We construct $\{x_{2m-1}\}_m$ as follows.

Denote by $\{w_1, w_2, \dots\}$ the set of all places of $\QQ$ (including $\infty$).
For each place $w_i$ of $\QQ$, we have constructed a sequence $\{x_{m,i}\}_m$.
From this sequence, extract the terms with odd subscripts and denote them by $y_{m,i} := x_{2m-1,i}$.

By the constructions in Theorem \ref{thm counterexample for arch fibers} and Theorem \ref{thm counterexample for nonarch fibers}, for each $i\geq1$ and $j\geq1$, there exists an $a_{j,i}$ such that for each $\OH\in\HPic(F/\ZZ)_{\QQ,\nef}$, we have
\[
h_\OL^\OH(y_{m,i})\leq\frac{1}{j}\left(h_\OL^\OH(x_t)+C_\OH\right),\ \forall m>a_{j,i}.
\]
In fact, $a_{j,i}$ is determined by the size of each $Q_N$ for the place ${w_i}$. Moreover, we may take $a_{j,i}$ such that
\[
1<a_{1,i}<a_{2,i}<a_{3,i}<\cdots,\forall i\geq1.
\]
Set $a_{0,i}=0$.
For $j,i\geq1$, let
\[
Y_{i,j}:=\{y_{a_{j-1,i}+1,i},y_{a_{j-1,i}+2,i},\cdots,y_{a_{j,i},i}\}
\]
be a finite subsequence of $\{y_{m,i}\}_m$.
Define $\{x_{2m-1}\}_m$ by adding these finite sequences in the order
\[
Y_{1,1}, Y_{1,2}, Y_{2,2}, Y_{1,3}, Y_{2,3}, Y_{3,3}, Y_{1,4},\cdots.
\]
All terms appearing above {are} of the form $Y_{i,j}$ with $j\geq i$.
Then $\{x_{2m-1}\}_m$ is a generic and $h_\OL^\OH$-small sequence for all $\OH\in\HPic(F/\ZZ)_{\QQ,\nef}$. In addition, for every $w_i$ and every $y \in Y^\an_{w_i}$ satisfying the condition in Theorem \ref{thm counterexample supersuper}, there exists a subsequence of $\{x_{2m-1}\}_m$ consisting of
$$
Y_{i,i},\; Y_{i,i+1},\; Y_{i,i+2},\; \dots
$$
whose Galois orbits are not equidistributed in $\CX_y^\an$ with respect to the Haar measure of the unit circle when $w_i$ is archimedean, or with respect to the Dirac measure at the Gauss point when $w_i$ is non-archimedean.
In contrast, the Galois orbits of the sequence $\{x_{2m}\}_m$ are equidistributed in $\CX_y^\an$ with respect to the same measure.
It follows that the Galois orbit of $\{x_m\}_m$ is not equidistributed in $\CX_y^\an$ with respect to any measure, for any place $w$ and any $y$.
\end{proof}

\

{\footnotesize
\noindent Ruoyi Guo

\noindent Address: \emph{School of Mathematical Sciences, Peking University, Haidian District, Beijing 100871, China}

\noindent  Email: \emph{guoruoyi@pku.edu.cn}

\

\noindent Lai Shang

\noindent Address: \emph{School of Mathematical Sciences, Peking University, Haidian District, Beijing 100871, China}

\noindent Email: \emph{laishang@stu.pku.edu.cn}

\

\noindent Chengyuan Yang

\noindent Address: \emph{BICMR, Peking University, Haidian District, Beijing 100871, China}

\noindent  Email: \emph{chengyuanyang020416@pku.edu.cn}

\

\noindent Xinyi Yuan

\noindent Address: \emph{BICMR, Peking University, Haidian District, Beijing 100871, China}

\noindent Email: \emph{yxy@bicmr.pku.edu.cn}
}

 \end{document}